\documentclass[12pt]{article}
\usepackage[english]{babel}
\usepackage{amsfonts, dsfont, amsmath,amssymb,amsthm,mathrsfs, amsbsy}
\usepackage{fullpage}
\usepackage{graphicx,graphics}
\usepackage[dvipsnames]{xcolor}
\usepackage{epsfig}
\usepackage{latexsym}
\usepackage[applemac,utf8]{inputenc}
\usepackage[colorlinks=true]{hyperref}
\usepackage{enumitem}
\usepackage[toc,page]{appendix}
\usepackage{bbm}
\usepackage{stmaryrd}
\usepackage{comment}
\usepackage{pgf,tikz}
\usepackage{enumitem}
\usepackage{mathtools}
\usepackage{ulem} 
\usepackage{textcomp}
\usepackage[cal=euler]{mathalfa}
\usepackage[mono=false]{libertine}
\useosf
\usepackage[letterpaper,top=2cm,bottom=2cm,left=2cm,right=2cm,marginparwidth=1.5cm]{geometry}
\usepackage[font={small,sf}, labelfont={sf,bf}, margin=1cm]{caption}
\usepackage{csquotes}

\usepackage[%
backend=biber,
hyperref=true,
style=ext-alphabetic,
url=true,
isbn=false,
giveninits=true,
maxnames=3,
backref=false,
date=year,
abbreviate=true,
eprint=true,
articlein=false,
maintitleaftertitle=true,
language=english
]{biblatex}
\usepackage[auto]{contour}
\usetikzlibrary{shapes, angles, decorations.text, patterns, fadings,decorations.pathreplacing,fit}

\newcommand{\PP}{\mathbb{P}}
\newcommand{\E}{\mathbb{E}}
\newcommand{\D}{\mathbb{D}}
\newcommand{\N}{\mathbb{N}}

\newcommand{\R}{\mathbb{R}}

\renewcommand{\d}{\mathrm{d}}

\newcommand{\U}{\mathbb{U}}
\newcommand{\HH}{\mathbb{H}}

\renewcommand{\geq}{\geqslant}
\renewcommand{\leq}{\leqslant}

\renewcommand{\u}{\mathtt{u}}

\makeatletter
	\newcommand{\leqnomode}{\tagsleft@true\let\veqno\@@leqno}
	\newcommand{\reqnomode}{\tagsleft@false\let\veqno\@@eqno}
\makeatother

\newcommand{\Expect}[2][]{%
	\E_{#1}\mathopen{}\mathclose\bgroup\left[\,#2\,\aftergroup\egroup\right]}
\newcommand{\condExpect}[3][]{%
	\E_{#1}\mathopen{}\mathclose\bgroup\left[\,#2\;\middle\vert\;#3\,\aftergroup\egroup\right]}
\renewcommand{\P}{\PP}
\newcommand{\Prob}[2][]{\P_{#1}\mathopen{}\mathclose\bgroup\left(\,#2\,\aftergroup\egroup\right)}
\newcommand{\condProb}[3][]{%
	\P_{#1}\mathopen{}\mathclose\bgroup\left(\,#2\;\middle\vert\;#3\,\aftergroup\egroup\right)}
\newcommand{\anyProb}[2]{#1\mathopen{}\mathclose\bgroup\left(\,#2\,\aftergroup\egroup\right)}

\newcommand{\magic}{\mathsf{m}}

\newtheorem{theorem}{Theorem}
\newtheorem*{theorem*}{Theorem}

\newtheorem{proposition}{Proposition}[section]
\newtheorem{lemma}{Lemma}

\theoremstyle{definition}
\newtheorem{definition}{Definition}[section]

\theoremstyle{remark}
\newtheorem*{remark}{Remark}

\theoremstyle{definition}
\newtheorem{example}{Example}

\date{}

\title{\bf Growing Self-Similar Markov Trees}

\author{Nicolas Curien\thanks{Institut de Math\'ematique d'Orsay, Universit\'e Paris-Saclay, \url{nicolas.curien@gmail.com}}, \quad William Fleurat \thanks{Institut de Math\'ematique d'Orsay, Universit\'e Paris-Saclay, \href{mailto:william.fleurat@universite-paris-saclay.fr}{\texttt{william.fleurat@universite-paris-saclay.fr}}}\quad \&\hspace{0.1cm} Adrianus Twigt\thanks{Institut de Math\'ematique d'Orsay, Universit\'e Paris-Saclay, \href{mailto:adrianus.twigt@universite-paris-saclay.fr}{\texttt{adrianus.twigt@universite-paris-saclay.fr}}}}

\begin{document}
\maketitle
\begin{abstract}\begin{center} \textit{Can we obtain a Brownian CRT of mass $1/2$ from a CRT of mass $1$ by cutting certain branches?}\end{center} In this paper, we will answer that  question in the much more general setting of self-similar Markov trees. Self-similar Markov trees (ssMt) are random decorated trees that encode the genealogy of a system of particles carrying positive labels, and where particles undergo splitting and growth depending on their labels in a self-similar fashion. Introduced and developed in the recent monograph \cite{bertoin2024self}, they provide a broad generalization of Brownian and stable continuum random trees and arise naturally in various models of random geometry such as the Brownian sphere/disk. The law of a ssMt is characterized by its quadruplet $(\mathrm{a}, \sigma^2, \boldsymbol{\Lambda}; \alpha)$, which specifies the features of the underlying growth--fragmentation mechanism, together with the initial decoration $x>0$. 

In this work, we focus on special cases of ssMt in which the trees started from different initial values $x>0$ can be coupled into a continuous, increasing family of nested subtrees. In the case of the Brownian and stable continuum random trees, this yields  surprisingly simple novel dynamics corresponding to the scaling limit of the leaf-growth algorithms of {\L}uczak--Winkler and Caraceni--Stauffer.
\end{abstract}

\begin{figure}[h!]
    \centering
    \hspace{-1cm}\includegraphics[width=0.3\linewidth]{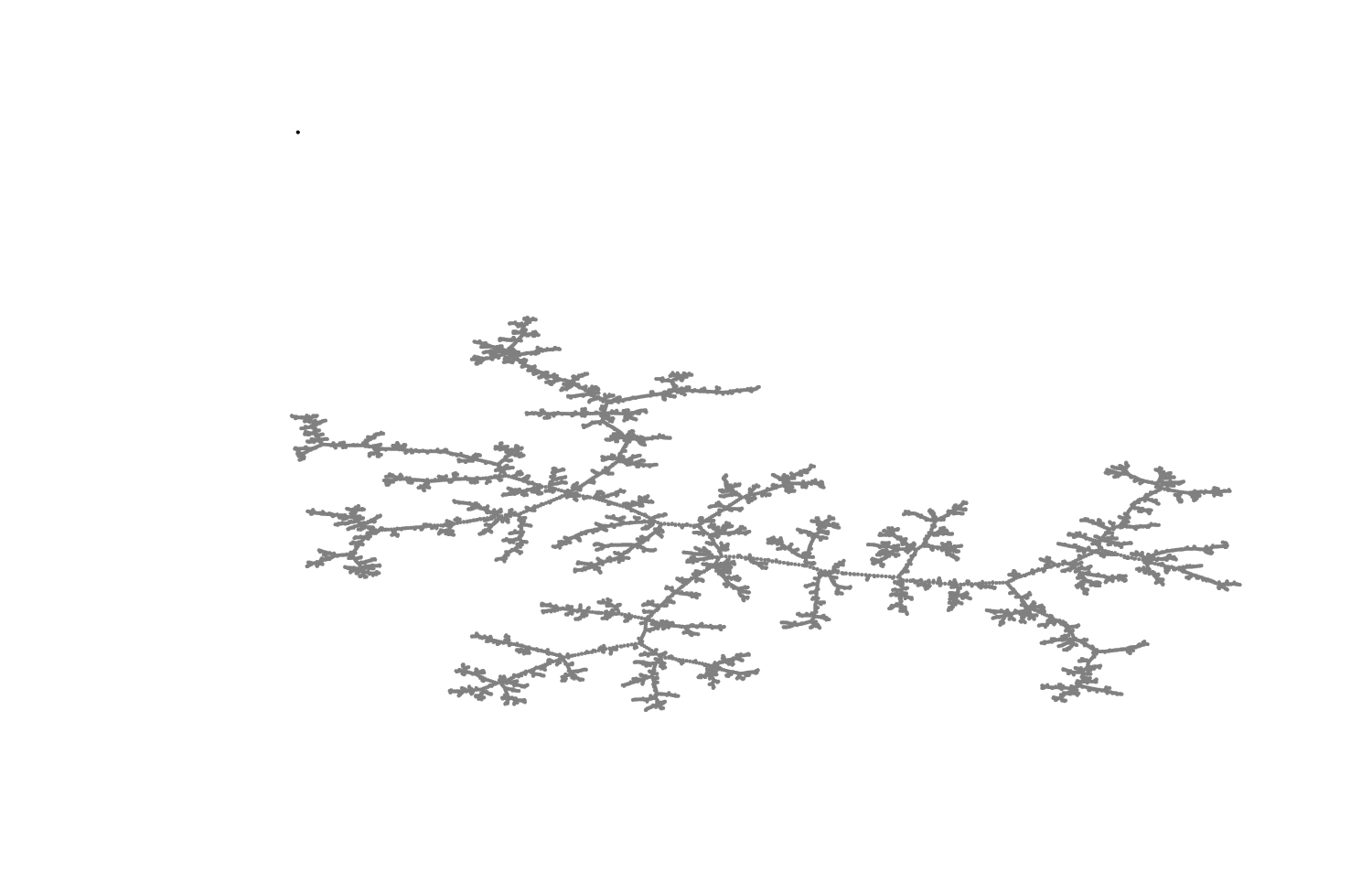} \hspace{-1cm}\includegraphics[width=0.3\linewidth]{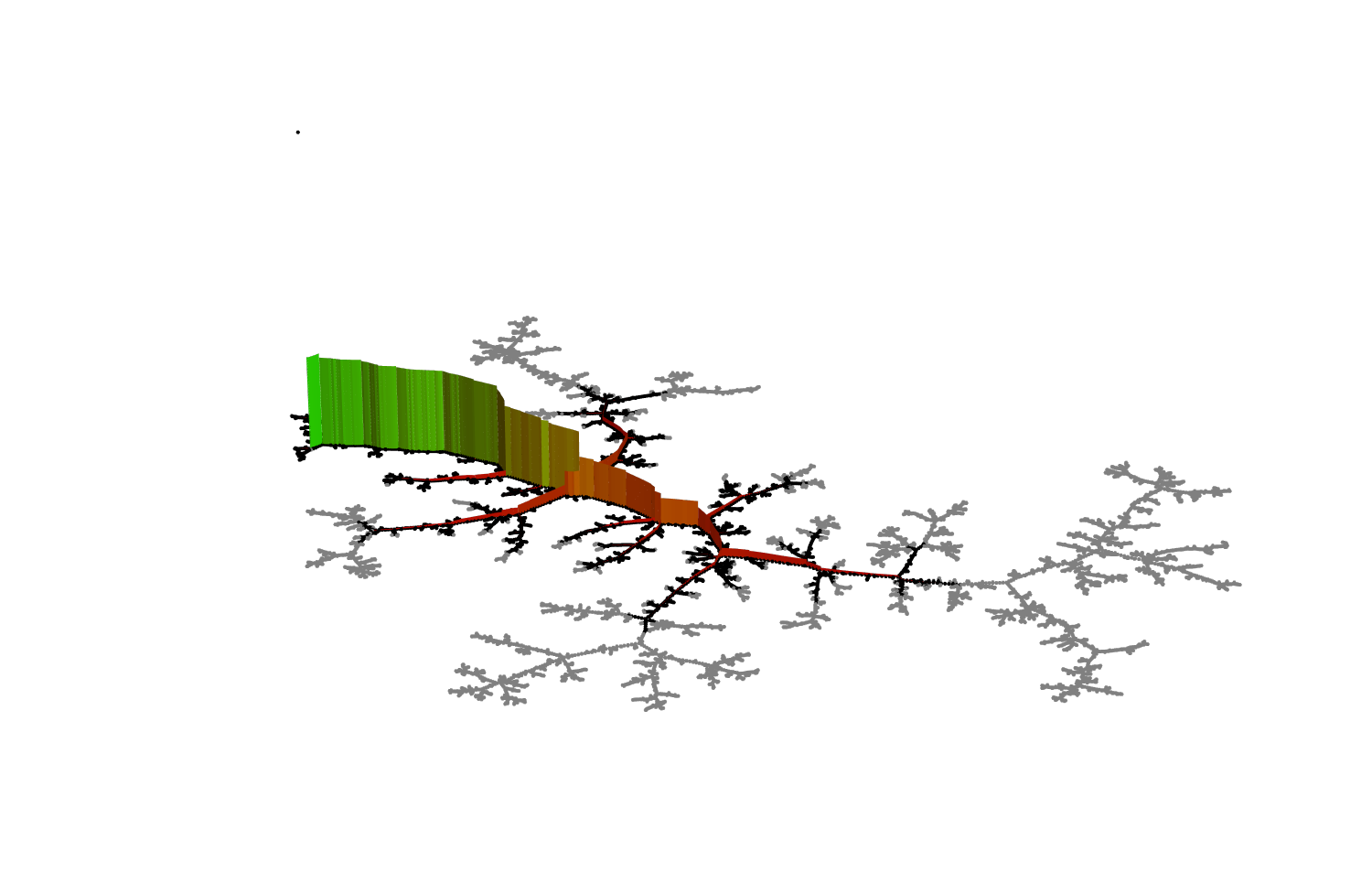}\hspace{-1cm}\includegraphics[width=0.3\linewidth]{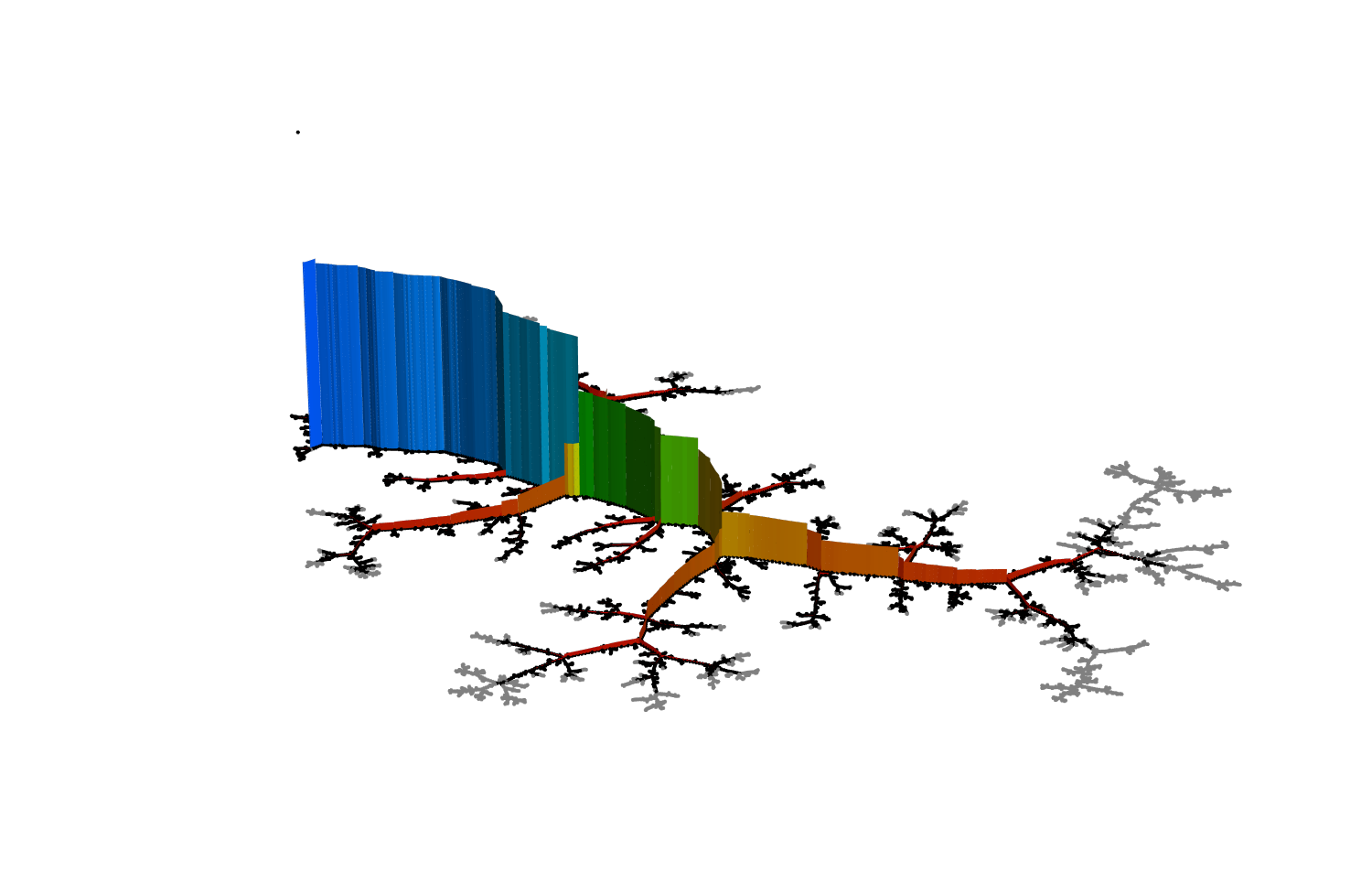}\hspace{-1cm}\includegraphics[width=0.3\linewidth]{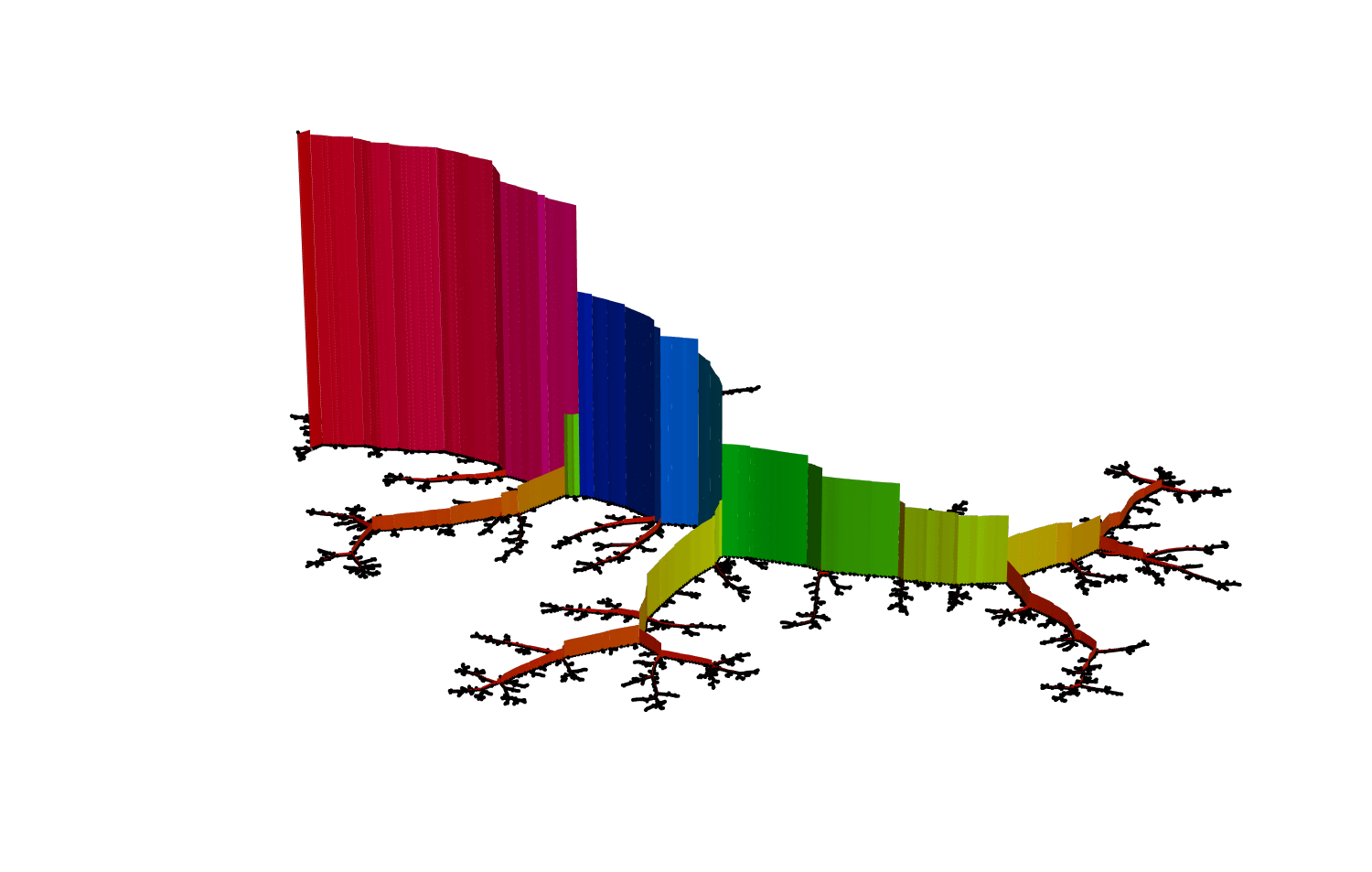}
    \caption{Illustration of the increasing family of ssMt in the case of the Brownian CRT: the underlying Brownian CRT of mass $1$ is displayed as the base tree, whereas its subtrees of mass $0.3 , 0.6$ and $1$ are depicted as hypographs.}
    \label{fig:intro}
\end{figure}

\clearpage

\section{Introduction}\label{sec:intro}

\paragraph{Self-Similar Markov trees.} In a nutshell, self-similar Markov trees (ssMt), are the branching extension of the famous positive self-similar Markov processes (pssMp) of John Lamperti. As for pssMp, it is expected that they are characterized by their Markov and scaling property and that they form all the possible scaling limits of multi-type Galton--Watson trees with types in $ \mathbb{Z}_+$. They encompass and unify various models such as Aldous' Brownian Continuum Random tree \cite{aldous1991continuum} and its stable generalizations \cite{le2002random} as well as the Brownian Cactus \cite{curien2013brownian} sitting inside the Brownian sphere \cite{le2020growth,bertoin2018random}, or the related scaling limits of peeling or parking trees \cite{bertoin2018martingales,contat2025universality}.  Formally, a self-similar Markov tree is actually a family of laws $( \mathbb{Q}_x : x >0)$ of decorated random trees $\mathtt{T} = (T, d_T, \rho, g)$ where $\rho$ is the root of the tree and $g : T \to \mathbb{R}_+$ is a decoration which is upper semi-continuous (usc) on $T$ and positive on its skeleton. They can be defined through their characteristic quadruplets $( \mathrm{a}, \sigma^2, \boldsymbol{\Lambda} ; \alpha)$, which through a Lamperti transformation encapsulate the law of their underlying branching L\'evy processes. In particular, the so-called generalized L\'evy measure $\boldsymbol{\Lambda}$ is a (possible infinite) measure on the space $ \mathcal{S} = \{ \mathbf{u}=(u_0,u_1, ...) : u_0  \in \mathbb{R} \mbox{ and }  u_1\geq u_2 \geq ... \in \mathbb{R} \cup \{ - \infty\}\}$ which describes the splitting probabilities. In this work we shall mainly work with the push-forward $\boldsymbol\Xi$ of $\boldsymbol{\Lambda}$ by the exponential map, which is now a (generally infinite) measure on 
$$ \mathcal{E} = \exp( \mathcal{S})= \{ \mathbf{y}=(y_0,y_1, ...) : y_0 >0  \mbox{ and }  y_1\geq y_2 \geq ... \geq 0\},$$ and is called the \textbf{splitting measure}. Informally, an individual of decoration $x >0$ sees its decoration instantaneously moved to $y_0 \cdot x$ while giving rise to a  family of particles with decorations $x \cdot y_1, x \cdot y_2, ... $ (which are interpreted as the birth of new individuals) 
 \begin{eqnarray} x \to (x \cdot y_0, (x \cdot y_1, x \cdot y_2, \cdots )) \quad \mbox{ at a rate } \quad x^{-\alpha} \cdot \boldsymbol{\Xi}( \mathrm{d} \mathbf{y}),  \label{eq:intensitysplit} \end{eqnarray}  where $\alpha >0$ is the self-similarity parameter. The  coefficient $ \mathrm{a} \in \mathbb{R}$ encodes the drift term while  $\sigma^2$ controls the Brownian part of the evolution of the decoration along branches. See Section \ref{sec:background} or \cite{bertoin2024self} for details. When the starting decoration $x>0$ is fixed, the law $ \mathbb{Q}_x$ is the distribution of the ssMt with initial decoration $x$. The self-similarity property of ssMt entails that the law of the rescaled tree $(T, x^\alpha\cdot d_T, \rho, x \cdot g)$ under $\mathbb{Q}_1$ is precisely $\mathbb{Q}_x$, but one can wonder whether a tree of law $\mathbb{Q}_x$ can be obtained by ``cutting branches" from a tree of law $\mathbb{Q}_1$ when $0<x<1$. This natural question is already non-trivial for the Brownian Continuum Random tree (CRT) \cite{aldous1991continuum} where the decoration represents the mass of the fringe subtrees: \begin{center} \textit{Can we obtain a Brownian CRT of mass $1/2$ from a CRT of mass $1$ by cutting certain branches?}\end{center}
 In this paper, we will answer the above question positively and construct in particular continuous and increasing couplings of Brownian and stable CRT and of many more self-similar Markov trees.
 \clearpage 
 \subsection{The growing theorem}
 
\paragraph{Growing condition.} The main idea behind our construction is to require a certain property of $( \boldsymbol{\Xi};\alpha)$ which amounts heuristically to the following: For any $0<x'<x$, one can use the splitting of a particle of decoration $x>0$ to emulate the splitting of a particle of smaller decoration $x'>0$ which we view as ``sitting inside" the larger particle. This is done formally by requiring the existence of a family of functions $G_x(\cdot) : \mathcal{E}\to \mathcal{E}$ for $x >0$ such that 
\begin{align}
    \label{eq:quasipreservation}   (G_x)_\sharp\ \boldsymbol{ \Xi}=  x^{-\alpha}\cdot \boldsymbol{ \Xi}.  \end{align}
Recalling \eqref{eq:intensitysplit}, when such a function exists, one can indeed use the splittings of a particle of decoration $x$, after mapping by $G_{x'/x}(\cdot)$, to emulate the splitting of a particle of decoration $x'$. In order to keep a geometric meaning to this, we shall furthermore require that when $0<x'<x$, the offspring particles of $x'$ naturally sit inside those of $x$, that is we have coordinate-wise
 \begin{eqnarray}
     \label{eq:monotonecoupling}
  x \cdot G_x(\mathbf{y})  \leq \mathbf{y}, \quad \mbox{ for }x \in (0,1).
  \end{eqnarray} Those conditions are depicted schematically in Figure \ref{fig:emulation}.\begin{figure}[h!]
    \centering
    \includegraphics[width=0.5\linewidth]{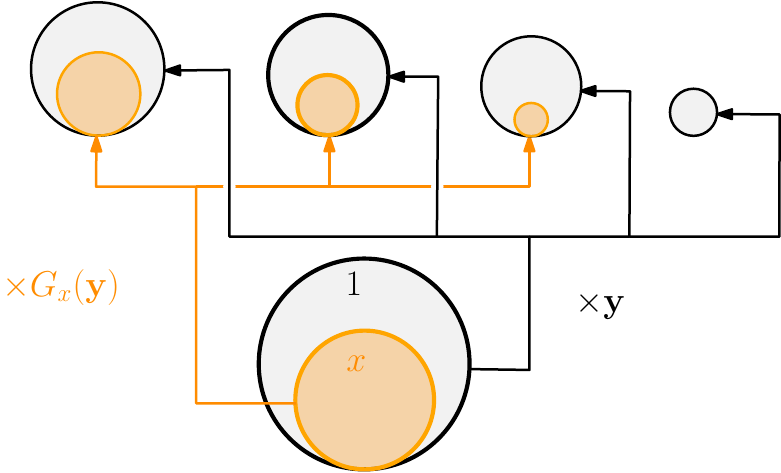}
    \caption{Illustration of the existence of the growing functions $ G$ preserving the splitting intensity and obeying a monotonicity condition. The individual correspond to the particles with heavier lines. \label{fig:emulation}}
\end{figure}

 To be naturally compatible with the above interpretation, we shall also require that the functions $(G_x)_{x>0}$ enjoy the semi-group property $G_x \circ G_{x'} = G_{x\cdot x'}$. Under additional reasonable smoothness assumptions we then say that the ssMt with quadruplet $( \mathrm{a}, \sigma^2, \boldsymbol{\Xi} ; \alpha)$ is \textbf{$G$-growing}, see Definition \ref{def:growing} for details. In particular, the growing assumption depends only on $( \boldsymbol\Xi ; \alpha)$, and we shall see in Lemma \ref{lem:set-of-alpha-is-interval} that  for reasonable splitting measure  $\boldsymbol\Xi$ there exists $\alpha_c( \boldsymbol\Xi) \geq 0$ such that $(\boldsymbol{\Xi} ; \alpha)$ is growing iff $$ 0 < \alpha  \leq \alpha_c( \boldsymbol\Xi).$$ 
Except in the special case where the support of $\boldsymbol\Xi$ is one-dimensional (see Proposition \ref{prop:solution-conservative-binary-case}), we should not generally expect uniqueness of the family $G$ for which $(\boldsymbol\Xi ; \alpha)$ is growing, see Example \ref{ex:haas_stephenson_mass_cex}.  Our first result is informally the following:
\begin{theorem}[Growing self-similar Markov trees, informal] \label{thm:main} If the characteristic quadruplet $( \mathrm{a}, \sigma^2, \boldsymbol{\Xi} ; \alpha)$ of a ssMt is growing (in the sense of Definition \ref{def:growing}), then we can on the same probability space construct a family $ \mathtt{T}_x = (T_x, d_{T_x}, \rho_x, g_x)$ for $x > 0$ of decorated random trees so that 
\begin{enumerate}[label=(\roman*)]
    \item \textbf{Law}. For any $x > 0$ the decorated tree $ \mathtt{T}_x$ has law $ \mathbb{Q}_x$.
    \item \textbf{Monotonicity.} The application $x \mapsto \mathtt{T}_x$ is non-decreasing for the inclusion.
    \item \textbf{Continuity}. The application $x \mapsto \mathtt{T}_x$ is continuous.
    \item \textbf{Markov forward.} The decorated tree $\mathtt{T}_x$ comes with a sequence of points $p_x(i) \in T_x$ for $i \geq 1$. For $x'<x$, there are weights $w_{x' \to x}(i)>0$ such that $T_x$ is obtained from $T_{x'}$ by gluing independent ssMt of initial decoration $w_{x' \to x}(i)$ onto $p_{x'}(i)$.
    \end{enumerate}
\end{theorem}

Let us make some comments about the result (see Theorem \ref{thm:mainprecise} for the formal statement). \begin{itemize}
\item  We say above that  $ \mathtt{T}' \subset \mathtt{T}$ if the  decorated tree $\mathtt{T}'$ can be realized inside $\mathtt{T}$, formally it means that there exists $g : T \to \mathbb{R}$  usc such that $g' \leq g$ and where $ \mathtt{T}'$ is isomorphic to the hypograph of $g'$ on the closure of $\{ x \in T : g'(x) >0\}$ --the support of $g'$--. 
\item The continuity of the application $x \mapsto \mathtt{T}_x$ holds with respect to the hypograph metric introduced in \cite{bertoin2024self} which is stronger that Gromov--Hausdorff metric, see Section \ref{sec:background} for details. 
\item The sequence of points $(p_x(i) : i \geq 0)$ is part of the construction and is obtained as the ends of the branches in the gluing construction of $\mathtt{T}_x$ from the characteristics $(a, \sigma^2, \boldsymbol\Xi ; \alpha)$, see Section \ref{sec:background} for background on the construction of ssMt. In particular, those points depend in a crucial way on the characteristics $(a, \sigma^2, \boldsymbol\Xi ; \alpha)$ and not only on the (law of) the ssMt. Indeed, recall from \cite[Chapter 5]{bertoin2024self} that several characteristic  quadruplets may yield to the same law of ssMt. 
\end{itemize}

Two different quadruplets $( \mathrm{a},\sigma^2, \boldsymbol\Xi ; \alpha)$ and $( \mathrm{a}',\sigma'^2, \boldsymbol\Xi' ; \alpha)$ yielding the same laws $(\mathbb{Q}_x)_{x>0}$ are called \textbf{bifurcators} of one another. In this case, the splitting measures $ \boldsymbol{\Xi}$ and $\boldsymbol\Xi'$ are simply obtained by choosing another particle to follow in some splitting event, and they coincide up to an event of finite measure. The remaining parameters $\mathrm{a},\mathrm{a}',\sigma,\sigma'$ are simple functions of one-another, see \cite[Theorem 5.13]{bertoin2024self} for details. Perhaps surprisingly, different bifurcators  will yield to different growing mechanisms for the same ssMt! The worried reader may first concentrate on the so-called \textbf{locally largest bifurcator} which consists in following the largest particle at each splitting event. In this case,  the splitting measure $\boldsymbol{\Xi}^{\ell\ell}$ is supported on $\mathcal{E}^{\downarrow} :=\{ y_0 \geq y_1 \geq y_2 \geq \cdots \geq 0\}$. See \cite[Remark 1.20]{bertoin2024self}. In the case of the locally largest bifurcator, our growing theorem is intrinsic, i.e. the growing mechanism only depends on the geometry of the decorated tree (whereas in general it depends on the way to explore it using bifurcators).

\begin{figure}
    \centering
        \includegraphics[width=0.6\linewidth]{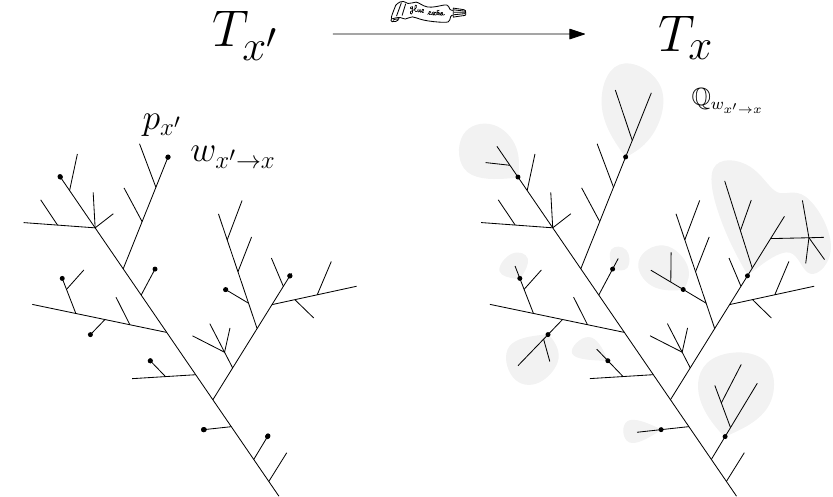}
        \caption{Illustration of Markov property of the construction. We get $T_x$ from $T_{x'}$ by gluing on $p_{x'}(i)$ independent ssMt with starting decorations $w_{x' \to x}(i)$. The points $p_{x'}(i)$ and the weights $w_{x' \to x}(i)$ depend on $\mathtt{T}_{x'}$ and on its characteristics, i.e.~on the bifurcator chosen to construct it.}
        \label{fig:glue_only}
\end{figure}
\medskip 

 Beyond its inherent naturalness, the above theorem is particularly useful for establishing monotonicity results and conditional limit theorems, see the techniques developed in \cite{fleurat2025tauberian}. For example, it shows that the mass, the height or the the width at a given height (or any other function which is increasing for the $\subset$-order) is stochastically increasing in the starting decoration $x>0$. 
The growing condition might seem awkward to the reader, but we shall see in Section \ref{sec:examples} that many natural (but not all!) examples of ssMt are actually growing. 

\paragraph{Idea of the proof.} As recalled above, a ssMt can be seen as a branching analog of a positive self-similar Markov process, and the main technical idea in the proof of Theorem \ref{thm:main} is to build a continuous flow of pssMp with characteristics\footnote{here $\Xi_0$ is the push-forward of the splitting measure $\boldsymbol\Xi$ on the first coordinate} $( \mathrm{a}, \sigma^2, \Xi_0; \alpha)$ as solutions to a stochastic differential equation (SDE) with jumps starting from different values $x >0$. Contrary to the SDE that we can find in the literature and which involve an indicator function, our SDE are based on the family of locally Lipschitz functions $G_\cdot$, so that existence and uniqueness of those solutions follow from standard techniques. The continuity of the flow of solutions, however, is highly non-trivial and is related to the subtle problem of absorption at $0$ of positive self-similar Markov processes. 

\subsection{The Brownian CRT}
 As a first application of our main result Theorem \ref{thm:main}, we prove that the Brownian CRT is growing. Introduced by Aldous \cite{aldous1991continuum}, the Brownian Continuum Random Tree (CRT) is the scaling limit of large critical Galton--Watson trees in the finite variance case. As demonstrated by Le Gall, it is also the random tree coded by a Brownian excursion of length $1$, see \cite{le2005random}, and in this coding, the push-forward of the Lebesgue measure endows the Brownian CRT with a random mass measure supported by its leaves. This measure enables us to decorate the Brownian CRT by the function which assigns to each point the mass of its fringe subtree. 
 
\paragraph{Locally largest bifurcator.} By a famous result of Bertoin \cite{bertoin2000fragmentation}, the resulting decorated tree is then a self-similar Markov tree (with initial decoration $1$) characterized by its quadruplet $( \mathrm{a}=0, \sigma^2=0,  \boldsymbol\Xi^{\ell\ell}_{\text{Bro}}; \alpha = \frac{1}{2})$ where the splitting measure  is 
 \begin{eqnarray} \int  \boldsymbol{\Xi}^{\ell \ell}_{\text{Bro}} (\mathrm{d} \mathbf{y})F( {y_0},{y_1},...) \propto \int_{1/2}^1 \frac{\mathrm{d}s}{(s(1-s))^{3/2}}  F(s,1-s,0,\cdots),  \label{eq:splittingbrownian}\end{eqnarray} for a generic function $F : \mathcal{E} \to \mathbb{R}_+$, see \cite[Chapter 3, Example 3.6]{bertoin2024self}. This concerns the locally-largest bifurcator, since $\boldsymbol\Xi^{\ell \ell}_{\text{Bro}}$ is supported by $\mathcal{E}^{\downarrow}$, which in this binary conservative setting equals $\{ y_0 \geq 1/2 \geq y_1 = 1-y_0, \, y_2 = \ldots = 0 \}$. In this case, an explicit computation shows that the ssMt with characteristics $(0,0, \boldsymbol\Xi^{\ell \ell}_{\text{Bro}} ; 1/2)$ is growing for the family of functions $(G^{\ell \ell}_x : x >0)$ defined by $G^{\ell \ell}_x(s,1-s) = (\mathtt{f}_x(s), \mathtt{f}_x(1-s), 0, \ldots)$ where for $s \geq 1/2$ the function $\mathtt{f}_x$ satisfies $\mathtt{f}_x(1-s) = 1- \mathtt{f}_x(s)$ and 
\begin{equation}\label{eq:sol_ode_bro_mass}
	\mathtt{f}_x(s) = \frac12 \left(1 + \sqrt{\frac{h(s)}{\left(4x^{-1}+h(s)\right)}}\right), \quad \mbox{with $h(s) = \frac{(1 - 2s)^2}{s(1-s)}$}, \quad s \geq 1/2.
\end{equation}
\begin{figure}[h!]
    \centering
    \includegraphics[width=0.8\linewidth]{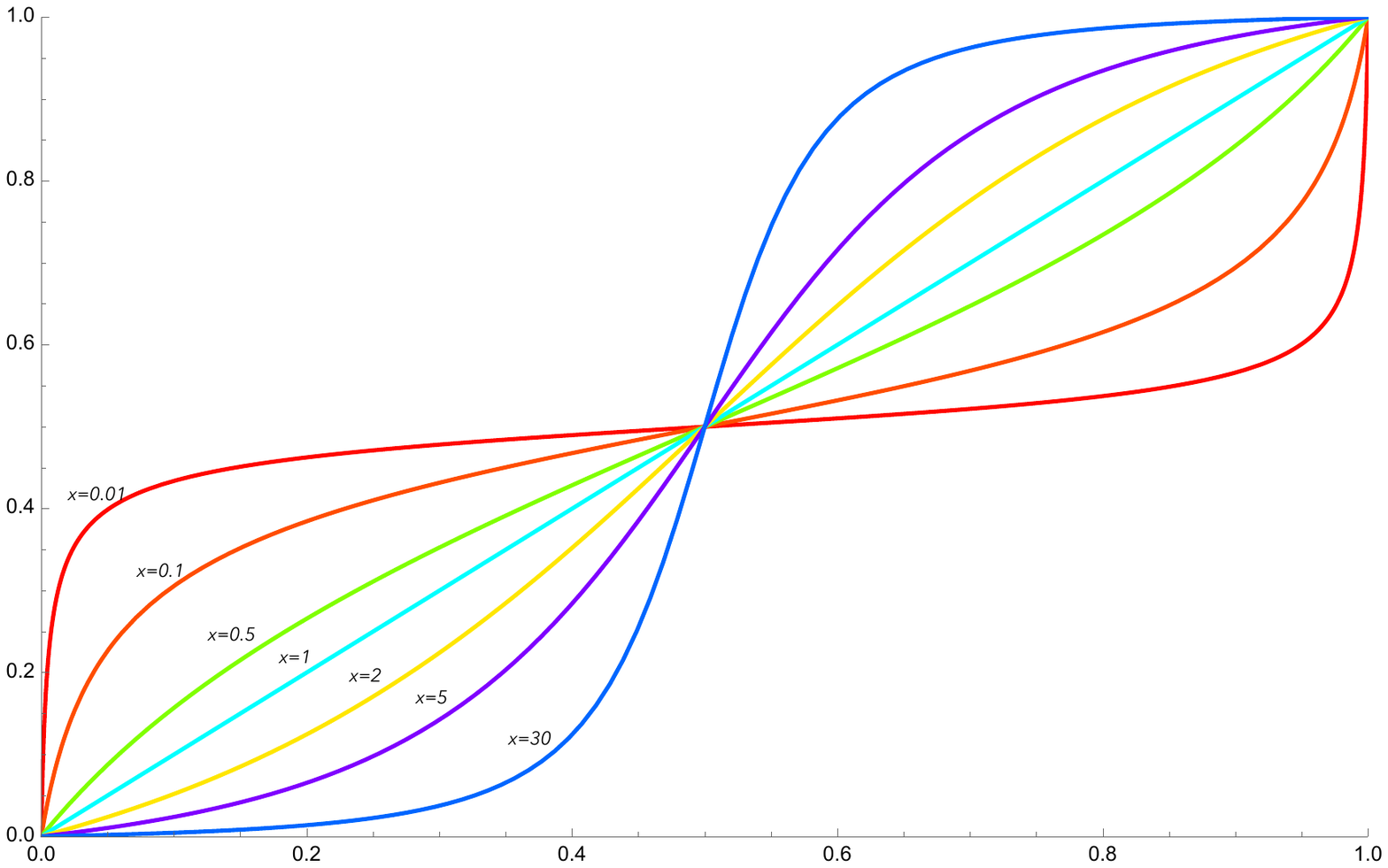}
    \caption{Plot of the function $\mathtt{f}_x$ for different values of $x$. Although explicit, this family of functions is not particularly obvious to guess! Notice in particular that if $x<1$  then $\mathtt{f}_x(u)$ is closer  to $1/2$ compared to $u$.}
    \label{fig:Deltax}
\end{figure}

Furthermore the growing ability of the Brownian CRT is critical since $ \alpha_c(\boldsymbol\Xi^{\ell \ell}_{\text{Bro}}) =1/2$. As a result of Theorem \ref{thm:main}, we can produce a continuous and increasing family $ (T^{\ell\ell}_x: x >0)$\footnote{here $T^{\ell\ell}_x$ is the underlying undecorated real tree of $\mathtt{T}_x^{\ell \ell}$.} of Brownian CRT of mass $x$, this process is illustrated in the opening Figure \ref{fig:intro}.  Several couplings of Brownian CRT with random masses were known before, especially the Aldous--Pitman Poisson cutting process \cite{aldous1998standard}, but to the best of our knowledge the above construction is a brand new canonical dynamic for this  object. Actually, we will show in a companion paper that the family $(T^{\ell\ell}_x: x >0)$ is the scaling limit of the discrete leaf-growth algorithm of Luczak--Winkler \cite{luczak2004building} and Caraceni--Stauffer \cite{caraceni2020polynomial} recently generalized in \cite{Fleurat24}. In its simplest incarnation, this algorithm produces a sequence of random binary trees $(B_n : n \geq 0)$ with $n$ leaves such that $B_n$ is uniformly distributed over the set of plane binary trees with $n$ leaves and where $B_{n+1}$ is obtained from $B_n$ by adding a cherry on a leaf, see Figure \ref{fig:LW}. Contrary to the basic intuition, the leaf to be transformed into a cherry is not uniform on $B_n$ and is rather picked in an intricate way, see \cite{caraceni2024random} for details. Despite its \textit{a priori}  challenging discrete dynamic, we will prove in our forthcoming work that the scaling limit of this process is provided by $(T^{\ell\ell}_x:x > 0)$ in the sense that we have the following convergence for the Gromov--Hausdorff topology
$$ \left(  \frac{B_{[nx]}}{ \sqrt{n} } : x > 0 \right) \xrightarrow[n \to \infty]{(d)} (T^{\ell\ell}_x : x > 0).$$

\begin{figure}[h!]
    \centering
    \includegraphics[width=0.75\linewidth]{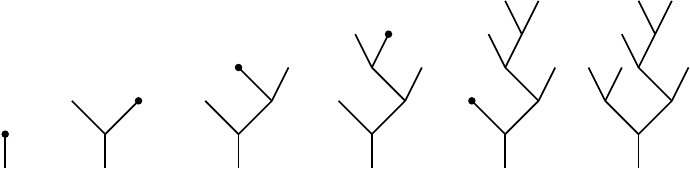}
    \caption{Illustration of the leaf-growth algorithm producing a sequence of increasing uniform binary trees which converge in the scaling limit towards the locally largest  growing coupling of Brownian CRT's.}
    \label{fig:LW}
\end{figure}

\paragraph{Another bifurcator, another growth!} As mentioned above, the Brownian CRT decorated with the mass of fringe subtrees can  be obtained as the ssMt with characteristics $(0,0, { \boldsymbol\Xi}' ; 1/2)$ for many other splitting measures $\boldsymbol{\Xi}'$ obtained as bifurcators of \eqref{eq:splittingbrownian}. For example the so-called \textbf{size-biased bifurcator} is obtained by following a size-biased particle at each branch point:
\begin{eqnarray} \int  \boldsymbol{\Xi}^*_{\text{Bro}} (\mathrm{d} \mathbf{y})F( {y_0},{y_1},...) \propto \int_{0}^1 \frac{\mathrm{d}s}{(s(1-s))^{3/2}} \cdot s \cdot  F(s,1-s,0,\cdots).  \label{eq:splittingbrowniansize-biased}\end{eqnarray}
Notice that $\boldsymbol\Xi^*_{\text{Bro}}$ is now supported by sequences $\{y_0 \geq 0 , y_1\geq y_2 \geq \cdots \geq 0\}$ where the first real number $y_0$ may not be the largest. In that case, an even simpler computation shows that $(\boldsymbol{\Xi}^*_{\text{Bro}} ; 1/2)$ is growing for the simpler function
  \begin{eqnarray}
      \label{def:magicfunction}G^{ \mathsf{M}}_x(y_0, (y_1, y_2, y_3, ...)) := \left(\frac{x y_0}{xy_0 + (1-y_0)}, \frac{1}{xy_0 + (1-y_0)}( y_1, y_2, ...)\right).\end{eqnarray}
      The growing family $(T^*_x: x > 0)$ of Brownian CRT produced by Theorem \ref{thm:main} is then \textbf{different} from $(T^{\ell\ell}_x : x > 0)$ obtained for the locally largest bifurcator. In particular, the points $p^*_x(i)$ through which the tree $T_x^*$ grow are different: in the first case, $p^{\ell\ell}_x(1)$ is the extremity of the branch obtained by starting at the root and following the locally largest tree (for the mass) at each branch point. In the second case,  the point $p^*_x(1)$ is just a typical point of the tree. Although simpler, the second growing dynamic is not geometrically intrinsic (i.e. only depending on the equivalence class of decorated tree $\mathtt{T}_x$).

Down-to-earth, in both cases which fall into the pure-jump with no drift case,  the coupling of Theorem \ref{thm:main} works as follows: Consider $(T_1, \rho)$ a Brownian CRT  of mass 1 with root $\rho$ and decorated by the function $g_1$ which associates to each point $u$ the mass of the subtree above $u$ in $T_1$. We denote by $p \in T_1$ the extremity of the branch followed by the bifurcator (i.e. either the locally largest or the size-biased one) and consider the function $g_1$ along the branch $\llbracket \rho,  p\rrbracket$ as a pure jump decreasing c\`adl\`ag function  over $[0,d(\rho,p)]$. For $x \in (0,1]$, we then  create a function $g_x : [0, d(\rho,p)] \to \mathbb{R}_+$ by putting  for all $t \in [0, d(\rho,p)]$

$$
g_x(0)=x \quad \mbox{ and } \quad 
\left(\frac{g_x(t)}{g_x(t-)}, \frac{g_x(t-)-g_x(t)}{g_x(t-)}\right) :=  G_{g_x(t-)}\left(\frac{g_1(t)}{g_1(t-)}, \frac{g_1(t-)-g_1(t)}{g_1(t-)}\right),$$  where $G$ can be either $G^{\ell\ell}$ or $G^{\mathsf{M}}$. We then repeat the same procedure along the trees that branch off $\llbracket \rho,  p\rrbracket$ to create a  decoration $g_x : T_1 \to \mathbb{R}_+$. The closure of the support $\{u \in T_1 : g_x(u)>0\}$ is then our Brownian CRT of mass $x$ and $g_x$ is its mass-decoration function.

The fact that the preceding bifurcators both have $\alpha_c(\boldsymbol\Xi)=1/2$ is a happy coincidence. In fact, one can exhibit a bifurcator $\boldsymbol\Xi^{\mathrm{weird}}_{\text{Bro}}$ of \eqref{eq:splittingbrownian} having $\alpha_c(\boldsymbol\Xi^{\mathrm{weird}}_{\text{Bro}}) < 1/2$, see Example~\ref{ex:brownian_crt_mass_weird}. In general we show in Proposition \ref{prop:locallylargestthebest} that in the smooth case among all bifurcators of a given ssMt, the locally largest bifurcator $\boldsymbol\Xi^{\ell\ell}$ has the largest $\alpha_c$.

\subsection{Stable and Brownian trees and a magical function}
Perhaps surprisingly, the function $G^{\mathsf{M}}$ in \eqref{def:magicfunction} serves as the basis to show that many different (bifurcators of) ssMt are growing. On top of the canonical Brownian CRT of Aldous, we shall also consider the $\beta$-stable CRT of Duquesne \& Le Gall \& Le Jan which are the scaling limits of critical Galton--Watson trees whose offspring distribution is in the domain of attraction of a $\beta$-stable law for $\beta \in (1,2]$, see \cite{le2002random}. 
\begin{enumerate}
\item As mentioned above, the Brownian CRT of mass $1$, decorated with the function which associates to any point the mass of its subtree is a ssMt. Its size-biased bifurcator $(\boldsymbol\Xi_{\text{Bro}}^*;1/2)$ is then $G^\mathsf{M}$ growing, see Example \ref{ex:brownian_crt_mass_size_biased}.
\item By a result of Miermont \cite{Miermont2003}, the $\beta$-stable CRT of mass $1$ can similarly be seen as a ssMt with characteristics $(0,0, \boldsymbol \Xi^{ \ell \ell}_{\beta\text{-st}}; 1- 1 / \beta)$. The splitting measure $\boldsymbol{\Xi}^{\ell\ell}_{\beta\text{-st}}$ of the locally largest bifurcator is not binary anymore and its expression is quite involved. However, its size-biased bifurcator $(\boldsymbol \Xi^{*}_{\beta\text{-st}}; 1-1/\beta)$ is again $G^\mathsf{M}$-growing for the same function $G^\mathsf{M}$! See Example \ref{ex:stable_crt_mass_size_biased}
\item As seen above, the locally largest bifurcator of the Brownian CRT by mass is also growing but for a more complicated family of function. This could be put in perspective for $\beta$-stable CRT by mass as follows.  Consider the function 
$$ {G}_x^{\ell\ell}(\mathbf{y})=(\mathtt{y}_0(x),\mathtt{y}_1(x),\dots),\qquad \mathbf{y}\in \mathcal{E},$$
where the functions $\mathtt{y}_i\colon [0,\infty)\to [0,1]$, $i\geq0$, implicitly depend on $\mathbf{y}=(y_0,y_1, ...)$ through the system of differential equations:
\begin{align}
    \begin{cases}
        \partial_x \mathtt{y}_i(x)=\frac{1}{x} \cdot \mathtt{y}_i(x)\cdot\left(\mathtt{y}_i(x)-\sum_{j}\mathtt{y}_j(x)^2\right),   & i\geq0,\\
        \mathtt{y}_i(1)= y_i,
    \end{cases} \label{def:gellell}
\end{align}
which admits a unique solution when the initial data $\mathbf{y}$ is such that $\sum_i y_i=1$. {It can then be checked that when restricted to sequences $\mathbf{y}=(y_0,1-y_0, 0, 0, ...)$ with $y_0\geq 1/2$ we recover the function defined by \eqref{eq:sol_ode_bro_mass}.} We then prove in Examples \ref{ex:brownian_crt_mass_locally_largest} and \ref{ex:stable_crt_mass_locally_largest} that $(\boldsymbol\Xi_{\beta\text{-st}}^{\ell\ell}; 1-1/\beta)$ are $G^{\ell\ell}$-growing. 

\item The above examples are generalized to version of the Brownian and $\beta$-stable CRT which are conditioned on having height $1$ (instead of mass $1$). Endowing those random trees with the function which associates to any point the height of its subtree transform them into ssMt with self-similarity index $1$, see \cite[Examples 3.5 and 3.8]{bertoin2024self}. In the Brownian case, the resulting tree is a ssMt with characteristics $(-1,0,\boldsymbol\Xi^{\mathrm{h}}_{\text{Bro}};1)$ where  
\begin{eqnarray*} \int  \boldsymbol{\Xi}^{\rm h}_{\text{Bro}} (\mathrm{d} \mathbf{y})F( {y_0},{y_1},...) \propto \int_{0}^1 \frac{\d h}{h^2}F(1,h,0,0,...). \end{eqnarray*}
Notice that the above splitting measure corresponds to the locally largest bifurcator. It is then easy to check that $(\boldsymbol\Xi^{\mathrm{h}}_{\text{Bro}}; 1)$ is $G^{\mathsf{H}}$-growing for a function $G^{\mathsf{H}}$ closely related to $G^\mathsf{M}$. See Example \ref{ex:brownian_crt_height}. The same family of function can be used to see that the $\beta$-stable trees with fixed height are also growing, see Example \ref{ex:stable_crt_height}.
\item We also treat some more exotic fragmentation trees such as the Haas--Stephenson $k$-ary ssMt, see Examples \ref{ex:haas_stephenson_mass} and \ref{ex:haas_stephenson_mass_cex}.
\end{enumerate}

\subsection{Analytical conditions and generators}\label{ss:an_cond_generators} 
 The previous examples of growing characteristic quadruplets were eventually relying  on a (magical) explicit and simple computation. To check if $(\boldsymbol \Xi; \alpha)$ is growing, one should \textit{a priori} find the family of function $(G_x : x >0)$ satisfying the conditions of Definition \ref{def:growing}; a guess-check strategy only suitable for fakirs. Actually, thanks to the semi-group property, the family of functions $(G_x : x >0)$ can generically be described by a single function: if we denote by $\mathtt{G}_t = G_{\mathrm{e}^{-t}}$ then in smooth situations $\mathtt{G}$ is the flow of solution of a (infinite dimensional) differential equation
  \begin{eqnarray} \label{eq:flot}\partial_t \mathtt{G}_t = \mathtt{V}(\mathtt{G}_t), \quad \mbox{with} \quad \mathtt{G}_0 = \mathrm{Id}, \end{eqnarray} where the operator $\mathtt{V} : \mathcal{E} \to \mathbb{R}^{\mathbb{N}}$ is called the \textbf{generator}. The monotonicity condition \eqref{eq:monotonecoupling} is then easily re-interpreted in terms of the generator and amounts to require that $\mathtt{V} \leq \mathrm{Id}$. Since $\mathtt G$ is the flow of solutions to \eqref{eq:flot}, the measures $\boldsymbol\Xi_t=(\mathtt G_t)_\sharp\, \boldsymbol \Xi$ are obtained by moving the mass of $\boldsymbol\Xi$ along the vector field $\mathtt V$. In PDE theory, under reasonable regularity conditions, the evolution of $(\boldsymbol\Xi_t)_t$ is then dictated by the \textit{continuity equation}:
\begin{align*}
    \partial_t\boldsymbol\Xi_t+\mathrm{div}(\mathtt V\,\boldsymbol\Xi_t)=0,
\end{align*}
understood by duality against test functions, see ~\cite{ambrosio2005gradient}. In our setting, we shall then require that $\mathtt{V}$ obey the divergence PDE \begin{eqnarray} \label{eq:div} \mathrm{div}(\mathtt V\boldsymbol\Xi_t)=-\alpha \boldsymbol\Xi_t \end{eqnarray} so that the preceding (measure-valued) ODE becomes $\partial_t\boldsymbol\Xi_t=\alpha\,\boldsymbol{\Xi}_t$ and thus $\boldsymbol\Xi_t=\mathrm e^{\alpha t}\,\boldsymbol\Xi$ for $t\in\R$.

In particular, when the support of $\boldsymbol\Xi$ is one-dimensional, then the above differential equation fixes the generator and we have at most one growing family of functions for $(\boldsymbol\Xi; \alpha)$. This uniqueness is lost in the multi-dimensional case, see Example \ref{ex:haas_stephenson_mass_cex}. These analytical conditions on the generator are much easier to check  and enables us to compute $\alpha_c(\boldsymbol \Xi)$ in many situations, although the family $G$ extracted from  \eqref{eq:flot} is usually not explicit.  As an example, for $\gamma>0$ consider the  binary conservative locally largest splitting measures
$$ \int F({y_0},{y_1},...) \boldsymbol{\Xi}^{\ell\ell}_{\gamma \text{-bin}}( \d \mathbf{y}) \propto \int_{1/2}^1 \frac{\d s}{(s(1-s))^\gamma} F(s,1-s, 0,0,0...).$$ In this case the generator is explicit in terms of Euler's incomplete beta-function and $\alpha_c (\boldsymbol{\Xi}^{\ell\ell}_{\gamma \text{-bin}})\equiv \alpha_c(\gamma)$ can be computed: When $\gamma=3/2$ we have already seen that $\alpha_c(3/2)=1/2$ so that the growing ability of the Brownian CRT is a critical case. When $\gamma = \frac{5}{2}$, which corresponds to the generalized L\'evy measure of the Brownian growth-fragmentation tree \cite[Chapter 3, Example 3.11]{bertoin2024self}, then $\alpha_{\mathrm{c}}(5/2)$ is the solution $x \approx 1.21$ of the equation 
$$ 256 -162 x -42 x^2 +x^3= 0.$$
In particular, if $\alpha = 3/2$ then the associated ssMt is not growing, whereas it is growing if the self-similarity parameter is $\alpha=1/2$. Interestingly, the ssMt corresponds in this case with the Brownian Cactus tree sitting inside the Brownian disk, \cite{curien2013brownian,bertoin2018random,le2020growth}. Various other examples are discussed in Section \ref{sec:examples} and \ref{sec:generator}.\medskip

\textbf{Acknowledgments.} The authors are supported by the SuPerGRandMa, the ERC Cog no 101087572.

\tableofcontents

\section{Growing pssMp}\label{sec:grow_pssmp}
In this section we show that under the growing condition it is possible to couple positive self-similar Markov processes, and their extensions the decoration-reproduction processes, which serves as the basis in the construction of the ssMt. The main idea is to use a construction of pssMp via stochastic differential equations with jumps. Before we start the construction of our coupling, let us first recall the construction of those objects from \cite{bertoin2024self}.

\subsection{Decoration-reproduction processes}\label{sec:deco-repro}
We provide a shortened version of some of the essential definitions and constructions which we will use. For this, we use the same notation as in Section~2.2 of \cite{bertoin2024self} to which we refer for details. \\

Let us consider the space $\mathcal{S}$ of real-numbered sequences non-increasing after the first term $\mathbf{u} = (u_0 \in \mathbb{R}, ( u_1 \geq u_2 \geq  \dots \geq -\infty))$  with $\lim_{i \to \infty} u_i = -\infty$.  
Applying the exponential function to all elements, we get the space $ \mathcal{E}= \exp( \mathcal{S})$ 
made of sequences starting with a positive term, then non-increasing tending to $0$. 

Borrowing the supremum norm induces a distance on $ \mathcal{E}$  that makes it Polish. Next, we consider a \textbf{generalized L\'evy measure} $\boldsymbol{\Lambda} = \boldsymbol{\Lambda}(\mathrm{d} \mathbf{u}) = \boldsymbol{\Lambda}( \mathrm{d}u_0 \mathrm{d}u_1...)$ on $\mathcal{S}$ and we denote $\Lambda_0$ and $\boldsymbol{\Lambda}_1$ the push-forward images by the projection on the first coordinate $(u_0,u_1, ...) \mapsto u_0 \in \R$ and by the projection on the remaining coordinates $(u_0,u_1, ...) \mapsto (u_1, u_2, ... )$. The push-forward of $ \boldsymbol\Lambda$ on  $\mathcal{E}$ by the exponential is denoted by $\boldsymbol\Xi$ and called \textbf{the splitting measure}. Denote equally $\Xi_0$ and $\boldsymbol{\Xi}_1$ the push-forward measures on $\R_+$ and $\mathcal{E}_1$ of $\Lambda_0$ and $\boldsymbol{\Lambda}_1$ respectively. We suppose that $ \boldsymbol{\Xi}$ satisfies the conditions 
\begin{equation*}
	\int_{-\infty}^\infty (1 \wedge u^2) \Lambda_0( \mathrm{d}u) < \infty \hspace{2mm} \text{ and } \hspace{2mm} \boldsymbol{\Xi}_1\big( \{ (y_1, y_2, \dots) \in \mathcal{E}^{\downarrow} : y_1 > \varepsilon \} \big)  <\infty \hspace{3mm} \text{ for all } \varepsilon > 0. 
\end{equation*}
Furthermore, to be able to compensate all jumps and to get more compact formulas, we suppose in this paper that 
\begin{equation}\label{eq:comp_all_jumps}
    \int_{\R_+} (y_0 - 1)^2 \, \Xi_0(\d y_0) < + \infty
\end{equation}
We stress that for Theorem~\ref{thm:deco-repro-grow} and Theorem~\ref{thm:mainprecise} to hold, condition \eqref{eq:comp_all_jumps} is superfluous. The changes that need to be implemented are discussed in Section~\ref{ss:general_case}. \\

\noindent Next, consider a quadruplet $(a, \sigma^2, \boldsymbol{\Lambda}; \alpha)$, with $a \in \R$ a drift coefficient, $\sigma^2 \geq 0$ a diffusion coefficient, $\boldsymbol{\Lambda}$ a generalized L\'evy measure and $\alpha > 0$ the index of self-similarity. We call $(a, \sigma^2, \boldsymbol{\Lambda} ; \alpha)$ the \textbf{characteristic quadruplet}. Notice that compared to \cite{bertoin2024self}, we impose here that the {killing rate}  $\mathtt{k} := \Lambda_0(\{-\infty\})$ is zero. See Section \ref{sec:killing} for a discussion on it. In the language of \cite[Assumption 2.8]{bertoin2024self}, the corresponding \textbf{cumulant function} is defined by
\begin{align} \label{def:kappa}
	\kappa(\gamma)&= \frac{1}{2}\sigma^2\gamma^2 + \mathrm{a} \gamma + \int_{ \mathcal{S}} \boldsymbol{\Lambda}( \d \mathbf{u}) \left( \mathrm{e}^{\gamma  u_0}-1 - \gamma u_0 \mathbf{1}_{|u_0| \leq 1} + \sum_{i=1}^\infty \mathrm{e}^{\gamma u_i}\right),\\
    &= \frac{1}{2}\sigma^2\gamma^2 + \mathrm{a} \gamma + \int_{ \mathcal{E}} \boldsymbol{\Xi}( \d \mathbf{y}) \left( y_0^\gamma-1 - \gamma \log y_0 \mathbf{1}_{|\log y_0| \leq 1} + \sum_{i=1}^\infty y_i^\gamma\right),
	\qquad \gamma> 0.\nonumber
\end{align}
Its support $\mathrm{supp}(\kappa)$ is the set of $\gamma>0$ such that $\kappa(\gamma)<\infty$. 
We shall henceforth suppose that we are in the \textbf{(sub)critical} case, that is there exists $0<\gamma_0<\gamma_1< \infty$ such that $\kappa(\gamma_1)< \infty$ and 
\begin{eqnarray} \label{def:subertoin2024selfitical} 
\kappa(\gamma_0) \leq 0. \end{eqnarray}
Equipped with this quadruplet, let $(W_s)_{s \geq 0}$ be a standard Brownian motion and $\textbf{N} := \textbf{N}(\d t, \d\mathbf{u})$ a Poisson random measure on $\R_+ \times \mathcal{S}$ with intensity measure $\d t \boldsymbol{\Lambda}(\d \mathbf{u})$.  
We suppose that $W$ and $\textbf{N}$ are independent. We denote $\textbf{N}_0$ the first projection of $\textbf{N}$ onto $\R_+ \times \R$, which, by the mapping theorem, is a Poisson point process with intensity $\mathbbm{1}_{u \in \R} \d t \Lambda_0(\d u)$. Denote also
\begin{equation*}
	\tilde{\textbf{N}}_0(\d s, \d u) = \textbf{N}_0(\d s, \d u) - \d s \Lambda_0(\d u)
\end{equation*}
the compensated Poisson measure. Then, for $t \in [0,\infty)$, we define the L\'evy process
\begin{equation*}
	\xi_t = at+ \sigma B_t + \int_0^t \int_{|u| > 1} u \hspace{1mm} \textbf{N}_0(\d s, \d u) + \int_0^t \int_{|u| \leq 1} u \hspace{1mm} \tilde{\textbf{N}}_0(\d s, \d u),
\end{equation*}
and its shifted version $\xi^{(x)} := \xi + \log x$. In particular, the L\'evy--Khintichine exponent of $\xi$ is given by $ \mathbb{E}[\exp(\gamma \xi_t)] = \exp( t \psi(\gamma))$ where 
 \begin{eqnarray} \label{eq:levykhintchine}\psi(\gamma) =  \frac{1}{2}\sigma^2\gamma^2 + \mathrm{a}\gamma + \int_{ \mathbb{R}_{*}} \big( \mathrm{e}^{\gamma u}-1- \gamma u \mathbbm{1}_{|u|\leq1} \big) \Lambda_0( \d u). \end{eqnarray}
 
 We now introduce the necessary definitions and quantities for the Lamperti transformation. Consider the following exponential functionals:
\begin{equation*}
	\vartheta(t) := \int_0^t \exp{(\alpha \xi_s)} \d s \hspace{5mm} \text{for } 0 \leq t < \infty, \hspace{5mm} z:= \vartheta(\infty) = \int_0^\infty \exp{(\alpha \xi_s)} \d s < \infty \hspace{3mm} \text{a.s.}
\end{equation*}
Noting that $\vartheta: [0,\infty] \to [0, z]$ is an increasing bijection a.s., we can define $\theta$ as the reciprocal bijection, which gives
\begin{equation*}
	\int_0^{\theta_t} \exp{(\alpha \xi_s)}  \d s = t, \hspace{5mm} \text{ for all } 0 \leq t < z.
\end{equation*}
Then the famous Lamperti transformation of $\xi$,
\begin{equation*}
	X_t = \exp{(\xi(\theta_s))}, \hspace{5mm} \text{for all } 0 \leq t < z,
\end{equation*}
is both Markovian and self-similar with exponent $\alpha > 0$. More precisely, $X_0 = 1$ and for every $x > 0$, the rescaled process 
\begin{equation*}
	\tilde{X}_t := x X_{x^{-\alpha} t} , \hspace{5mm} \text{for all } 0 \leq t < x^\alpha \cdot z,
\end{equation*}
is a version of $X$ started from $x > 0$, meaning that the law of $\tilde{X}$ equals that of the the Lamperti transformation of the shifted L\'evy process $\xi^{(x)}$. The lifetime $z$ of $X$ is finite almost-surely iff $\lim_{t \to \infty} \xi_t = -\infty $ almost surely, which is always the case in the subcritical case \eqref{def:subertoin2024selfitical} which we supposed. We then manually set $X_t := 0$ for all $t \geq z$, so that $X$ is a positive process on the segment $[0,z]$ and killed/ absorbed at $0$ at time $z$. The quadruplet $( \mathrm{a}, \sigma^2, \Lambda_0 ; \alpha)$ will be called the \textbf{characteristics} of $X$ in what follows.\\

Next, we define the reproduction process $\eta$. First, we define $\hat{\eta}$ on $[0, \infty)$. The second projection of $\textbf{N}$ is denoted by $\textbf{N}_1 := \textbf{N}_1(\d t,\d \mathbf{u})$. We then \textit{expand} each atom $(s,\mathbf{u})$ as a sequence $(s, u_i)_{i \geq 1}$ in $\R_+ \times [-\infty , \infty)$. Define the point process $\hat{\eta}$ on $[0, \infty) \times \R_+$ as
\begin{equation*}
	\hat{\eta} := \sum_{(s,\mathbf{u}) \in \textbf{N}_1} \sum_{i \geq 1}  \delta_{(s, u_i)},
\end{equation*}
which, after the push-forward by the Lamperti time change, yields
\begin{equation*}
	\eta := \sum_{(s,\mathbf{u}) \in \textbf{N}_1} \sum_{i \geq 1} \mathbbm{1}_{\{ \vartheta(s) < z \}} \delta_{(\vartheta(s), X_{\vartheta(s)-} \cdot \exp{(u_i)})},
\end{equation*}
making it a point process on $[0,z) \times [0, \infty)$. \\

We define the law of the couple $(X, \eta)$ as $P = P_1$ and for all $x > 0$, we denote $P_x$ the law of the transformed process $(X^{(x)}, \eta^{(x)})$, where the scaling for $\eta$ is done by transforming the atoms $\delta_{(s,u)}$ into $\delta_{(x^\alpha s, xu)}$. The kernel $(P_x)_{x > 0}$ is called the \textit{self-similar Markov decoration-reproduction kernel} associated to the characteristic quadruplet $(a, \sigma^2, \boldsymbol{\Lambda}; \alpha)$.

\begin{figure}
    \centering
    \includegraphics[width=0.45\linewidth]{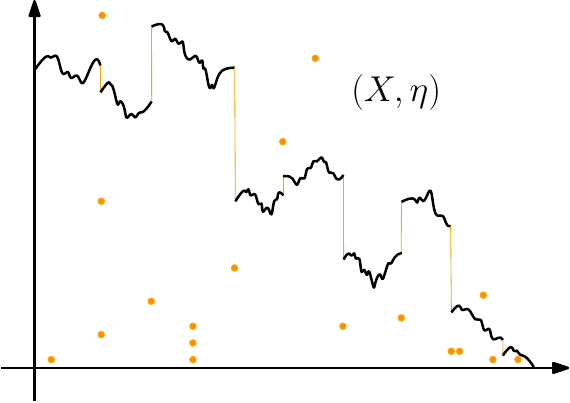}
    \caption{Illustration of a decoration-reproduction process $(X, \eta)$ driven by a characteristic quadruplet $(a, \sigma^2, \boldsymbol{\Lambda}; \alpha)$.}
\end{figure}

\subsection{Growing hypothesis}\label{sec:mag_func}
We make here precise what we mean by the growing assumption.
Fix a splitting measure $\boldsymbol\Xi$ and a self-similarity parameter $\alpha>0$. Recall that we supposed that $\boldsymbol\Xi$ has no killing term $\Xi_0(\{  {0} \})=0$.

\begin{definition}[$G$-Growing] \label{def:growing} We say that $(\boldsymbol\Xi; \alpha)$ is growing if there exists a measurable $\mathcal E_0\subset \mathcal E$ with $\boldsymbol\Xi(\mathcal E\setminus\mathcal E_0)=0$, and a family $G=(G_x : x > 0)$ of mappings $ \mathcal{E} \to \mathcal{E}$ verifying the following assumptions:
\begin{itemize}
\item \textbf{Semi-group property.} The function $G_1$ is the identity and for all $x,x'>0$ we have 
$$ G_x \circ G_{x'} = G_{x \cdot x'}.$$
    \item \textbf{Quasi-preservation of the measure.} For all $\phi \in C_b(\mathcal{E})$, the set of continuous bounded functions from $\mathcal{E}$ to $\R_+$, and for all $x>0$, we have
	\begin{equation}\label{eq:ode_xi}
		\int_{\mathcal{E}} \phi \left( G_x( \mathbf{y}) \right) \boldsymbol{\Xi}(\d  \mathbf{y}) = x^{-\alpha} \int_{\mathcal{E}} \phi \left( \mathbf{y} \right) \boldsymbol{\Xi}(\d \mathbf{y}).
	\end{equation}
    \item \textbf{Monotonicity.} For each $ \mathbf{y} \in \mathcal{E}_0$ the map \begin{eqnarray}
        \label{eq:monotonicity} x \mapsto x \cdot G_x( \mathbf{y}) \quad \mbox{is increasing (coordinate wise)},\end{eqnarray} so that in particular for all $x \in (0,1]$, we have $x \cdot G_x( \mathbf{y}) \leq \mathbf{y}$.
    
    \item \textbf{Regularity of $G$.} The mapping $(x, \mathbf{y}) \mapsto G_x( \mathbf{y})$ is continuous from  $(0,\infty)\times \mathcal{E}_0$ to $\mathcal E_0$ when $\mathcal E_0\subset \R_+^\N$ is equipped with the product topology. Moreover, if $G^{(0)}$ is the first coordinate of $G$, then for every compact set $K \subset (0, \infty)$, there exists a constant $C_K > 0$ such that for all $x,x' \in K$, 
	\begin{equation} \label{ass:lipschitz}
		\int_{\mathcal{E}} \left( x \left( G_x^{(0)}( \mathbf{y})- 1 \right) - x' \left(G_{x'}^{(0)}( \mathbf{y}) - 1 \right) \right)^2 \boldsymbol{\Xi}(\d \mathbf{y}) \leq C_K\cdot (x-x')^2.
	\end{equation}
\end{itemize}
When $(\boldsymbol\Xi; \alpha)$ is growing for some  specific family $G$, we say that it is $G$-growing.
\end{definition}

Let us make a few comments on this definition:
\begin{itemize}
\item \textbf{Dependence in $\alpha$.} The existence of $G_x$ satisfying the semi-group and the quasi-preservation of the measure is true in most of the cases, however the monotonicity assumption  generally puts a constraint on the self-similarity parameter $\alpha$ in function of $\boldsymbol\Xi$.
Lemma~\ref{lem:alpha-vs-cumulant} and \ref{lem:set-of-alpha-is-interval} below give  quantitative restrictions on the values that $\alpha$ can take and show it must be an interval, implying that there must be a maximal $\alpha_c$ for which $(\boldsymbol{\Xi}; \alpha)$ is growing.
\item \textbf{Atoms.} Note that the quasi-preservation of the measure forces the function $G_x$ to ``continuously shift" the mass of the splitting measure $\boldsymbol\Xi$, for different values of $x > 0$. In particular, such a family of functions cannot exists if $\boldsymbol\Xi$ has an atom in $\mathcal{E}$. See Section \ref{sec:killing} for a discussion about the re-incorporation of a killing term.
\item \textbf{Infinite intensity.} For $( \boldsymbol\Xi; \alpha)$ to grow, one must necessarily have $\boldsymbol\Xi( \mathcal{E}) = \infty$ thus excluding the finite intensity ssMt.
\item \textbf{Interpretation when $x>1$.} The ``geometric'' interpretation made in Figure \ref{fig:emulation} of $G_x$ when $x>1$ does not hold anymore. However, if $G_x$ is defined for all $x \in(0,1)$, we can just set $G_{x^{-1}} = (G_x)^{-1}$ which is well defined on the support of $\boldsymbol \Xi$ since $G_x$ is supposed to be bijective on the support of $\boldsymbol \Xi$.
\item \textbf{Weaker assumption.} Whenever condition \eqref{eq:comp_all_jumps} does not hold, the condition \eqref{ass:lipschitz} needs to be weakened by replacing $\mathcal{E}$ with the set $\mathcal{E}_0 = \{ (y_0, \mathbf{y}) \in \mathcal{E}: y_0 \in [e^{-1}, e] \}$. We refer to Section~\ref{ss:general_case} to see what changes need to be made in the proof of Theorem~\ref{thm:deco-repro-grow} under this weaker assumption.

\end{itemize}

\begin{lemma}[Upper bound on $\alpha$]\label{lem:alpha-vs-cumulant}
	Let $\boldsymbol{\Xi}$ be a splitting measure such that $\boldsymbol{\Xi}(\{ \mathbf{y} \in \mathcal{E}: y_1 = 0 \}) = 0$
      and let $\alpha > 0$. Then, for the pair $(\boldsymbol\Xi;\alpha)$ to be $G$-growing, it is necessary that
	\begin{align}\label{eq:bound_alpha}
		\alpha\leq 
			\inf\left\{\gamma>0\colon
			\int_{\mathcal E} y_1^\gamma\,\boldsymbol{\Xi}\bigl(\d (y_0,y_1,\dots)\bigr) <\infty
			\right\}.
	\end{align}
	In particular, if $(\mathrm a,\sigma^2,\boldsymbol{\Xi};\alpha)$ is a non-trivial characteristic quadruplet such that $(\boldsymbol{\Xi};\alpha)$ is $G$-growing, then $\alpha$ is at most the infimum of the support of the corresponding cumulant function $\kappa$ defined in \eqref{def:kappa}.
\end{lemma}

\begin{proof}
    Suppose that $(\boldsymbol\Xi;\alpha)$ is $G$-growing, and let $(G_x,x>0)$ be a family of mappings $\mathcal E\to\mathcal E$ satisfying the conditions in Definition~\ref{def:growing}.
	For $z>0$, we let:
	\begin{align*}
		{E}_z=
		\left\{
			\mathbf y=(y_0,y_1,\dots)\in\mathcal E\colon  y_1\geq z
		\right\}.
	\end{align*}
    The assumption on the support of $\boldsymbol\Xi$ gives that there is some $z_0>0$ such that $\boldsymbol\Xi(E_{z_0})>0$.
	The mapping $z\mapsto z\cdot{G_z(\mathbf y)}$ is coordinate-wise non-decreasing, so that for all $z\geq z_0$ we have $G_z(E_{z_0})\subset E_{z_0/z}$, and thus $E_{z_0}\subset G_z^{-1}(E_{z_0/z})$. Hence, by the ``quasi-preservation of the measure'', we have the inequality
	$
		\boldsymbol{\Xi}(E_{z_0})\leq \boldsymbol{\Xi}\bigl(G_z^{-1}(E_{z_0/z})\bigr)
			=z^{-\alpha} \boldsymbol{\Xi}(E_{z_0/z}).
	$
	  In terms of $x=z_0/z\in (0,1]$, this entails that:
	\begin{align*}
		\boldsymbol\Xi(E_x)\geq C_0 x^{-\alpha},\qquad x\in(0,1],
	\end{align*}
    where $C_0:=z_0^\alpha\boldsymbol\Xi(E_{z_0})$ is positive since $\boldsymbol\Xi(E_{z_0})>0$.
	This estimate suffices to conclude. Indeed, for all $\gamma>0$, we have by the Fubini--Tonnelli theorem and the preceding estimate:
	\begin{align*}
		\int_{\mathcal E}\boldsymbol{\Xi}\bigl(\d \mathbf y\bigr)\, y_1^\gamma 
			= 
				\int_{\mathcal E}
					\boldsymbol{\Xi}\bigl(\d \mathbf y\bigr)
				\int_0^\infty \d x\, 
				\mathbf 1_{\{x\leq y_1^\gamma\}}
			= \int_0^\infty \boldsymbol\Xi\bigl(E_{x^{1/\gamma}}\bigr)\d x
			\geq \int_0^1\d x\, C_0 x^{-\alpha/\gamma}.
	\end{align*}
    Since $C_0>0$, if the left-hand side is finite, then we have $\alpha<\gamma$. The conclusion follows. The consequence in terms of the support of $\kappa$ is immediate from the definition \eqref{def:kappa}.
\end{proof}

\begin{lemma}[Critical self-similarity]\label{lem:set-of-alpha-is-interval}
	Let $\boldsymbol{\Xi}$ be a splitting measure. Then the set of $\alpha>0$ such that $(\boldsymbol\Xi;\alpha)$ is $G$-growing is either empty, or an interval of the form $(0,\alpha_c)$ or $(0,\alpha_c]$ for some $\alpha_c=\alpha_c(\boldsymbol\Xi)>0$.
\end{lemma}

\begin{proof}
	It suffices to show that if $(\boldsymbol{\Xi};\beta)$ is $G$-growing, then so is $(\boldsymbol{\Xi};\alpha)$ when $0 \leq \alpha \leq \beta$.
	We therefore assume the former and let $(G_x,x>0)$ be a family of mappings $\mathcal E\to\mathcal E$ satisfying the conditions of Definition~\ref{def:growing} for the pair $(\boldsymbol\Xi;\beta)$. We set:
	\begin{align*}
		\widetilde G_x = G_{x^{\alpha/\beta}},\qquad x>0.
	\end{align*}
It is then straightforward to check that $(\widetilde G_x,x>0)$ satisfies the conditions of Definition~\ref{def:growing} for the pair $(\boldsymbol\Xi;\alpha)$. Indeed, the semi-group property and ``quasi-preservation of the measure'' are readily verified. The monotonocity condition follows by observing that for all $\mathbf y\in\mathcal E$, the mapping $x\mapsto x\cdot  {\widetilde G_x(\mathbf y)}$ is the product of the non-decreasing function $x\mapsto x^{1-\alpha/\beta}$ with the coordinate-wise non-decreasing mapping $x\mapsto x^{\alpha/\beta}\cdot { G_{x^{\alpha/\beta}}(\mathbf y)}$.
	Lastly, we check \eqref{ass:lipschitz}. Let $0 < z \leq x$ in a compact $K \subset (0, +\infty)$. Then a straightforward calculation shows, using a varying constant $C_K$,
    \begin{align*}
        &\int_{\mathcal{E}} \left( x \left( \tilde{G}_x^{(0)}( \mathbf{y})- 1 \right) - z \left( \tilde{G}_z^{(0)}( \mathbf{y}) - 1 \right) \right)^2 \boldsymbol{\Xi}(\d \mathbf{y}) \\
        & \leq C_K \int_{\mathcal{E}}\left( x^{\alpha / \beta} \left( G_{x^{\alpha / \beta}}^{(0)}( \mathbf{y})- 1 \right) - z^{\alpha / \beta} \left( \tilde{G}_{z^{\alpha / \beta}}^{(0)}( \mathbf{y}) - 1 \right) \right)^2 \boldsymbol{\Xi}(\d \mathbf{y})\\
        & + C_K \int_{\mathcal{E}} \left( z^{\alpha / \beta} - x^{(\alpha / \beta) - 1} z \right)^2 \left( G_{x^{\alpha / \beta}}^{(0)}( \mathbf{y})- 1 \right)^2 \boldsymbol{\Xi}(\d \mathbf{y}) \\
        &\leq C_K \left( x^{\alpha / \beta} - z^{\alpha / \beta} \right)^2 + C_K \left( z - x \right)^2 \int_\mathcal{E} (y_0 - 1)^2 \Xi_0(\d y_0) \leq C_K (z-x)^2,
    \end{align*}
    where we used that $x,z \in K$ in the first inequality and conditions \eqref{ass:lipschitz} and \eqref{eq:comp_all_jumps} in the last inequality. The last condition for $(\boldsymbol\Xi;\alpha)$ to be $G$-growing is thus verified.
\end{proof}

\begin{remark}
	We expect that in most reasonable cases, the set of self-similarity parameters described in Lemma~\ref{lem:set-of-alpha-is-interval} is in fact of the form $(0,\alpha_c]$, by taking suitable limits of families $(G_{x}: x>0)$ satisfying the conditions of Definition~\ref{def:growing} for pairs $(\boldsymbol\Xi;\alpha_n)$ with $\alpha_n\to\alpha_c$ from below.
\end{remark}

\subsection{Coupling decoration-reproduction processes}\label{sec:coupling_all}
 We now present the technical crux of the paper which consists in a pssMp version of Theorem \ref{thm:main}:
\begin{theorem}[Baby version of Theorem \ref{thm:main} for a single branch] \label{thm:deco-repro-grow} Let $( \mathrm{a}, \sigma^2, \boldsymbol\Xi ; \alpha)$ be a characteristic quadruplet with subcritical cumulant $\kappa$ and without killing such that $(\boldsymbol\Xi; \alpha)$ is $G$-growing. Then, on a common probability space $(\Omega, \mathcal{F}, \PP)$, one can construct a family of decoration-reproduction processes $$\left\{ \left( X^{(x)}, \eta^{(x)} \right): X_0^{(x)} = x \right\}_{x > 0}\quad \mbox{absorbed at $0$ at time $z^{(x)}$ }$$ for which the following properties hold $\PP$-almost surely:
\begin{enumerate}[label=(\roman*)]
\item (\textbf{Law $P_x$}) For every $x > 0$, the process $\left( (X_t^{(x)})_{t\geq0}, \eta^{(x)} \right)$ has law $P_x$.
	\item (\textbf{Monotonicity}) For every $0 < x' \leq x$, and for all $t \geq 0$,
	\[
	X^{(x')}_t \leq X^{(x)}_t.
	\] 
    	\item (\textbf{Synchronized and monotone reproduction}) For every $\, 0 < x' \leq x$ and for all $t \geq 0$, if $X^{(x')}_{t-} > 0$, then $\Delta X^{(x')}_t > 0 \iff \Delta X^{(x)}_t > 0$, i.e. processes have the same jump times, whenever both are positive. Also, for every unexpanded atom $\left( s, \mathbf{y}^{(x')} \right)$ of $\eta^{(x')}$, there is exactly one corresponding atom $\left( s, \mathbf{y}^{(x)} \right)$ of $\eta^{(x)}$, and we have coordinate-wise,
	\[
	\mathbf{y}^{(x')}  \leq   \mathbf{y}^{(x)}
	\]
	This means that at every birth event, the labels of all the children of the process $\left( X^{(x')}, \eta^{(x')} \right)$ are always smaller than those of the corresponding children of $\left( X^{(x)}, \eta^{(x)} \right)$.
	\item (\textbf{Continuity}) The process $x \mapsto \left( X^{(x)}, \eta^{(x)} \right)$ is continuous on the space $\D \times \mathbb{S}_{\R_+^2}$ for the product topology of the Skorokhod $J_1$ topology on the space $\D$ of c\`adl\`ag functions absorbed when touching $0$ and the topology induced by convergence of compactly supported continuous functions on $(0,\infty)^2$. Furthermore, for every $x \geq 0$ fixed, the process $t \mapsto X^{(x)}_t$ is c\`adl\`ag and similarly for every $t \geq 0$ fixed, the process $x \mapsto X^{(x)}_t$ is continuous.

	\item (\textbf{Measurability}) For every $0 < x' \leq x$, the decoration-reproduction process $\big( X^{(x')}, \eta^{(x')} \big)$ is a measurable function of   $\big(X^{(x)}, \eta^{(x)} \big)$. Conversely, the process $\big(X^{(x)}, \eta^{(x)} \big)$ restricted to $[0,z^{(x')}]$ is a measurable function of $\big( X^{(x')}, \eta^{(x')} \big)$.
    \item (\textbf{Markov property}) For every $\, 0 < x' \leq x$, conditionally on $(X^{(x')}, \eta^{(x')})$, the decoration reproduction-process $(X^{(x)}, \eta^{(x)})$ is entirely known  on $[0,z^{(x')}]$ and the law of the shifted process on the time interval $[z^{(x')}, \infty)$ is  $P_{X^{(x)}_{z^{(x')}}}$.
\end{enumerate}
\end{theorem}

\begin{figure}
    \centering
    \includegraphics[width=0.4\linewidth]{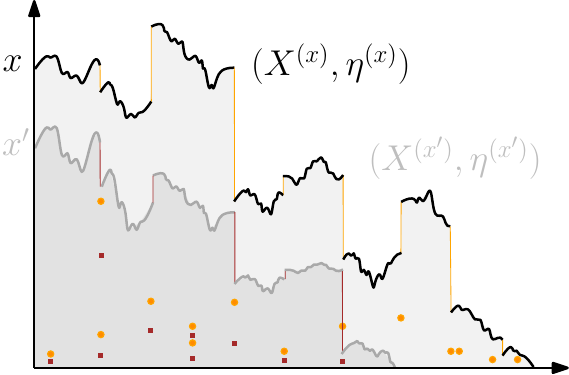} \includegraphics[width=0.4\linewidth]{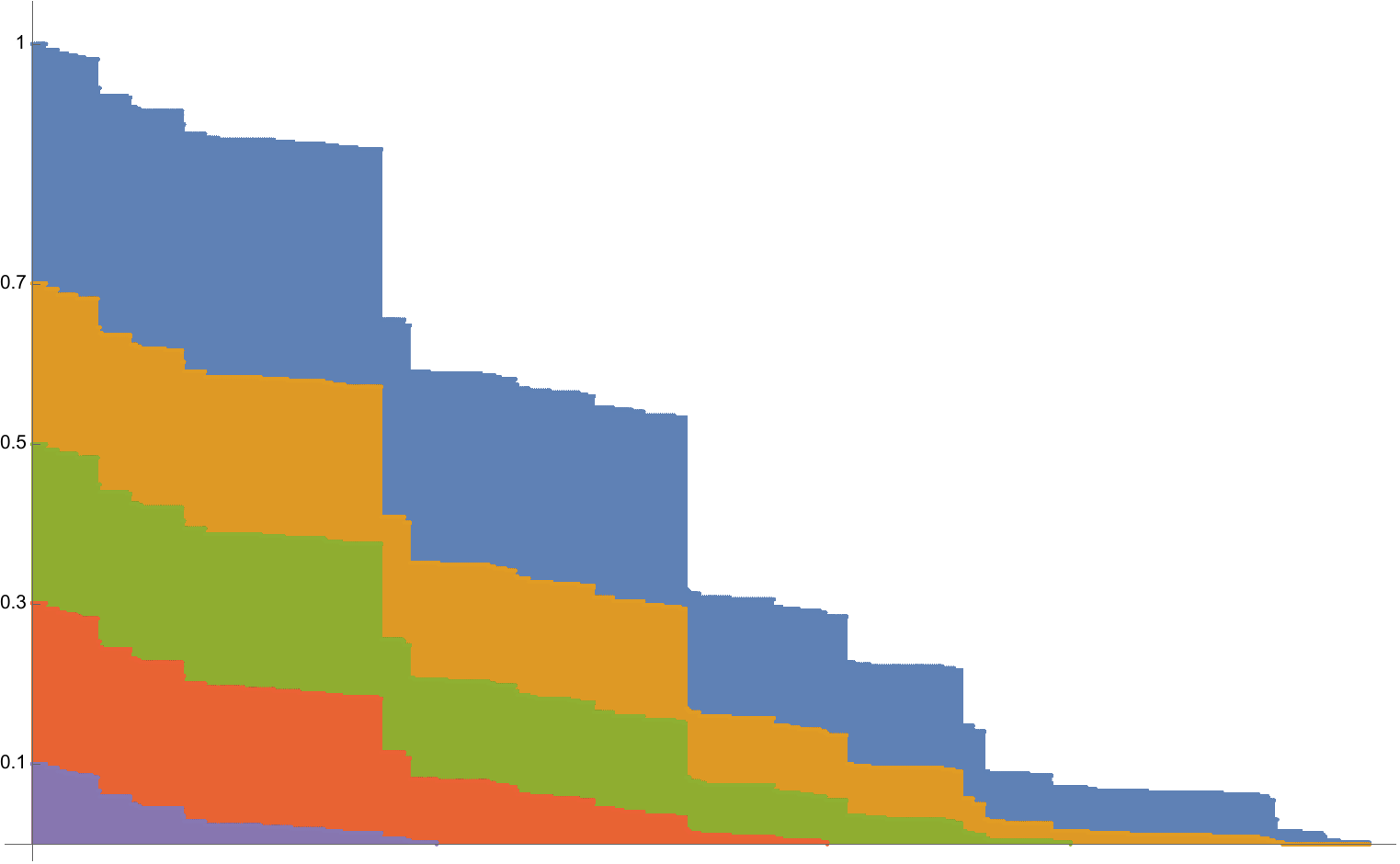}
    \caption{Illustrations of Theorem \ref{thm:deco-repro-grow}. Left: two coupled decoration-reproduction processes, in particular notice that the decoration $\eta^{x'}$ is made of atoms which are coordinate wise smaller than those of $\eta^{(x)}$. Right: A simulation of coupled pssMp appearing in the construction of the Brownian CRT (the reproduction are not displayed for better readeability).}
\end{figure}

The high-level idea of the proof is to represent the decoration (and reproduction) processes $X^{(x)}$ as flows to a stochastic differential equation using a common Brownian motion and Poisson measure. The fact that pssMp can be written as solution to SDEs is well-known, see in particular the recent work of D\"oring \& Barczy \cite{DB} who proved the existence and uniqueness of pathwise strong solutions to an SDE yielding quite general pssMp. They use it to  (re)start the pssMp at zero, which is a well-known subtle problem. The  SDE they used to construct a pssMp $X$ with characteristics $( \mathrm{a}, \sigma^2, \Lambda_0;\alpha)$  is
\begin{align*}
    X_t = x + \mathrm{a} \int_0^t X_s^{1 - \alpha} \d s + \sigma \int_0^t X_s^{1 - \alpha/2} \d B_s  &+ \int_0^t \int_0^\infty \int_{|u| > 1} \mathbbm{1}_{\{r X_{s-}^\alpha \leq 1\}} X_{s-} (e^u - 1) \mathcal{N}(\d s,\d r,\d u) \\
        &+ \int_0^t \int_0^\infty \int_{|u| \leq 1} \mathbbm{1}_{\{r X_{s-}^\alpha \leq 1\}} X_{s-} (e^u - 1) \tilde{\mathcal{N}}(\d s,\d r,\d u) \\
        &+ \int_0^t \int_0^\infty  \mathbbm{1}_{\{r X_{s-}^\alpha \leq 1\}} X_{s-} \hspace{1mm} \mathcal{M}(\d s,\d r),
\end{align*}
Here  $(B_s)_{s \geq 0}$ is a standard Brownian motion, $\mathcal{N}$ a Poisson random measure associated to a Poisson point process with intensity $\d s \d r \Lambda_0(\d u)$ and $\tilde{\mathcal{N}} = \mathcal{N} - \mathcal{N}'$ its compensated Poisson random measure. Lastly, $\mathcal{M}$ is a Poisson random measure associated to another Poisson point process with intensity $ \mathtt{k} \hspace{1mm} \d s \d r$. The objects $B$, $\mathcal{N}$ and $\mathcal{M}$ are all taken to be independent. The pathwise existence and uniqueness to such SDE is not straightforward due to the non-smoothness of the indicator function making the classical Ikeda-Watanabe criterion unapplicable when the L\'evy measure $\Lambda_0$ does not integrate $|u|$ near $0$. However, it follows from very recent techniques in the field, see in particular \cite{li2012strong}, that those equations indeed admit strong solutions. Our approach is however different since we shall use the $G$-growing assumption to replace the annoying indicator function with a smoother ``Lipschitz" function involving $G$. This is obviously far less general (since it only works for $G$-growing cases) but shows via the standard Ikeda--Watanabe criterion the  existence of pathwise unique strong solutions. However, the continuity of the flow of solutions $x \mapsto X^{(x)}$ is non trivial and will involve  the aforementioned subtle problem of the behavior of the SDE at the singular point $0$.

\begin{proof}[Proof of Theorem \ref{thm:deco-repro-grow}] The proof of the theorem, the technical crux of this paper, will occupy us for the rest of the section. Suppose that $(\alpha; \boldsymbol\Xi)$ is $G$-growing, in particular $\boldsymbol \Xi$ has no killing. Recall that in Definition \ref{def:growing} we supposed that $\Xi_0$ integrates $(y_0 - 1)^2$ globally. We refer the reader to the discussion in Section~\ref{ss:general_case} for the necessary adaptations in the general case. 

\paragraph{Step 1: Decorations as  strong solutions to an SDE.} First, we focus on the decoration processes $X^{(x)}$. We denote by $G_\cdot^{(i)}(\cdot) : \mathbb{R}_+ \times \mathcal{E} \to \mathbb{R}_+$ for $i \in\{0,1,2, ...\}$ the coordinates of $ G_\cdot(\cdot)$ and consider the following SDE:
\begin{equation}\label{eq:SDE}
	\begin{aligned}
		X_t &= x + \beta \int_0^t X_s^{1 - \alpha} \d s + \sigma \int_0^t X_s^{1 - \alpha/2} \d B_s + \int_0^t \int_{\mathcal{E}} X_{s-} \cdot \left( G_{X_{s-}}^{(0)}( \mathbf{y}) -1 \right) \ \tilde{\mathcal{N}}(\d s, \d \mathbf{y})
	\end{aligned}
\end{equation}
Here $\beta = a + \frac{\sigma^2}{2} + \int_{\R_+} (y_0 - 1 - \log y_0) \Xi_0( \d y_0)$ is the compensated  drift, $(B_s)_{s \geq 0}$ is a standard Brownian motion, $\mathcal{N}$ a Poisson random measure on $\R_+ \times \mathcal{E}$ associated to a Poisson point process with intensity $\d s \boldsymbol{\Xi}(\d \mathbf{y})$ and $\tilde{\mathcal{N}} = \mathcal{N} - \mathbb{E}[\mathcal{N}]$ its compensated Poisson random measure. The objects $B$ and $\mathcal{N}$ are taken to be independent. For simplicity of notation, we define, for $(x, \mathbf{y}) \in \R_+ \times \mathcal{E}$,
\begin{equation*}
	\beta(x) = \beta x^{1 - \alpha}, \quad \sigma(x) = \sigma x^{1 - \alpha/2} \quad \text{ and } \quad h(x, \mathbf{y}) = x \cdot \left(G_x^{(0)}( \mathbf{y})-1 \right)
\end{equation*}
This allows us to rewrite equation \eqref{eq:SDE} as
\begin{equation}\label{eq:SDE_rewritten}
	\begin{aligned}
		X_t &= x + \int_0^t \beta(X_s) ds + \int_0^t \sigma(X_s) dB_s + \int_0^t \int_{ \mathcal{E}} h(X_{s-}, \mathbf{y}) \tilde{\mathcal{N}}(\d s, \d \mathbf{y})
	\end{aligned}
\end{equation}

\noindent
We mean by solution to \eqref{eq:SDE} a random c\`adl\`ag function $(X_t)_{t \geq 0}$ taking values in $[0,+ \infty]$ and absorbed at either $0$ or $+\infty$ at time $\chi$, given by
\[
\chi = \inf \{ t \geq 0: \inf_{0 \leq s < t} X_s = 0 \text{ or } \sup_{0 \leq s < t} X_s = + \infty \},
\]
defined on a probability space $(\Omega, \mathcal{F}, \PP)$ with filtration $\mathcal{F}_t$ such that
\begin{enumerate}[label=(\roman*)]
    \item there exists a Brownian motion $B$ and Poisson random measure $\mathcal{N}$ which are $\mathcal{F}_t$-adapted,
    \item $X$ is $\mathcal{F}_t$-adapted,
    \item $\PP$-almost surely, for all $t < \chi$, the function $X_t$ satisfies the equation \eqref{eq:SDE}.
\end{enumerate} 
Furthermore, it is standard that a strong solution is one such that $X$ is adapted to the filtration generated by the Brownian motion and the Poisson random measure and pathwise uniqueness means that if $X, \tilde{X}$ are both solutions with $X_0 = \tilde{X}_0 = x > 0$, then $\PP$-almost surely, for all $t \geq 0$, we have $X_t = \tilde{X}_t$. 

\begin{proposition}\label{prop:puss_until_tau}
	Fix $x >0$.  The SDE \eqref{eq:SDE} has a pathwise unique strong solution and its law coincides with that of the pssMp started from $x$ with characteristics $( \mathrm{a}, \sigma^2, \Lambda_0 ; \alpha)$. In particular, it is a.s. absorbed at $0$.
\end{proposition}
\begin{proof}
	We first consider the domain $[\varepsilon,  \varepsilon^{-1}]$, for $0 < \varepsilon < x < \varepsilon^{-1}$. 
	Here, by \eqref{ass:lipschitz} all the integrands $\beta(\cdot), \sigma(\cdot)$ and $\int_{\mathcal{E}} h(\cdot, \mathbf{y}) \boldsymbol{\Xi}(\d \mathbf{y})$ are locally Lipschitz and satisfy the classical Ikeda--Watanabe criterion (see (9.3) in Theorem~IV.9.1 of \cite{IkedaWatanabe1989stochastic}) so that there exists a pathwise unique strong solution $X$ to \eqref{eq:SDE} until the exit time of the compact $[ \varepsilon, \varepsilon^{-1}]$, i.e.
    \[
    z_\varepsilon = \inf \{ t \geq 0 : X_t \notin [\varepsilon, \varepsilon^{-1}] \}.
    \]
    Let us show that $X$ has the same law as the pssMp with characteristics $( \mathrm{a}, \sigma^2, \Lambda_0 ; \alpha)$ in the same domain. To see this, let us perform the inverse Lamperti transformation and put  \begin{eqnarray} \label{eq:inverselamperti} \xi_t = \log X_{\vartheta_t}, \quad \mbox{where}\quad t=\int_0^{\vartheta_t} \frac{ \d s}{X_s^\alpha}. \end{eqnarray} 
    Using $\dot{\vartheta}_t = X_t^{-\alpha}$ in the time-change formula and Ito's formula, we see that $\xi$ satisfies an SDE whose drift term is simply $(\beta - \frac{\sigma^2}{2}) \d t$, whose diffusion term is $\sigma X_{\vartheta_t}^{1-\alpha/2}  \frac{1}{X_{\vartheta_t}} \sqrt{\dot{\vartheta_t}}\d B_{X_{\vartheta_t}} = \sigma \d \tilde{B}_t$ (for another Brownian motion $\tilde{B}$) and where the last term is an integral against the compensated measure $\tilde{\mathcal{M}}$, where $ \mathcal{M}$ is the image of the measure $\mathcal{N}$ by the map which sends an atom at $(\vartheta_t, \mathbf{y})$ to the atom at $(t, G_{X_{\vartheta_t-}}^{(0)}(\mathbf{y}))$:
    $$ \xi_t = \int_0^t \left( \beta-\frac{\sigma^2}{2} \right) \mathrm{d}s + \int_0^t \sigma \mathrm{d} \tilde{B}_s + \int_0^t\int_{\mathcal{E}} (y_0 - 1) \tilde{ \mathcal{M}}( \d s, \d \mathbf{y}) .$$ Let us assume for a moment that the random measure $ \mathcal{M}$ is independent of the Brownian motion $\tilde{B}$, and is Poisson with intensity $ \d s \boldsymbol{\Xi}( \d \mathbf{y})$, then the above stochastic equation (we know that $\int  X_s^{-\alpha}  \d s = \infty$ so $t$ is unbounded) is a representation of the L\'evy process $\xi$ except that we have compensated all jumps, i.e.~with L\'evy Khintchine formula given by 
    $$ \psi(\gamma) =  \left( \mathrm{\beta- \frac{\sigma^2}{2}} \right) \gamma + \frac{1}{2}\sigma^2 \gamma^2 +\int_{\R_*}  \Lambda_0(\d y_0) \big( \mathrm{e}^{\gamma u_0}-1 - \gamma(\mathrm{e}^{u_0}-1) \big).$$
    Using the expression of $\beta$ we recover exactly  \eqref{eq:levykhintchine}. It thus remains to prove the claim.  Using  \cite[Theorem II.4.8]{jacod2013limit}, it suffices to show that for any non-negative predictable function $H : \Omega \times \R_+ \times \mathcal{E} \to \R_+$, we have 
	\begin{equation*}
		\E \left( \int_0^\infty \int_{\mathcal{E}} H(s, \mathbf{y})  \mathcal{M}(\d s,\d \mathbf{y}) \right) = \E \left( \int_0^\infty \int_{\mathcal{E}} H(s,\mathbf{y}) \d s \boldsymbol\Xi( \d \mathbf{y}) \right).
	\end{equation*}
	By a straightforward calculation, 
	\begin{align*}
		&\E \left[ \int_0^\infty \int_{\mathcal{E}} H(t, \mathbf{y})  \mathcal{M}(\d t,\d \mathbf{y}) \right] \\
		&= \E \left[ \int_0^{\infty} \int_{\mathcal{E}} H \left( \int_0^s \mathrm{d}u X_u^{-\alpha}, h(X_{s-}, \mathbf{y})\right) \mathcal{N}( \d s \d \mathbf{y}) \right] \tag*{\text{by change of variables $t = \int_0^s \mathrm{d}u X_u^{-\alpha}$}} \\
        	&= \int_0^{\infty} \d s \mathbb{E} \left[ \mathbbm{1}_{s < z_\varepsilon} \int_{\mathcal{E}} \boldsymbol \Xi (\d \mathbf{y}) H \left( \int_0^s \mathrm{d}u X_u^{-\alpha}, h(X_{s-}, \mathbf{y})\right)  \right] \tag*{\text{$\mathcal{N}$ is Poisson with intensity $ \d s \boldsymbol\Xi ( \d \mathbf{y})$}} \\
            &=\int_0^{\infty} \d s \mathbb{E} \left[ \mathbbm{1}_{s < z_\varepsilon} \cdot X_{s-}^{-\alpha} \int_{\mathcal{E}} \boldsymbol \Xi (\d \mathbf{y}) H \left( \int_0^s \mathrm{d}u X_u^{-\alpha},  \mathbf{y}\right)  \right] \tag*{\text{quasi-preservation of $\boldsymbol \Xi$ \eqref{eq:quasipreservation}}} \\
        &=\int_0^{\infty} \d t  \int_{\mathcal{E}} \boldsymbol \Xi (\d \mathbf{y}) \E\left[H \left( t,  \mathbf{y}  \right)\right] \tag*{\text{change of variable $t = \int_0^s \mathrm{d}u X_u^{-\alpha}$.}} \\
	\end{align*}
    showing that $\mathcal{M}$ is a Poisson random measure associated to a Poisson point process on with intensity $\d s \boldsymbol{\Xi}(\d \mathbf{y})$. The independence between $\tilde{B}$ and $\mathcal{M}$ is immediate, by Theorem~II.6.3 of \cite{IkedaWatanabe1989stochastic}, since the compensator measure is non-random. We are left to prove that we can take the limit for $\varepsilon \to 0$. Since $z_\varepsilon$ is non-decreasing in $\varepsilon$, we denote $z_0 := \lim_{\varepsilon \to 0} z_\varepsilon$ its almost sure limit. By the distributional properties of the pssMp $X$, we have $X_{z_0}=0$ and  $ \int_0^{z_0} \frac{\d s}{X_s^\alpha}= \infty$ almost surely, so that the process $\xi$ is indeed a L\'evy process with characteristic quadruplet $(a, \sigma^2, \Lambda_0; \alpha)$ defined on the whole of $\mathbb{R}$ from which $X$ is built using the Lamperti transformation.
    \color{black}
    \end{proof}

\paragraph{Step 2: Extension to the reproduction, Markov and monotonicity properties.} In the rest of the proof, to highlight the dependance in the initial condition we shall write  $X^{(x)}$ for the solution to $(\ref{eq:SDE})$ with initial condition $x>0$, which is absorbed at $0$ at time $z^{(x)}$.  Let us now extend the construction to the reproduction process which is associated to the solution $X^{(x)}$. We pose
\begin{equation}\label{eq:repro_process}
	\eta^{(x)} := \sum_{(s,\mathbf{y}) \in \mathcal{N}} \sum_{i \geq 1} \mathbbm{1}_{\{ s \leq z^{(x)} \}} \delta_{ \big( s, X^{(x)}_{s-} \cdot G^{(i)}_{X^{(x)}_{s-}}( \mathbf{y}) \big)}.
\end{equation}
The proof of the above proposition (in particular the fact that $ \mathcal{M}$ is a Poisson random measure associated to a Poisson point process on $\R_+^2$ with intensity $\d s \boldsymbol{\Xi}(\d \mathbf{y})$) shows that the pair $(X^{(x)}, \eta^{(x)})$ actually has the law of the decoration-reproduction process under  $P_x$. 
Obviously, one can construct simultaneously the process $(X^{(x)},\eta^{(x)})$ for finitely or even countably many starting points $x$ simultaneously. The goal of Theorem \ref{thm:deco-repro-grow} is to construct simultaneously those solutions for all $x>0$ in a continuous manner, but before that let us state a couple of properties that we will need. We start with the expected monotonicity:
\begin{proposition}[Monotonicity] \label{prop:monotonicity} Fix $0< x'\leq x$, then we have 
\begin{equation}\label{eq:mono}
    ( X^{(x')}, \eta^{(x')}) \leq ( X^{(x)}, \eta^{(x)})\qquad a.s.
\end{equation}
in the sense that $z^{(x')}\leq z^{(x)}$ and for all $s \leq z^{(x')}$ we have $X_s^{(x')} \leq X_s^{(x)}$. Furthermore, the atoms of $\eta^{(x')}$ on $[0,z^{(x')})$ correspond to the atoms of $\eta^{(x)}$ on $[0,z^{(x')})$ and are coordinate-wise smaller. 
 \end{proposition}
 \begin{remark}[Strict monotonicity]  Under the additional assumption that the monotonicity condition \eqref{eq:monotonicity} is strict, i.e. the map $x \mapsto x \cdot  G_x(\mathbf{y})$ is strictly increasing in $x > 0$ for all $\mathbf{y}$ in the support of $\mathbf{\Xi}$, then under a mild assumption we expect to have the strict inequality,
 \begin{eqnarray} \label{eq:strict} X_s^{(x')} < X_s^{(x)}, \quad \mbox{ for all } 0 \leq s  < z^{(x')}. \end{eqnarray}
 See e.g. \cite{yamada1981strong} in the diffusion case, and \cite[Corollary~3.3]{xizhu2019jump} in presence of jumps. We did not require those assumptions to keep our results light. However, let us emphasize that even in the strict monotonic case we may have $z^{(x')} = z^{(x)}$ for $x'<x$ and so the above strict inequality which is allegedly valid for $s \ < z^{(x')}$ cannot be extended to all $s \leq z^{(x')}$.
\end{remark}

\begin{proof} 
    We use a proof which is a simple adaptation of \cite[Theorem~295,~page~292]{situ2005theory}. Both settings are not exactly similar, so we give this proof for completeness. We denote $\tilde{t} = t \wedge z_\varepsilon$ where $z_\varepsilon = \inf \{ t \geq 0: X_t^{(x)} \notin [\varepsilon, \varepsilon^{-1}] \text{ or } X_t^{(x')} \notin [\varepsilon, \varepsilon^{-1}] \}$. The key idea of Situ is to apply a Tanaka type formula \cite[Theorem 152, page 120]{situ2005theory}. If we denote $(\cdot)^+$ the positive part of a function, then
    \begin{align*}
        &\left( X^{(x')} - X^{(x)} \right)_{\tilde{t}}^+ = (x' - x)^+ + \int_0^{\tilde{t}} \mathbbm{1}_{X^{(x')}_s > X^{(x)}_s} \left( \beta(X^{(x')}_s) - \beta(X^{(x)}_s) \right) \d s \\
        &+ \int_0^{\tilde{t}} \mathbbm{1}_{X^{(x')}_s > X^{(x)}_s} \left( \sigma(X^{(x')}_s) - \sigma(X^{(x)}_s) \right) \d B_s + \int_0^{\tilde{t}} \int_{\mathcal{E}} \mathbbm{1}_{X^{(x')}_{s-} > X^{(x)}_{s-}} \left( h(X^{(x')}_{s-}, u) - h(X^{(x)}_{s-},u) \right) \tilde{\mathcal{N}}(\d s, \d u) \\
        &+ \int_0^{\tilde{t}} \int_{\mathcal{E}} \left( X^{(x')}_{s-} - X^{(x)}_{s-} + h(X^{(x')}_{s-}, u) - h(X^{(x)}_{s-},u) \right)^+ - \mathbbm{1}_{X^{(x')}_s > X^{(x)}_s} \left( X^{(x')}_{s-} - X^{(x)}_{s-} \right) \\
        & \hspace{15mm} - \mathbbm{1}_{X^{(x')}_s > X^{(x)}_s} \left( h(X^{(x')}_{s-}, u) - h(X^{(x)}_{s-},u) \right) \mathcal{N}(\d s, \d u).
    \end{align*}
    Notice that under the local Lipschitz assumption \eqref{ass:lipschitz}, the Brownian and compensated jump integrals are true martingales and therefore vanish when taking the expectation. Furthermore, the crucial observation is that the monotonicity condition \eqref{eq:monotonicity} implies 
    \[
    \left( X^{(x')}_{s-} - X^{(x)}_{s-} + h(X^{(x')}_{s-}, u) - h(X^{(x)}_{s-},u) \right)^+ = \mathbbm{1}_{X^{(x')}_{s-} > X^{(x)}_{s-}} \left( X^{(x')}_{s-} - X^{(x)}_{s-} + h(X^{(x')}_{s-}, u) - h(X^{(x)}_{s-},u) \right),
    \]
    which in turn implies that the terms in the last integral cancel out completely. Taking the expectation and using the local Lipschitz property of $\beta(\cdot)$ yields
    \begin{align*}
            &\E \left[ \left( X^{(x')} - X^{(x)} \right)_{\tilde{t}}^+ \right] \leq C_\varepsilon \E \left[ \int_0^{\tilde{t}} \mathbbm{1}_{X^{(x')}_s > X^{(x)}_s} \left| X^{(x')}_s - X^{(x)}_s \right| \d s \right] \\
            &= C_\varepsilon \E \left[ \int_0^{\tilde{t}} \left( X^{(x')} - X^{(x)} \right)_s^+ \d s \right] \leq C_\varepsilon \int_0^{\tilde{t}} \E \left[ \left(X^{(x')} - X^{(x)} \right)_s^+ \right] \d s
    \end{align*}
    By a standard Gronw\"all argument, we conclude that $\left( X^{(x')} - X^{(x)} \right)_{t \wedge z_\varepsilon}^+ = 0$ for all $t \geq 0$ almost surely. Now letting $\varepsilon \to 0$ yields the desired monotonicity result on $[0, z^{(x')})$. Furthermore, by construction all atoms of $\eta^{(x)}$ and $\eta^{(x')}$ are synchronized and \eqref{eq:monotonicity} and \eqref{eq:mono} combined yield the monotonicity between them. 
\end{proof}

We now discuss the strong Markov property of the solutions to \eqref{eq:SDE}. 

\begin{proposition}[Markov property]\label{prop:markov_pssmp} Fix $x >0$ and consider $(X^{(x)}, \eta^{(x)})$ the solution to \eqref{eq:SDE} and \eqref{eq:repro_process} starting from $x$. Then $X^{(x)}$ is a strong Markov process and it satisfies the strong Markov property, i.e. for every $\mathcal{F}_t(X, \eta)$-measurable stopping time $T$ with $T < \infty$ a.s., conditionally on $ \mathcal{F}_T$ the process $(X^{(x)}_{T+\cdot}, \eta^{(x)}_{T+\cdot})$ has the law of $(X,\eta)$ started from the value $X^{(x)}_{T}$, i.e. $P_{X_T^{(x)}}$. 
\end{proposition}
\begin{proof} 
This is a standard result concerning the strong Markov property for path-wise unique strong solutions to SDE's driven by Brownian motion and a Poisson random measure, see for example \cite[Theorem~6.4.6]{applebaum2004levy} . In particular, if $B, \mathcal{N}$ are the Brownian motion and Poisson random measure used in \eqref{eq:SDE}, the process $(X^{(x)}_{T+\cdot}, \eta^{(x)}_{T+\cdot})$ is a strong solution to the equation \eqref{eq:SDE} but driven by the Brownian motion $B_{T + \cdot}$ and by the Poisson measure  $\mathcal{N}_{T+ \cdot}$ in which we discarded the possible atom at $0$. In case when $T$ is predictable, there is a.s. no atom at $T$. The Markov property follows from that  on those objects.
\end{proof}

\paragraph{Step 3: Continuity of $x \mapsto (X^{(x)}, \eta^{(x)})$.}  Gathering-up the pieces, for all positive rationals $ q \in \mathbb{Q}_{>0}$ one can define simultaneously the processes $(X^{(q)}, \eta^{(q)})$ strong solutions to \eqref{eq:SDE} so that \eqref{eq:strict} holds. Our goal is to prove that $q \mapsto (X^{(q)}, \eta^{(q)})$ can be extended by continuity to all $x >0$. We start with the continuity of the first component. To be precise, we shall see the trajectories of the processes $X^{(x)}$ as c\`adl\`ag functions defined over a finite interval $[0,z^{(x)}]$ equipped with the standard Skorokhod topology. In particular, the continuity of $x \mapsto X^{(x)}$ implies the continuity of $x \mapsto z^{(x)}$. The previous point is actually the main issue: Indeed,  we will show below that a classical continuity argument based on the Lipschitz hypothesis can be applied when the processes are within a compact interval of $(0,\infty)$. The issue is that for some (exceptional) values of $x$, the process ${X}_\cdot^{(x)}$ may not be  absorbed at its first hitting time of $0$, yielding a non-continuity of $x \mapsto z^{(x)}$ at this point, see Figure \ref{fig:continuite-restart}. This behavior will be ruled out in Proposition \ref{prop:continuityz} below using Lemma \ref{lem:alpha-vs-cumulant} in a crucial way. \medskip

We now restrict to $x \in [0,1]$ to fix ideas, and start with the classical continuity of flow of solutions to Lipschitz SDE: Fix $ \varepsilon>0$ and  focus on the decoration processes restricted to the interval $[ \varepsilon, \varepsilon^{-1}]$. For $q \in \mathbb{Q}\in [0,1]$ we denote by 
$$ z_{\varepsilon}^{(q)} = \inf\{s \geq 0 : {X}_s^{(q)} \notin [ \varepsilon, \varepsilon^{-1}]\}.$$ With these notations at hands, by \cite[Proposition 6.6.2]{applebaum2004levy}, we have for all $q,q' \in \mathbb{Q}_{>0}$ and any $t_0>0$
\begin{eqnarray} \label{eq:kolmo1} \mathbb{E}\left[ \sup_{0 \leq s < z_{\varepsilon}^{(q)} \wedge z_{\varepsilon}^{(q')} \wedge t_0} \left( X_s^{(q)}-X_s^{(q')}\right)^2 \right] \leq K_{ \varepsilon, t_0} \cdot (q-q')^2, \end{eqnarray}
since $X^{(q)}$ and $X^{(q')}$ are solutions to \eqref{eq:SDE} with Lipschitz properties. {Note that Applebaum requires a Lipschitz property for all $p \geq 2$, yielding a result that equally holds for all $p \geq 2$. However, a direct calculation using Itô's formula shows that the above result holds if we only suppose $p = 2$.} Since $z^{(q)} \leq z^{(1)} < \infty$ for all $q \in \mathbb{Q}\cap[0,1]$ by Proposition \ref{prop:monotonicity}, by  Kolmogorov's continuity theorem, we deduce that the pseudo-distance \begin{eqnarray} d(X^{(q)},X^{(q')}) = \sup_{0 \leq t < z_{\varepsilon}^{(q)} \wedge z_{\varepsilon}^{(q')}} \left( X_t^{(q)}-X_t^{(q')}\right) \label{eq:kolmo2} \end{eqnarray} is  almost surely H\"older continuous when restricted to  $\mathbb{Q} \cap [0,1]$. However, this does not guarantee (by letting $ \varepsilon \to 0$) the continuity $q \mapsto {X}^{(q)}$ for the supremum norm nor of $q \mapsto z^{(q)}$ because of the phenomenon described above.

\begin{proposition}[Everything's Gonna Be Alright] \label{prop:continuityz} Almost surely we have $$ \lim_{ \varepsilon \to 0} \left(\sup_{q \in \mathbb{Q} \cap [0,1]}  z^{(q)}- z^{(q)}_{ \varepsilon}\right) =0.$$
\end{proposition}

\begin{proof} If the above display does not hold almost surely, this implies that there exists $\delta,t >0$ rationals satisfying $\delta < t < t+\delta <\delta^{-1}$ together with a sequence of rationals $(q_n)_{n\geq 1}$ with $q_n \leq 1$ and $( \varepsilon_n)_{n\geq 1}$ with $0<\varepsilon_n \to 0$ such that for all $n \geq 1$ large enough,  \begin{align} \begin{array}{l}
z^{(q_n)}_{\varepsilon_n} < t\\
z^{(q_n)} > t+\delta \\ \end{array}   
& \quad \mbox{with probability at least }\delta. \label{eq:probadelta} \end{align}

We shall detect such a discontinuity using an auxiliary process $C^{( \varepsilon)}$ constructed by restarting the SDE \eqref{eq:SDE} when becoming small. Specifically, let  $\varepsilon > 0$ rational. We define a new process $C_\cdot^{(\varepsilon)}$ which coincides with $X^{(\varepsilon)}$ up to the first time $T_1^{( \varepsilon)}$ it drops below $ \varepsilon/2$. We then ``restart" from the value $\varepsilon$ at time $T_1^{( \varepsilon)}$ using Proposition \ref{prop:markov_pssmp}. Formally this is done as follows: Let $B, \mathcal{N}$ the Brownian motion and Poisson random measure used in \eqref{eq:SDE} for the definition of $X^{( \varepsilon)}$. By the strong Markov property, $B_{z^{(\varepsilon)} + \cdot}$ and the measure $\mathcal{N}_{z^{(\varepsilon)}+ \cdot}$ to which we remove the possible atom at $T_1^{(\varepsilon)}$ are independent of $\mathcal{F}_{T_1^{(\varepsilon)}}(B, \mathcal{N})$ and have the same law as $(B, \mathcal{N})$. We can then define $C_{ T_1^{(\varepsilon)} + \cdot }$ as the pathwise unique strong solution to 
\[
C_{T_1^{(\varepsilon)} + t}^{(\varepsilon)} = \varepsilon + \int_{T_1^{(\varepsilon)}}^t \beta(C_s^{(\varepsilon)}) ds + \int_{T_1^{(\varepsilon)}}^t \sigma(C_s^{(\varepsilon)}) dB_s + \int_{T_1^{(\varepsilon),+}}^t \int_{ \mathcal{E}} g(C_{s-}^{(\varepsilon)}, \mathbf{y}) \tilde{\mathcal{N}}(\d s, \d \mathbf{y})
\] in the notation of \eqref{eq:SDE_rewritten}, until it drops below $ \varepsilon/2$. We can then iterate the construction and define the process $C^{( \varepsilon)}$ on the whole of $\mathbb{R}_+$ with restart times $$ 0< T_1^{(  \varepsilon)} < T_2^{( \varepsilon)}< T_{3}^{( \varepsilon)} < T_{4}^{( \varepsilon)}...$$
so that $(T_{i+1}^{(\varepsilon)}-T_{i}^{( \varepsilon)})_{i\geq 0}$ is a sequence of i.i.d.~r.v.~of law $ \varepsilon^\alpha \cdot \tau$ with $\tau=\inf\{s \geq 0 : X^{( 1)}_s \leq 1/2\}$, with the convention $T_{0}^{( \varepsilon)}=0$. In particular for any $ \varepsilon>0$ we have $T_i^{( \varepsilon)} \to \infty$ a.s.~as $i\to \infty$. We then consider the first restart time after $ t+ \frac{\delta}{3}$, 
$$ I^{( \varepsilon)} = \inf \left\{ i \geq 1 : T_i^{(\varepsilon)} \geq t + \frac{\delta}{3}\right\},$$and after this restart time, we just wait for the process $C^{( \varepsilon)}$ to be absorbed at $0$ at time $T_{I^{( \varepsilon)}}^{( \varepsilon)} + Z^{(\varepsilon)}$ instead of being reset when dropping below $ \varepsilon/2$.  By the Markov property the variable $Z^{( \varepsilon)}$  is independent of $(T_i^{(\varepsilon)} : i \leq I^{(\varepsilon)})$ and has law $ \varepsilon^\alpha \cdot \tau^+$ where $\tau^+ = \inf\{s \geq 0 : X^{( 1)}_s =0\}$.

The construction of this process can be made simultaneously with all $X^{(q)}$ for $q$ rationals  such that the monotonicity of Proposition \ref{prop:monotonicity} holds; in particular, as long as $X^{(q)}\geq  \varepsilon$ then $X^{(q)} \geq C^{( \varepsilon)}$ but if $X^{(q)}$ drops below $ \varepsilon/2$ before time $t$, then $C^{( \varepsilon)}$ stays above $X^{(q)}$ from this time on, and in particular $C^{( \varepsilon)} \geq X^{(q)}$ after time $t$. 

We now use this process to prove the proposition. The crux is to notice, see Figure \ref{fig:continuite-restart}, that on the event involved in \eqref{eq:probadelta}, since for $n$ large enough, $X^{(q_n)}$ drops below $ \varepsilon/2$ before time $t$, we must have $C^{( \varepsilon)} \geq X^{(q_n)}$ over $(t, \infty)$ and in particular the last two excursions of $C^{( \varepsilon)}$ must be large. We distinguish the cases $T_{I^{( \varepsilon)}}^{(\varepsilon)} \geq t+ \frac{2\delta}{3}$ or $T_{I^{( \varepsilon)}}^{(\varepsilon)} < t+ \frac{2\delta}{3}$. Since $T_{I^{( \varepsilon)}}^{(\varepsilon)} +Z^{( \varepsilon)} \geq z^{(q_n)} \geq t + \delta$, in both cases we must have \begin{eqnarray}
    \label{eq:firstmoment} \max\left(Z^{(\varepsilon)} ;  T_{I^{( \varepsilon)}}^{( \varepsilon)} - T_{I^{( \varepsilon)}-1}^{( \varepsilon)}\right) \geq \frac{\delta }{3}.\end{eqnarray}

\begin{figure}
    \centering
    \includegraphics[width=0.75\linewidth]{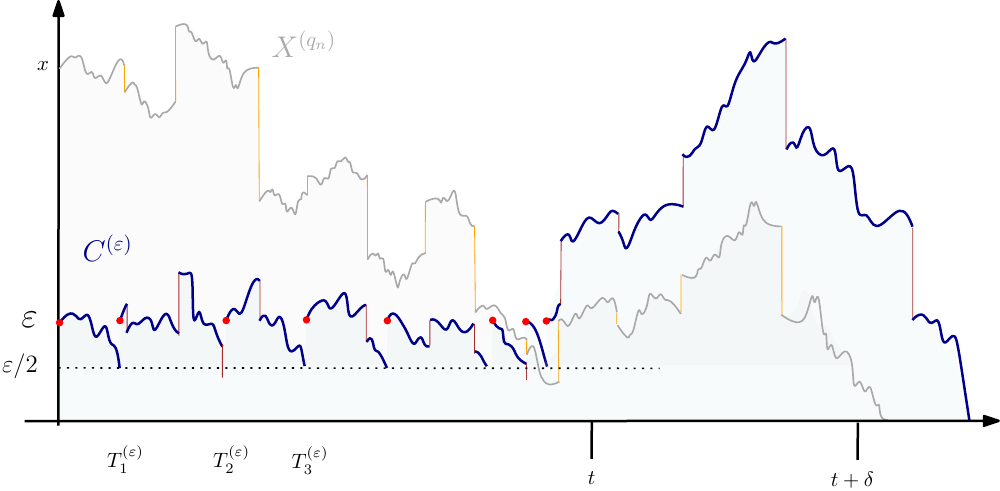}
    \caption{ Detecting a non-continuity in the absorbing time using the process $C^{( \varepsilon)}$.\label{fig:continuite-restart}}
\end{figure}

\noindent Since $Z^{( \varepsilon)} = \varepsilon^\alpha \cdot \tau^+$, we have  $\mathbb{P}( Z^{( \varepsilon)} \geq \delta/3) \to 0$ as $ \varepsilon \to 0$. To argue similarly that $$\mathbb{P}( T_{I^{( \varepsilon)}}^{( \varepsilon)} - T_{I^{( \varepsilon)}-1}^{( \varepsilon)} \geq \delta/3) \to 0, \quad \mbox{ as } \varepsilon\to 0$$ we need a little more since the index $I^{( \varepsilon)}$ may not be typical. Recall first from Lemma \ref{lem:alpha-vs-cumulant} that under the growing hypothesis we have $\alpha \leq \min \{ \gamma : \kappa( \gamma) < \infty\}$ and in particular $\alpha \leq \omega_-$ so that $\frac{\omega_-}{\alpha}\geq 1$. Furthermore, the L\'evy--Khintchine exponent satisfy $\psi(\omega_-) < \kappa(\omega_-) = 0$ by \eqref{def:subertoin2024selfitical}. It then follows from \cite[Lemma 9.1]{bertoin2024self} and the trivial stochastic bound $\tau\leq \tau^+$ that $\mathbb{E}[\tau^{\omega_-/\alpha}]< \infty$ and in particular it has a finite expectation. By the uniform integrability improvement of Markov inequality, we have 

$$ \mathbb{P}( \exists i \leq K : T_{i+1}^{( \varepsilon)}-T_i^{( \varepsilon)} \geq \delta) \leq K \cdot \mathbb{P}(  \tau \geq \delta / \varepsilon^{\alpha}) \underset{ \mathrm{Markov}}{\leq}  o( \varepsilon)\cdot K \cdot \frac{\varepsilon^\alpha}{\delta} \cdot \mathbb{E}[\tau],$$ where $o(\cdot)$ is a function tending to $0$ when $ \varepsilon\to 0$.  On the other hand, by the weak-law of large numbers as soon as $K \geq {2 \delta^{-1}}/({ \varepsilon^{\alpha} \mathbb{E}[\tau]})$ we have $ T_K^{ ( \varepsilon)} > (t+\delta/3)$ with high probability as $  \varepsilon\to 0$, so that $I^{( \varepsilon)} \leq {2 \delta^{-1}}/({ \varepsilon^{\alpha} \mathbb{E}[\tau]})$ whp. Combining those two remarks shows that indeed \eqref{eq:firstmoment} is unlikely as $ \varepsilon \to 0$. This disproves \eqref{eq:probadelta} and finally shows the statement of the proposition. \end{proof}

We can now use Proposition \ref{prop:continuityz} and \eqref{eq:kolmo1} to prove that $q \mapsto z^{(q)}$ and $q \mapsto X^{(q)}$ are almost surely uniformly continuous (in the second case, for the Skorokhod topology), the details are left to the reader. The continuous extended processes $X^{(x)}$, which are then c\`adl\`ag functions defined for all $x \in [0,1]$ are equipped with their reproduction processes $\eta^{(x)}$ defined by
 \begin{eqnarray}
 \label{eq:etax} {\eta}^{(x)} = \sum_{(s, \mathbf{y}) \in \mathcal{N} } \sum_{i \geq 1}\mathbbm{1}_{s \leq {z}^{(x)} } \delta_{\big( s, {X}_{s-}^{(x)} \cdot  G_{{X}_{s-}^{(x)}}^{(i)}( \mathbf{y})\big)}.\end{eqnarray}

 \paragraph{Step 4: Checking the properties.} We shall now prove that the above coupled processes $({X}^{(x)},{\eta}^{(x)})$ have the desired properties listed in Theorem \ref{thm:deco-repro-grow}.

\color{black}

\begin{itemize}
\item \textbf{Law}. Recall that $(X^{(q)}, \eta^{(q)})$ has law $P_q$ for $q \in \mathbb{Q}\cap [0,1]$, by continuity, the same property holds true for all $x \in [0,1]$.
\item \textbf{Continuity of reproduction process.} The continuity of $x \mapsto \eta^{(x)}$ defined in \eqref{eq:etax} follows immediately from the continuity of $x \mapsto X^{(x)}$ and the continuity of $x \mapsto x \cdot G_x(\mathbf{y})$. The synchronicity and monotonicity of the jumps is still guaranteed by the continuity. 
\item \textbf{Monotonicity} of $x \mapsto (X^{(x)}, \eta^{(x)})$ follows from  Propositions~\ref{prop:monotonicity} and \eqref{eq:monotonicity}. 
\item \textbf{Measurability.} Fix $x >0$. Since $X^{(x)}$ is a strong solution to \eqref{eq:SDE}, we know that it is measurable with respect to the filtration generated by the Brownian motion and the Poisson random measure. Given a pair $\left( X^{(x)}, \eta^{(x)} \right)$ until time $t \leq z^{(x)}$, we can recover $B_t$ and $\mathcal{N}_t$, i.e. show that for all $t \leq z^{(x)}$,
	\begin{equation*}
		\mathcal{F}_t \left( X^{(x)}, \eta^{(x)} \right) = \mathcal{F}_t(B, \mathcal{N}),
	\end{equation*}
    which is sufficient to recover any other strong solution $\left( X^{(x')}, \eta^{(x')} \right)$ up to time $t$. \begin{itemize}
        \item 
    The atoms of $\mathcal{N}$ are recovered at any jump time $s < z^{(x)}$ of $X^{(x)}$ by applying the function inverse of $ G_{X_{s-}^{(x)}}$ on $\frac{1}{X_{s-}^{(x)}} \cdot (X_s^{(x)} ; (y_i^{(x)} : i \geq 1)) $ where $y_i^{(x)}$ are the atoms of $\eta^{(x)}$ at time $s$ ranked in decreasing order. This is possible since $G_x$ is bijective on the support of $\boldsymbol{\Xi}$ except for possibly a set of zero $\boldsymbol{\Xi}$-measure, see Definition~\ref{def:growing}. Since the process $X^{(x)}$ is absorbed continuously at $0$, the stopping time $z^{(x)}$ is predictable with respect to $\mathcal{F}_t \left( X^{(x)}, \eta^{(x)} \right)$ and it follows that almost surely $ \mathcal{N}$ has no atom at $z^{(x)}$, so nothing to recover there. 
    
    \item In order to find the Brownian motion, we subtract the jump and drift integrals from $X^{(x)}$, which are both $\mathcal{F}_t \left( X^{(x)}, \eta^{(x)} \right)$-measurable functionals. Since the previsible compensator of the jump measure $\mathcal{N}$ is deterministic ($\d s \boldsymbol{\Xi}(\d \mathbf{y})$), we can uniquely recover the $\mathcal{F}_t \left( X^{(x)}, \eta^{(x)} \right)$-measurable process
    \[
    M^c_t := X_t^{(x)} - x - \int_0^t \beta \left( X_{s-}^{(x)} \right) ds - \int_0^t g \left( X_{s-}^{(x)}, \mathbf{u} \right) \tilde{\mathcal{N}}(\d s, \d\mathbf{u}) = \int_0^t \sigma \left(X_s^{(x)} \right) \d B_s
    \]
     which is the continuous local martingale part of the process. Since the diffusion coefficient is continuous and does not vanish, i.e. $\sigma(\cdot) \neq 0$ for $X^{(x)}$ on any compact $[\varepsilon, \varepsilon^{-1}]$, it is a simple matter using Dubins--Schwarz theorem to recover the Brownian motion $B$ from the process $M^{c}$.
    \end{itemize}
\item \textbf{Markov property.} Fix $0 < x' \leq x$. Given $(X^{(x')}, \eta^{(x')})$ by the previous point, one can recover the Brownian motion and the Poisson measure, hence the process $(X^{(x)}, \eta^{(x)})$ up to time $z^{(x')}$. The law of the future of that process is then $P_{X^{(x)}_{z^{(x')}}}$ by Proposition \ref{prop:markov_pssmp}.
\end{itemize}
\end{proof}
\noindent We briefly discuss the pure jump case which is much more explicit.

\begin{remark}[Pure jump case] In the pure-jump case where $\Xi_0$ integrates $(y_0 - 1)$, there is no need to compensate the jump integral in \eqref{eq:SDE} and the pssMp $X^{(x)}$ can be obtained from one another by simply adapting their jumps. More precisely, given the process $(X^{(x)}, \eta^{(x)})$ with atoms $(s, \mathbf{y}^{(x)})$ for $s < z^{(x)}$, one can transform the jumps using
\begin{equation}\label{eq:transform_jumps}
    \tilde{G}(X_{s-}^{(x)}, X_s^{(x)}, X_{s-}^{(x')}) := G_{X^{(x')}_{s-}} \left( \frac{X_s^{(x)}}{X_{s-}^{(x)}},  \frac{\mathbf{y}^{(x)}}{X_{s-}^{(x)}} \right) = G_{{X_{s-}^{(x')} / X^{(x)}_{s-}}} (G_{X_{s-}^{(x)}}(\mathbf{y})) =  G_{X_{s-}^{(x')}}(\mathbf{y}).
\end{equation}
The process $(X^{(x')}, \eta^{(x')})$ for $0 < x'\leq x$ is the pure jump decoration-reproduction process satisfying
$$ X^{(x')}_t = x' - \sum_{\substack{(s, \mathbf{y}^{(x)}) \in \eta^{(x)} \\ s \leq t}} X_{s-}^{(x')} \cdot \tilde{G}^{(0)}(X_{s-}^{(x)}, X_s^{(x)}, X_{s-}^{(x')}),$$
$$\eta^{(x')} = \sum_{\substack{(s, \mathbf{y}^{(x)}) \in \eta^{(x)} \\ s \leq z^{(x')}}} \sum_{i \geq 1} \delta_{\left( s, X_{s-}^{(x')} \cdot \tilde{G}^{(i)}(X_{s-}^{(x)}, X_s^{(x)}, X_{s-}^{(x')}) \right)} $$

By \eqref{eq:transform_jumps}, the process $(X^{(x')}, \eta^{(x')})$ has the right law $P_{x'}$ and can be obtained by applying the growing function $G_{x'/x}$ directly on the jumps of $X^{(x)}$ and on the atoms of $\eta^{(x)}$, so that $(X^{(x')}, \eta^{(x')})$ is measurable explicitly with respect to $(X^{(x)}, \eta^{(x)})$. 
\end{remark}

\section{Construction of a coupled family of ssMts}\label{sec:coupling_trees}
Building on Theorem \ref{thm:deco-repro-grow}, we construct in this section the family of random trees presented in our main result Theorem \ref{thm:main}, which we restate in a more precise way in Theorem~\ref{thm:mainprecise}. The high-level idea of the proof is just to perform the coupling of Theorem \ref{thm:deco-repro-grow} branch by branch by following the construction of   ssMt in \cite{bertoin2024self}. Let us first recall the necessary background, in particular the definition of the space  of equivalence class of decorated trees $\mathbb{T}$ and of its Gromov-hypograph metric $d_{\mathbb{H}}$.

\subsection{Background on construction of ssMts}
\label{sec:background}
\paragraph{Decorated trees.}
We first recall the definition of decorated trees and the metric used to compare them. We refer to Section 1.4 of \cite{bertoin2024self} for details (in our case, we forget about the measure component, see the discussion after Corollary 1.18 in \cite{bertoin2024self}). Let  $K$ be a compact metric subspace of some metric space $(Y,d_Y)$ and $g : K \to \mathbb{R}_+$ an upper semi-continuous (usc) function. We consider the hypograph of the function $g$
\begin{equation*}
	\text{Hyp}(g) := \{ (x,r) : x \in K \text{ and } 0 \leq r \leq g(x) \}.
\end{equation*}
Since $g$ is supposed to be usc, this space is a  compact subspace of $Y \times \R_+$ which is naturally equipped with the distance $
d_{Y \times \R_+} ((y,r), (y',r')) := d_Y(y,y') \vee |r - r'|$. Two hypographs $\text{Hyp}(g)$ and $\text{Hyp}(g')$ with domains $K,K' \subset Y$ and decorations $g,g'$ respectively can be compared using the Hausdorff distance:
\[
d_{\text{Hyp}}(g, g') := d_{\text{Haus}}(\text{Hyp}(g), \text{Hyp}(g')).
\]
When the hypographs come with distinguished points $\rho, \rho' \in K$, we add $d_Y(\rho,\rho')$ in the above formula to compare them. By Proposition 1.14 of \cite{bertoin2024self}, the space $(\HH(Y), d_{ \mathrm{Hyp}})$ of all $(K,g)$ with $K \subset Y$ and $g$ as above is Polish whenever $(Y, d_Y)$ is. If the base domain $K$ is a real tree, then $(K,g,\rho)$ is called a \textbf{decorated tree}. The above definitions are then ``Gromovized" when objects are not \textit{a priori} living in the same ambient space by taking infimum under common isometric embeddings and modding out by isometries: If $\mathtt{K}=(K,d_K,\rho,g)$ and $\mathtt{K'}=(K',d_{K'},\rho,g')$ are two rooted compact metric space decorated by an usc function, the Gromov-hypograph distance between them is 
 \begin{equation}\label{def:d:H}
\d_{\mathbb{H}}\left( \mathtt{K},\mathtt{K}'\right) \coloneq 
  \inf_{Y,\phi,\psi}   \d_{\mathrm{Hyp}}\big( \phi(\mathtt{K}), \psi(\mathtt{K}')\big),
  \end{equation}
where again the infimum is taken  over all the Polish spaces $(Y,d_Y)$ and all the isometric embeddings $\phi: (K, d_K) \hookrightarrow (Y,d_Y)$ and $\psi: (K', d_{K'})\hookrightarrow  (Y,d_Y)$ and where the image $\phi(K)$ is equipped with the distinguished point $\phi(\rho)$ and decorated by $g\circ \phi^{-1}$ (and similarly for $K'$). By \cite[Theorem 1.16]{bertoin2024self}, this equips the set $\mathbb{T}$ of isometry classes of compact rooted decorated trees with a distance which makes it Polish. It is common in the literature to identify a decorated real tree $\mathtt{T}$ with its equivalence class, however to highlight some subtle details below, we shall keep tract in these pages whether or not we are dealing with an equivalence class and denote by $\underline{\mathtt{T}} \in \mathbb{T}$ the equivalence class of a decorated tree $\mathtt{T}$. The reader may forget this distinction in a first reading. For technical reasons that will become clear below, we shall also deal with the space $\mathbb{T}^{\bullet}$ of equivalence classes of decorated compact trees given with a sequence (indexed by $\mathbb{U}$) of points. This space is equipped with a distance similar to \eqref{def:d:H} by also comparing the distinguished points. See \cite[Section 1.4]{bertoin2024self} for details.

\paragraph{Building decorated trees by gluing branches.} A convenient way to construct decorated tree is to obtain them by gluing branches. This formalism is developed in Chapters 1 \& 2 of \cite{bertoin2024self}, and let us recall the main ingredients here. We are given ``building blocks"  $(f_u)_{u \in \U}$ and $\big((t_{u i})_{i \in \mathbb{Z}_{\geq 0}} : u \in \mathbb{U}\big)$, which are respectively c\`adl\`ag functions $f_u : [0,z_u] \to \mathbb{R}_+$ and points $t_{u i} \in [0,z_u]$ which are indexed by  the Ulam tree $\U = \bigcup_{k\geq 0} ( \mathbb{Z}_{\geq 0})^k$ with the convention $( \mathbb{Z}_{\geq 0})^0 = \{\varnothing\}$. Intuitively, for each $u \in \mathbb{U}$, the function\footnote{actually, the usc version of it, see \cite[Eq. (1.4)]{bertoin2024self}} $f_u$ defines a decorated branch over the segment $[0,z_u]$ and the branch with index $u i$ is glued onto the branch with index $u$ at the point $t_{ui}$. See Figure \ref{fig:gluing} below. The tree $T$ is then obtained first by quotienting 
$$ T^\sqcup := \bigcup_{u \in \mathbb{U}} \{ u, [0,z_u]\},$$
by the gluing $\{ui,0\} \sim \{u,t_{ui}\}$, and then taking the completion. The decoration function $g$ is then obtained by gluing all functions $f_u$ accordingly. Theorem 1.6 in \cite{bertoin2024self} provides a sufficient condition for the resulting tree to be compact. In this case we denote by $\mathtt{T}=(T,d_T,\rho_T,g)$ the resulting decorated tree where $\rho_T$ is the image of the point $(0,0)$. Actually, the above construction enables us to equip $\mathtt{T}$ with a family of points $(p(u) : u \in \mathbb{U})$ which are the projections of $(u, z_u) \in T^\sqcup$ for non fictitious individuals $u$ in the above construction (if $u$ is fictitious we can declare $p(u) = \rho_T$ by convention). We then use the shorthand notation 
$$ \mathtt{T}^\bullet = ( \mathtt{T}, (p(u) : u \in \mathbb{U})).$$ The point is that if one is then given the equivalence class $\underline{\mathtt{T}}^\bullet \in \mathbb{T}^{\bullet}$, then one is able to reconstruct the building blocks, that is the arrow 
$$ \big(f_u, (t_{ui})_{i \geq 0}\big)_{u \in \mathbb{U}} \to  \underline{\mathtt{T}}^\bullet \in \mathbb{T}^{\bullet} \quad \mbox{ is injective}.$$

\begin{figure}[h!]
    \centering
    \includegraphics[width=0.5\linewidth]{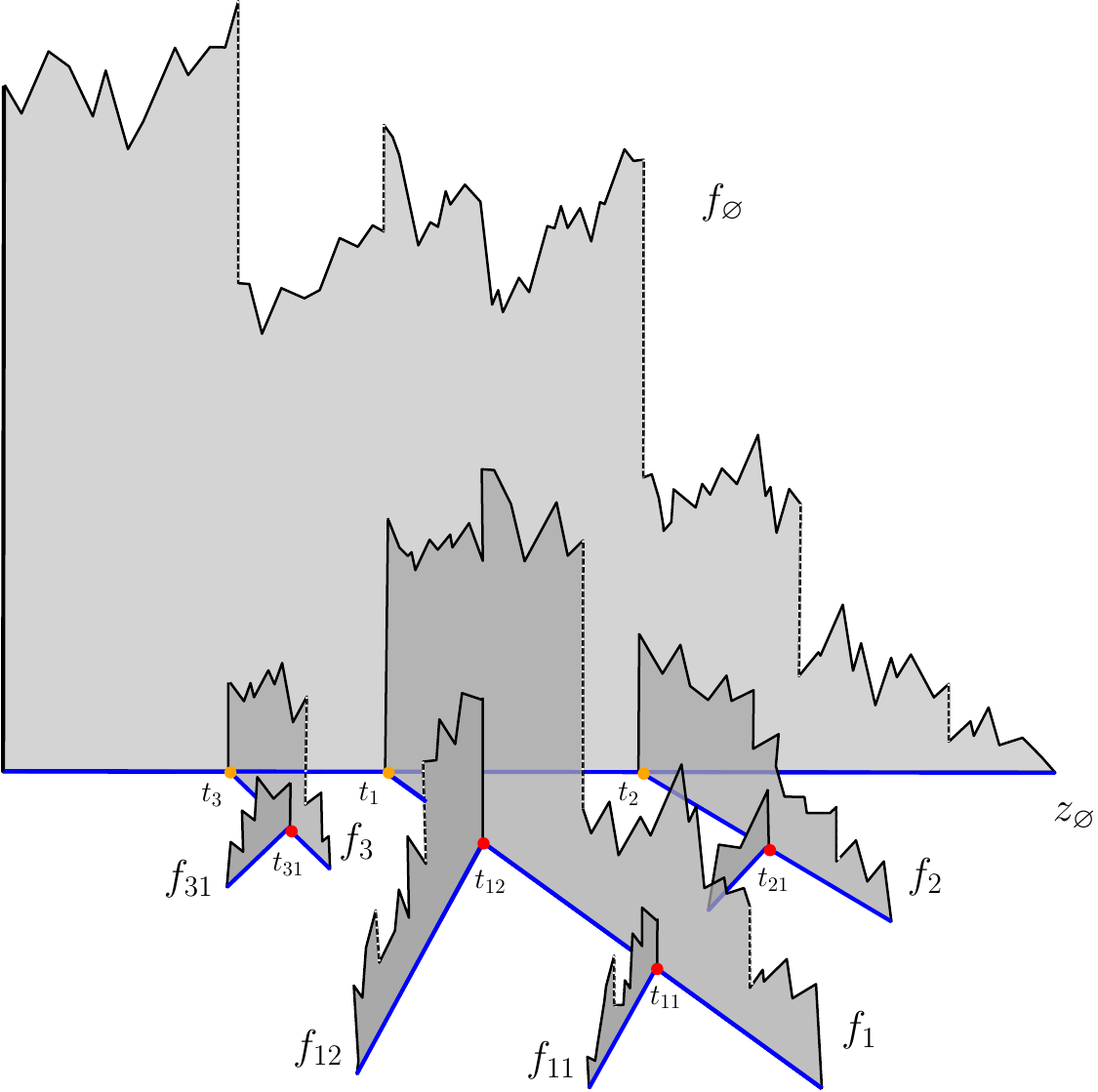}
    \caption{Figure from \cite{bertoin2024self} illustrating the gluing procedure to get a decorated real tree from a family of building blocks indexed by Ulam's tree. \label{fig:gluing}}
    
\end{figure}

In the random setting of the ssMt, the building blocks are defined recursively using the decoration-reproduction laws $(P_x : x >0)$ --introduced in Section \ref{sec:deco-repro}-- together with a labeling $\mathcal{X} : \mathbb{U}\to \mathbb{R}_{>0}$ of Ulam's tree. More precisely, under $\mathbb{P}_x$, we set $\mathcal{X}(\varnothing)=x$ and  let $(X,\eta)$ be a decoration-reproduction process with law $P_x$. The first building block is then $f_\varnothing = X$. We then enumerate the atoms of $\eta$ in \textbf{co-lexicographical order}, that is, largest atoms first and in case of ties, we use the time coordinate. For simplicity, we shall suppose that we have a.s. no ties during this enumeration\footnote{We leave the reader check that our results adapt, with a few contortions, to the general case.}. If  the atom $\delta_{(s_i,x_i)}$ is the $i$th atom in this enumeration, then we put $t_{\varnothing i} =s_i$ and we denote by $\mathcal{X}(\varnothing i) = x_i$. Conditionally on this first branch, we then sample for each branch associated to the vertices $1,2, ...$ independent decoration-reproduction processes of law $P_{\mathcal{X}(1)}, P_{\mathcal{X}(2)}, ...$ and iterate the procedure recursively to produce building blocks $(f_u,(t_{u i})_{i \geq 0})_{u \in \mathbb{U}}$.  It is then proved in \cite[Proposition 2.10 and Proposition 2.17]{bertoin2024self} that under the assumption \eqref{def:subertoin2024selfitical}, the gluing of those building blocks indeed yield to a compact decorated tree $\mathtt{T}$, under the law $\mathbb{P}_x$, whose equivalence class $\underline{\mathtt{T}}$ \footnote{There is a subtle difference between $\mathbb{P}_x$ which is the law of the building blocks and $\mathbb{Q}_x$ which is the law of the resulting decorated tree seen up to isometry. In particular, under $\mathbb{P}_x$ the gluing of the building blocks gives a concrete  decorated tree $\mathtt{T}$.} has law $\mathbb{Q}_x$. We shall also denote $\mathbb{Q}_x^\bullet$ the law of $\underline{\mathtt{T}}^\bullet \in \mathbb{T}^{\bullet}$ endowed with the sequence  of extremities of the branches in the construction.  Recall then from \cite[Chapter 5]{bertoin2024self} that several different building blocks whose laws are prescribed by different characteristics $( \mathrm{a},\sigma^2, \boldsymbol\Lambda; \alpha)$ may yield after gluing to the same law of equivalence class $\underline{\mathtt{T}}\in \mathbb{T}$, that is $(\mathrm{a}, \sigma^2, \boldsymbol\Lambda ; \alpha) \to (\mathbb{Q}_x)_{x \geq 0}$ is not generally injective. However the law of the points $(p(u) \in \mathtt{T} : u \in \mathbb{U})$ is different for different class of bifurcators, that is $(\mathrm{a}, \sigma^2, \boldsymbol\Lambda ; \alpha) \to (\mathbb{Q}^\bullet_x)_{x \geq 0}$ is injective.

\begin{remark}[Locally largest] Among all bifurcators of a given ssMt, the locally largest one is singled out. It consists in the splitting rule which is supported by $\{ y_0 \geq y_1 \geq y_2 ... \geq 0\}$, or in other words where individuals corresponds to following the lineage of the locally largest particle at each splitting. In this case, if there are no ties, the points $(p(u) : u \in \mathbb{U})$ can be recovered from the decorated tree $\underline{\mathtt{T}}$ only.  See \cite[Remark 1.20]{bertoin2024self} for details. 
\end{remark}

\subsection{Construction of decorated trees}
We can now state a more precise version of Theorem \ref{thm:main}:
\begin{theorem} \label{thm:mainprecise} Let $( \mathrm{a}, \sigma^2, \boldsymbol\Xi ; \alpha)$ a characteristic quadruplet satisfying \eqref{def:subertoin2024selfitical} and denote by $ (\mathbb{Q}_x : x >0)$ and $ (\mathbb{Q}^\bullet_x : x >0)$ the associated family of laws on $\mathbb{T}$ and $\mathbb{T}^\bullet$ respectively. If $(\boldsymbol\Xi ; \alpha)$ is $G$-growing in the sense of Definition \ref{def:growing} then  we can construct, on a common probability space $( \Omega, \mathcal{F}, \mathbb{P})$, a family of equivalence class of random decorated trees $ \underline{\mathtt{T}}^\bullet_x \in \mathbb{T}^\bullet$ endowed with a sequence of points, satisfying the following properties  $\PP$-almost surely: 
\begin{enumerate}[label=(\roman*)]
	\item \textbf{Law}. For all $x > 0$, the equivalence class $\underline{\mathtt{T}}^\bullet_x$ has law $\mathbb{Q}_x^\bullet$.
    \item \textbf{Continuity}. The process $x \mapsto \underline{\mathtt{T}}_x$ is continuous for the Gromov--hypograph distance $d_{\HH}$. 
    \item \textbf{Monotonicity}. The process $x \mapsto \underline{\mathtt{T}}_x$ is non-decreasing.
    \item \textbf{Measurability backward}. For any $0 < x' \leq x$, the r.v.  $\underline{\mathtt{T}}_{x'}^\bullet$ is a measurable function of $\underline{\mathtt{T}}_x^\bullet$. 
    \item \textbf{Markov forward}. For any $\, 0 < x' \leq x$, conditionally on $\underline{\mathtt{T}}_{x'}^\bullet=(T_{x'},d_{T_{x'}},\rho_{x'},g_{x'} , (p_{x'}(u))_{u \in \mathbb{U}})$, the tree $\underline{\mathtt{T}}^\bullet_x=(T_{x},d_{T_{x}},\rho_{x},g_{x},(p_{x}(u))_{u \in \mathbb{U}})$ is obtained as follows:
    \begin{itemize}
        \item Update measurably the decoration $g_{x'}$ on $\underline{\mathtt{T}}_{x'}^\bullet$ to $g_x$. In particular, this defines the weights $w_{x' \to x}(u) = g_x(p_{x'}(u))$.
        \item Graft independent ssMt's of law $\mathbb{Q}^\bullet_{w_{x' \to x}(u)}$ on $p_{x'}(u)$ for each $u\in \mathbb{U}$.
        \item Re-index the branches in co-lexicographical order and name their extremities $(p_{x}(u))_{u \in \mathbb{U}}$.
    \end{itemize}
\end{enumerate}
\end{theorem}

We say above that two equivalence classes of decorated trees satisfy $\underline{\mathtt{T}} \subset \underline{\mathtt{T}}'$ if they admit representatives $\mathtt{T}=(T,d_T,\rho_T,g)$ and $\mathtt{T}' = (T',d_{T'},\rho_{T'},g')$ satisfying $T \subset T'$ and $g(v) \leq g'(v)$ for all $v \in {T}$ and $\rho_T= \rho_{T'}$ together with $d_{T} = d_{T'}|_{{T}}$. 

\begin{remark}[Starting from infinite decoration] Given the monotonicity condition it is natural to wonder whether $\underline{\mathtt{T}}_\infty = \cup_{x>0}\underline{\mathtt{T}}_x$ makes sense: except in a few cases, this random tree may not be locally compact.
\end{remark}

\begin{proof}
Fix $x_0>0$. We will prove the theorem by constructing a family of actual decorated trees $\mathtt{T}_x^\bullet$ (not equivalence classes thereof) satisfying the requirments for all $0 < x \leq x_0$. In that direction, we  shall construct a family $\big(f_u^{(x)},(t^{(x)}_{u i})_{i \in \mathbb{Z}_{\geq 0}} \big)_{ u \in \mathbb{U}}$  of building blocks simultaneously for all $x \in (0,x_0]$.  Under some probability measure $\mathbb{P}$, sample independently for each index $u \in \mathbb{U}$  a family $(X^{u,(x)},\eta^{u,(x)})_{x>0}$ of coupled decoration-reproduction processes with absorption times $z^{u,(x)}$ as in Theorem \ref{thm:deco-repro-grow}; and we add to this family the degenerate cases $(X^{u,(0)}, \eta^{u,(0)})$ made of the zero function over the point $\{0\}$, i.e. absorbed at $z^{u,(0)}=0$ and with zero reproduction. Now for $x  \in (0,x_0]$ fixed, we define a labeling $\mathcal{X}^{(x)}: \mathbb{U} \to \mathbb{R}_+$ together with the building blocks as follows. First, we set $\mathcal{X}^{(x)}(\varnothing) = x$ and the first decoration $f_\varnothing^{(x)}$ is the process $X^{\varnothing,(x)}$ over the time interval $[0,z^{u,(x)}]$. We then enumerate the atoms of $\eta^{\varnothing,(x)}$ and name them 
\begin{eqnarray} \delta_{\big(t_{\varnothing i}^{(x)}, \mathcal{X}^{(x)}(i)\big)}, \quad \mbox{ for } i \geq 1.  \label{eq:strangelabeling}\end{eqnarray}
As in \cite{bertoin2024self}, we may complete the above sequence with fictitious\footnote{a fictitious building block is the zero function over a single point, and all marks $t_{ui}$ are then equal to $0$. The offspring of a fictitious building block is always made of fictitious individuals.} building blocks which are ineffective in the gluing procedure. To ensure monotonicity, \textit{a crucial point here}, is that the labeling of the atoms be coherent in $x$, that is, the $i$th atom of $\eta^{\varnothing,(x')}$ should be naturally linked in the construction of Theorem \ref{thm:deco-repro-grow} to the $i$th atom of  $\eta^{\varnothing,(x)}$ for $x'<x$. In general, the co-lexicographical enumeration based on $\eta^{\varnothing,(x)}$ does not ensure this. We shall therefore, once and for all, enumerate the  atoms of $\eta^{\varnothing,(x_0)}$ for the co-lexicographical order;  by the properties of the coupling of Theorem \ref{thm:deco-repro-grow}, the atoms of $\eta^{\varnothing,(x)}$ form a subset of the atoms of $\eta^{\varnothing,(x_0)}$ and so they naturally inherit their labeling. In particular, if the $i$th atom of $\eta^{\varnothing,(x_0)}$ disappeared in $\eta^{\varnothing,(x)}$, then we replace it by a fictitious element by putting $t_{\varnothing i}^{(x)} =0$ and $\mathcal{X}^{(x)}(i)=0$. 
We then iterate recursively the construction, for example, the decoration $f_1^{(x)}$ is the function $X^{1,(\mathcal{X}^{(x)}(1))}$ over the time interval $[0,z^{1,(\mathcal{X}^{(x)}(1))}]$ and the enumeration of the atoms of $\eta^{1,(\mathcal{X}^{(x)}(1))}$ is 
$$ \delta_{\big(t_{1 i}^{(x)}, \mathcal{X}^{(x)}(1i)\big)}, \quad \mbox{ for } i \geq 1.$$
Here also, to ensure a coherent labeling, we notice that by monotonicity we have $\mathcal{X}^{(x)}(1) \leq \mathcal{X}^{(x_0)}(1)$, so that we enumerate once and for all the atoms of $\eta^{1,(\mathcal{X}^{(x_0)}(1))}$ to get a coherent labeling for all $0 < x \leq x_0$. In particular, as long as  $t_{ui}^{(x)}$ is strictly positive, it is constant for $x \in (0,x_0]$. The building blocks constructed this way inherit the monotonicity and continuity properties of Theorem \ref{thm:deco-repro-grow}, in particular:
\begin{eqnarray}
    \label{eq:monotonecouplingulam}
 \mbox{for each $u \in \mathbb{U}$ the mapping $ x \mapsto \mathcal{X}^{(x)}(u)$ and $x \mapsto f_u^{(x)}$ are continuously increasing on $(0,x_0]$,}\end{eqnarray}
Now, it is plain from Theorem \ref{thm:deco-repro-grow} that  for $x_0>0$ fixed, the law of the building blocks $(f_u^{(x_0)})_{\u \in \U}$ and $\big((t^{(x_0)}_{u i})_{i \in \mathbb{Z}_{\geq 0}} : u \in \mathbb{U}\big)$ is that of the building blocks under the measure $\mathbb{P}_{x_0}$ described in the previous Section~\ref{sec:background}, and in particular, by Proposition 2.10 \cite{bertoin2024self}, these blocks are glueable and yield a compact decorated tree $\mathtt{T}_{x_0} = (T_{x_0}, d_{x_0}, \rho_{x_0}, g_{x_0})$.  Formally this tree is obtained first by taking the disjoint union of all segments 
$$ T^{\sqcup}_{x_0} :=  \bigsqcup_{u \in \mathbb{U}} \{u,[0, z^{u,(x_0)}]\},$$ where the union is taken over non fictitious individuals, and then making the gluing identification $\sim_{x_0}$ where the origin $(ui,0)$ of the segment $\{ui,[0,z^{ui,(x_0)}]\}$ is glued on $(u,t_{ui}^{(x)})$. When equipped with the one-dimensional distance,  the closure of $T^\sqcup_x/\sim_{x_0}$  produces a compact real tree $T_{x_0}$, which is then equipped with the decoration usc $g_{x_0}$ obtained by gluing\footnote{And changing from rcll to usc convention using \cite[Eq. (1.4)]{bertoin2024self}} the decoration $f_u^{(x_0)}$, as in  \cite[Section 1.2]{bertoin2024self}. By our monotonic coupling, the glueability conditions in \cite[Definition 1.4 \& Theorem 1.6]{bertoin2024self} are also satisfied for the building blocks $\big( f_u^{(x)},(t^{(x)}_{u i})_{i \in \mathbb{Z}_{\geq 0}} \big)_{u \in \mathbb{U}}$ simultaneously for $x \in (0,x_0]$, and their gluing yields  a decorated tree ${\mathtt{T}}_x$ obtained as above as the closure of the decorated quotient of $$ T_x^\sqcup:=\bigsqcup_{u \in \mathbb{U}} \{u,[0, z^{u,(x)}]\} \subset T^\sqcup_{x_0},$$ equipped with the decoration $g_x$ obtained by gluing the functions $f_u^{(x)}$. Let us now check the desired properties of this construction:
\begin{itemize} 
\item \textbf{Monotonicity}. The trees $\mathtt{T}_x$ are naturally nested inside $\mathtt{T}_{x_0}$: since $z^{u,(x)} \leq z^{u,(x_0)}$  and $f_u^{(x)} \leq f_u^{(x_0)}$, the function $g_x$ can be extended over $T_{x_0}$ so that $g_{x}$ is constant equal to $0$ on images of segments of the form $\{u,[z^{u,(x)}, z^{u,(x_0)}]\}$. By monotonicity we have $g_{x'} \leq g_x$ over $T_{x_0}$ for any $0 < x' \leq x \leq x_0$. The support $\{ p \in T_{x_0} : g_x(p)>0\} \subset T_{x_0}$ of $g_x$ is then easily identified with the (projection of) $\bigsqcup_{u \in \mathbb{U}} \{u,[0, z^{u,(x)})\}$, and it is tedious but straightforward to check that $\mathtt{T}_x$ is isomorphic to the hypograph of the function $g_x$ over the closure of $\{ p \in T_{x_0} : g_x(p)>0\} \subset T_{x_0}$. The monotonicity follows.
\item \textbf{Continuity}. Fix $K \geq 1$ and consider the decorated tree $\mathtt{T}_x^{[K]}$ obtained by the above gluing procedure but restricted to the finitely many branches indexed by $u \in \bigcup_{k=1}^K \{0,1, ... ,K\}^k$. An iterative application of Theorem \ref{thm:deco-repro-grow} shows that $x \mapsto \mathtt{T}_x^{[K]}$ is continuous for the hypograph distance $d_{\mathbb{H}}$. Since all decorated trees are monotonically nested within $ \mathtt{T}_{x_0}$ we have 
$$ \sup_{x \in [0,x_0]} d_{\mathbb{H}}( \mathtt{T}_x^{[K]}, \mathtt{T}_x) \leq d_{\mathbb{H}}( \mathtt{T}_{x_0}^{[K]}, \mathtt{T}_{x_0}) \xrightarrow[K \to \infty]{a.s.} 0,$$ because $(f_u^{(x_0)}, (t_{ui}^{(x_0)})_{i \geq 0})_{u \in \mathbb{U}}$   is compactly glueable, see \cite[Lemma 1.5]{bertoin2024self}. This uniform control entails that $x \mapsto \mathtt{T}_x$ is indeed continuous.
\end{itemize}
To check the next items, we  need to define the sequence of points $p_x(u) \in \mathtt{T}_x$ for $u \in \mathbb{U}$. The natural choice is to take $p_x(u)$ as the images of the extremity of the branches $(u,z^{u,(x)}) \in T_{x_0}^\sqcup$, however this will not directly work: the issue is that the indexation of those points is inherited from  \eqref{eq:strangelabeling} and does not coincide with the co-lexicographical indexation that we used in Section \ref{sec:background}. To get around this issue, for each $x$, we re-index the building blocks $\big(f_u^{(x)},(t^{(x)}_{u i})_{i \geq 0}\big)_{u \in \mathbb{U}}$ using the co-lexicographical order: first relabel the first generation $t_{\varnothing i}$ for $i \geq0$ using $f_\varnothing$ in co-lexicographical order\footnote{again we suppose that there is no tie to simplify}, then iterate. This relabeling has no effect on the genealogical structure and we obtain the new building blocks $\big(\tilde{f}_u^{(x)}, (\tilde{t}^{(x)}_{u i})_{i \geq 0}\big)_{u \in \mathbb{U}}$  which yield to another decorated tree $\tilde{\mathtt{T}}_x$.  Due to the change of labeling, those trees are not intrinsically nested anymore, however it is easy to see that $\mathtt{T}_x$ is canonically isomorphic to $\tilde{\mathtt{T}}_x$. The points  $p_x(u) \in \mathtt{T}_x$ are then the images of the extremities  $(u,\tilde{z}^{u,(x)})$ of the branches of $\tilde{\mathtt{T}}_x$ for $u \in \U$ in this canonical identification (equivalently, they are obtained by re-indexing the extremities of the branches in $\mathtt{T}_x$ for the co-lexicographical order).
\begin{itemize}
\item \textbf{Law}. As noticed above, for $x=x_0$ the building blocks   $(f_u^{(x_0)}, (t_{ui}^{(x_0)})_{i \geq 0})_{u \in \mathbb{U}}$ have law $\mathbb{P}_{x_0}$. This is not true for $x<x_0$ due to indexing issue, however one can check that their re-indexed versions $\big(\tilde{f}_u^{(x)}, (\tilde{t}^{(x)}_{u i})_{i \geq 0}\big)_{u \in \mathbb{U}}$ have law $\mathbb{P}_x$. Since $\underline{\mathtt{T}}_x^\bullet = \underline{\tilde{\mathtt{T}}}_x^\bullet$  the statement follows. 
\item \textbf{Measurability}. Fix $0 < x' \leq x \leq x_0$. If one is given $\underline{\mathtt{T}}_{x'}^\bullet \subset \underline{\mathtt{T}}_x^\bullet$, one can recover in a measurable fashion\footnote{we leave the measurability issues as a delicate exercise for the reader} the re-indexed building blocks $\big(\tilde{f}_u^{(x)}, (\tilde{t}^{(x)}_{u i})_{i \geq 0}\big)_{u \in \mathbb{U}}$ and $(\tilde{X}^{u,(x)}, \tilde{\eta}^{u,(x)})_{u \in \mathbb{U}}$, where the tilde notation above means that those decoration-reproduction processes have been re-indexed in the co-lexicographical order. By Theorem \ref{thm:deco-repro-grow} applied to each $u \in \mathbb{U}$, and relabeling according to the co-lexicographical order, this is sufficient to recover measurably $(\tilde{X}^{u,(x')}, \tilde{\eta}^{u,(x')})_{u \in \mathbb{U}}$, which in the end give access to $\tilde{\mathtt{T}}^\bullet_{x'}$ hence to $\underline{\mathtt{T}}_{x'}^\bullet$.

\item \textbf{Markov property.} Fix $0 < x' \leq x \leq x_0$. By Theorem \ref{thm:deco-repro-grow} again  applied to each $u \in \mathbb{U}$, we know that conditionally on $(X^{u,(x')}, \eta^{u,(x')})_{u \in \mathbb{U}}$, one can recover the value $X^{u,(x)}(z^{u,(x')})$ as well as the process $(X^{u,(x)}, \eta^{u,(x)})$ on the time interval $[0,z^{u,(x')}]$. Introducing the shorthand $w_{x' \to x}(u) = X^{u,(x)}(z^{u,(x')})$, under the same conditioning as above, the decoration-reproduction process $(X^{u,(x)}, \eta^{u,(x)})$ over $[z^{u,(x')} ,z^{u,(x)}]$ has the law of $(X^{(w_{x'\to x}(u))}, \eta^{(w_{x'\to x}(u))})$ which is $P_{w_{x' \to x}(u)}$. This property is preserved if one re-indexes the building blocks using the co-lexicographical rule and implies in terms of trees that the decorated trees emanating from the images of the point $(u,z^{u,(x')})$ in $\mathtt{T}_x$ are independent ssMt with law $\mathbb{Q}_{w_{x'\to x}(u)}$.  We deduce that  conditionally given $\underline{\mathtt{T}}_{x'}^\bullet$ the equivalence class  $\underline{\mathtt{T}}_x^\bullet$ is obtained by:
    \begin{itemize}
    \item Updating the decoration $g_{x'}$ to $g_x$ over $T_{x'}$, in particular defining the weights ${w}_{x'\to x}(u)$ and points $p_{x'}(u)$.
    \item Grafting independent ssMt with law $\mathbb{Q}^\bullet_{w_{x'\to x}(u)}$ on the points $p_{x'}(u)$.
    \item The new distinguished points of $\underline{\mathtt{T}}_x$ are obtained as the distinguished points of the grafted trees, re-indexed in co-lexicographical order.
    \end{itemize}
In the framework of \cite[Chapter 4]{bertoin2024self}, this produces a $\mathbb{Q}$-local decomposition of $\underline{\mathtt{T}}_x$ in terms of $\underline{\mathtt{T}}_{x'}$, with the additional fact that the gluing points and initial decorations are measurable function of $\mathtt{T}_{x'}^\bullet$ itself. 
\end{itemize}
\end{proof}

\begin{remark}[Locally largest bifurcator] We come back on the remark at the end of Section~\ref{sec:background}. We saw that the law of the points $(p(u): u \in \U)$ is different for different bifurcators coding for a same ssMt, implying in particular that one cannot recover these points in a measurable fashion solely from the equivalence class of decorated tree $\underline{\mathtt{T}}$. However, in the case where there are no ties, the special class of bifurcator which only follows the lineage of the locally largest particle yields a set of points $(p(u): u \in \U)$ which can be recovered measurably from $\underline{\mathtt{T}}$.
\end{remark}

\section{Examples!}\label{sec:examples}

It is finally time to give supporting applications of our construction. In this section, we will provide several examples of ssMt whose charcateristic quadruplets $(\mathrm{a},\sigma^2,\boldsymbol{\Xi};\alpha)$ are such that $(\boldsymbol{\Xi};\alpha)$ is $G$-growing.
We first highlight examples which are \textit{explicitly} $G$-growing, in the sense that the $G$-family is given by an explicit family of functions on which the growing conditions is readily checkable. In the other examples, the $G$-family is only \textit{implicitly} characterized through an autonomous differential equation involving their generators. Before that we factor out a few calculations.

\subsection{On bifurcators}
Recall from \cite[Section 5.3]{bertoin2024self} that two splitting measures $\boldsymbol{\Xi}$ and $\boldsymbol{\Xi}'$ are bifurcators of one another if 
\begin{align*}
    \mathrm{ord}_\sharp\,\boldsymbol{\Xi}=\mathrm{ord}_\sharp\, \boldsymbol{\Xi}',
\end{align*}
where $\mathrm{ord}(\mathbf{y})$ denotes the non-increasing re-arrangment of the vector $\mathbf{y}\in \mathcal{E}$. Given a splitting measure, its locally largest bifurcator $\boldsymbol{\Xi}^{\ell\ell}$ is defined as $\boldsymbol{\Xi}^{\ell\ell}= \mathrm{ord}_\sharp\,\boldsymbol{\Xi}$. If $\boldsymbol\Xi$ is a bifurcator of $\boldsymbol\Xi^{\ell\ell}$, its conditional law given its non-increasing re-arrangment is described by a family $\mathfrak p=(p_i)_{i\geq0}$ of mappings $\mathrm{supp}(\boldsymbol\Xi)\to[0,1]$ with
\begin{align*}
    \sum_i p_i(\mathbf{y})=1,\qquad \mathbf{y}\in \mathcal{E}.
\end{align*}
More precisely, $\boldsymbol\Xi$ has the same ``law" as the bifurcator $\boldsymbol{\Xi}^{\mathfrak p}$, in which $p_i(\mathbf{y})$, $i\geq0$ encode the probability to pick the $i$-th largest coordinate of $\mathbf{y}$ to be singled out --- in the ssMt construction, this corresponds to the particle which is ``followed'' by the individual. Formally it reads 
\begin{align} \label{eq:defbifurcatorp}
    \int_{\mathcal{E}} \boldsymbol{\Xi}^{\mathfrak p}(\d \mathbf{y}) F(y_0; y_1,y_2,\dots)
        := \int_{\mathcal{E}} \boldsymbol{\Xi}^{\ell\ell}(\d \mathbf{y})\sum_{i\geq0} p_i(\mathbf{y}) F(y_i; y_0,\dots,y_{i-1},y_{i+1},y_{i+2},\dots).
\end{align}
In particular, the bifurcator associated to $\mathfrak p^{\ell\ell}\colon \mathbf{y}\mapsto (1,0,0,\dots)$ is the locally largest one. Whenever $\boldsymbol\Xi$ is supported on summable sequences, one can also define the size-biased bifurcator $\boldsymbol{\Xi}^*$, obtained by taking $\mathfrak p^*=(p_i^*,i\geq0)$, where
\begin{align} \label{def:pstar}
    p_i^*(\mathbf{y})=\frac{y_i}{\sum_j y_j},\qquad i\geq0,\qquad \mathbf{y}\in \mathrm{supp}(\boldsymbol{\Xi}).
\end{align}
Notice also, that not all bifurcators $\boldsymbol\Xi^\frak{p}$ are splitting measures: the first coordinate should in particular produce a L\'evy measure, this is generically not the case for  $\frak{p}: \mathbf{y} \mapsto (0,1,0,0,...)$.

\subsection{A \texorpdfstring{``magic''}{"magic"} function}\label{ss:magic}

A single magic family of functions $\magic$ can be used to show that many natural ssMt are growing. This family is simply given by 
\begin{equation}\label{eq:magic_function}
    \magic_x(y)=\frac{xy}{xy+(1-y)},   \qquad x,y\in [0,1].
\end{equation}
One verifies by a straightforward computation that $\magic_x\circ \magic_{x'}=\magic_{x\cdot x'}$ for all $x,x'>0$. For many natural splitting measure $\boldsymbol \Xi$, (after a suitable choice of bifurcator) the first marginal of $\boldsymbol{\Xi}$ has the form $\nu_\gamma$ for some $\gamma\neq 0$ where \begin{align*}
    \nu_\gamma(\d y)=\mathbbm{1}_{[0,1]}(y)\frac{\d y}{y^{1-\gamma}(1-y)^{1+\gamma}}.
\end{align*}
We will then see below that $(\magic_x)_\sharp\  \nu_\gamma = x^{-\gamma}\cdot \nu_\gamma$ so that the family $\magic$ can be used to build growing family of functions. More precisely, the splitting measures that we will consider have one of the two following forms:
\begin{itemize}
    \item \textbf{``Mass form'':} For $\gamma>0$, a splitting measure $\boldsymbol{\Xi}$ is said to have form $\mathsf M_\gamma$ if there is a {non-trivial} Borel measure $\boldsymbol{\Theta}$ on $\mathcal{E}$ such that for all Borel $F\colon \mathcal{E}\to\R_+$,
    \begin{equation}\label{eq:mass_form}
                \int_{\mathcal{E}}\boldsymbol{\Xi}(\d y_0, \d \mathbf{y}) \,F(y_0, \mathbf{y})
        =   \int_0^1\nu_\gamma(\d y_0)\int_{\mathcal{E}^{\downarrow}}\boldsymbol{\Theta}(\d\mathbf y)\,F\bigl(y_0,(1-y_0)\mathbf y\bigr).
    \end{equation}

    \item \textbf{``Height form'':} For $\gamma>0$, a splitting measure $\boldsymbol{\Xi}$ is said to have form $\mathsf H_\gamma$ if there is a non-trivial Borel measure $\boldsymbol{\Theta}$ on $\mathcal{E}$ such that for all Borel $F\colon \mathcal{E}\to\R_+$,
    \begin{equation}\label{eq:height_form}
        \int_{\mathcal{E}}\boldsymbol{\Xi}(\d y_0, \d \mathbf{y}) \,F(y_0, \mathbf{y})
        =   \int_0^1\nu_{-\gamma}(\d y_0)\int_{\mathcal{E}^{\downarrow}}\boldsymbol{\Theta}(\d\mathbf y)\,F\bigl(1,y_0,y_0\cdot\mathbf y\bigr).
    \end{equation}
\end{itemize}
The above discussion is then formally encapsulated in the following proposition.

\begin{proposition}\label{prop:fct-mag-G-mass-height}
    Let $\gamma> 0$ and let $\boldsymbol{\Xi}$ be a splitting measure.\\
    \noindent If $\boldsymbol{\Xi}$ has form $\mathsf M_\gamma$, then $(\boldsymbol{\Xi};\gamma)$ is $G^\mathsf{M}$-growing for 
    \begin{align*}
        G^{\mathsf M}_x(y_0,\mathbf{y})=\Bigl(\magic_x(y_0),\frac{1-\magic_x(y_0)}{1-y_0}\cdot \mathbf{y}\Bigr),\qquad (y_0,\mathbf{y})\in \mathcal{E}, \quad x >0.
    \end{align*}
    If $\boldsymbol{\Xi}$ has form $\mathsf H_{\gamma}$ then $(\boldsymbol{\Xi},\gamma)$ is $G^\mathsf{H}$-growing for
    \begin{align*}
        G^{\mathsf H}_x(1,y_1,\mathbf{y})=\Bigl(1,\magic_{x^{-1}}(y_1),\magic_{x^{-1}}(y_1)\cdot \mathbf{y}\Bigr),\qquad (1,y_1,\mathbf{y})\in \mathcal{E}, \quad x >0
    \end{align*}
    and $G^{\mathsf H}_x$ is taken to identical to zero on $\{(y_0,y_1,\mathbf y)\in\mathcal E\colon y_0\neq 1\}$.
\end{proposition}
\begin{proof} 
Let us check the conditions of Definition~\ref{def:growing} for the mass form $\mathsf{M}_\gamma$. The height form $\mathsf{H}_\gamma$ is analogous. 
\begin{itemize}
    \item \textbf{Semi-group property.} This readily follows by the same property of $\mathsf{m}_x$: if $0 < x' \leq x$, then 
    \[
    \left( G_x^\mathsf{M} \circ G_{x'}^\mathsf{M} \right) (y_0, \mathbf{y}) = \left( \mathsf{m}_x( \mathsf{m}_{x'}(y_0)), \frac{1 - \mathsf{m}_x( \mathsf{m}_{x'} (y_0))}{1 - \mathsf{m}_{x'}(y_0)} \cdot \frac{1 - \mathsf{m}_{x'}(y_0)}{1 - y_0} \cdot \mathbf{y} \right) = G_{x \cdot x'}^\mathsf{M} (y_0, \mathbf{y})
    \]
    \item \textbf{Quasi-preservation of the measure.} First, let us remark the function $\mathsf{m}_x$ is exactly the bijection from $[0,1]$ onto itself such that for any Borel function $f : [0,1] \to \R_+$, we have
    \begin{equation}\label{eq:ode_mass_form}
        \int_0^1 f(\mathsf{m}_x(y)) \nu_\gamma (\d y) = x^{- \gamma} \int_0^1 f(y) \nu_\gamma(\d y).
    \end{equation}
    Then a straightforward calculation shows
    \begin{align*}
        \int_{\mathcal{E}}\boldsymbol{\Xi}(\d y_0, \d \mathbf{y}) \,F\left( G_x^\mathsf{M} (y_0,\mathbf{y}) \right)
        &=   \int_0^1\nu_\gamma(\d y_0)\int_{\mathcal E}\boldsymbol{\Theta}(\d\mathbf y)\,F\bigl(\mathsf{m}_x(y_0),(1- \mathsf{m}_x(y_0))\mathbf y\bigr) \\
        &= x^{- \gamma} \int_0^1\nu_\gamma(\d y_0)\int_{\mathcal E}\boldsymbol{\Theta}(\d\mathbf y)\,F\bigl(y_0,(1- y_0)\mathbf y\bigr) \\
        &= x^{-\gamma} \int_{\mathcal{E}}\boldsymbol{\Xi}(\d y_0, \d \mathbf{y}) \,F(y_0,\mathbf{y}).
    \end{align*}
    \item \textbf{Monotonicity.} Let $x,y \in [0,1]$. Straight from the definition of $\mathsf{m}$, it is easily checked that $x\mathsf{m}_x(y) \leq y$ and $x(1 - \mathsf{m}_x(y)) \leq 1 - y$, so that 
    \[
    x \cdot G_x^\mathsf{M}(y_0, \mathbf{y}) \leq (y_0, \mathbf{y}).
    \]
    \item \textbf{Regularity of $G_x^\mathsf{M}$.} Let $x,x' \in K$ a compact of $(0, \infty)$. We check the regularity condition by hand using a derivation. Since $x$ is bounded away from $0$ and infinity,
    \[
    \left( \partial_x (x (\mathsf{m}_x(y) - 1)) \right)^2 = \left( \frac{(y - 1)(1 - x)}{xy + 1 - y} \right)^2 \leq C_K (y-1)^2.
    \]
    Then, since $\mathsf{m}_0(y) = 0$, a Taylor argument yields
    \[
    \int_0^1 (x (\mathsf{m}_x(y) - 1) - x'(\mathsf{m}_{x'}(y) - 1))^2 \nu_y (\d y) \leq C_K (x - x')^2 \int_0^1 (y-1)^2 \nu_y(\d s) = C_K (x- x')^2,
    \]
    and $\nu_\gamma$ integrates $(1-y)^2$ by assumption. 
\end{itemize}
\end{proof}
In many situations, the index $\gamma$ coincides with the infimum of the support of the cumulant function so that by Lemma \ref{lem:alpha-vs-cumulant} we have $\alpha_c(\boldsymbol\Xi) = \gamma$.

\subsection{The Brownian CRT}
We start our list of applications with the most canonical ssMt, namely the Brownian CRT of Aldous, seen as a fragmentation tree in the sense of Haas and Miermont \cite{HaasMiermont2004}.
\begin{example}[Brownian CRT by mass, locally largest bifurcator] \label{ex:brownian_crt_mass_locally_largest}
 Referring to \cite[Chapter 3, Example 3.6]{bertoin2024self}, the Brownian CRT of mass $1$, once decorated by the function which to a vertex associates the mass of the fringe subtree originating from this vertex, is distributed as the ssMt with initial decoration $1$ with characteristics $( \mathrm{a}=0, \sigma^2=0,  \boldsymbol\Xi^{\ell\ell}; \alpha = \frac{1}{2})$, where the splitting measure  is 
 \begin{eqnarray} \int  \boldsymbol{\Xi}^{\ell \ell}_{\mathrm{Bro}} (\mathrm{d} \mathbf{y})F( {y_0},{y_1},...) \propto \int_{1/2}^1 \frac{\mathrm{d}y}{(s(1-s))^{3/2}}  F(s,1-s,0,\cdots). \label{eq:defllbrow}
\end{eqnarray}
This description uses the locally-largest bifurcator in the ssMt formalism, where at each branch point we follow the branch with the largest decoration. In this case, the growing function is explicit and given by
\begin{equation} \label{def:fellell}
	\mathtt{f}_x(s) = \frac12 \left(1 + \sqrt{\frac{h(s)}{\left(4x^{-1}+h(s)\right)}}\right), \quad \mbox{with } h(s) = \frac{(1 - 2s)^2}{s(1-s)}, \quad s \geq 1/2,
\end{equation}
and $G_x^{\ell \ell}(s) = (\mathtt{f}_x(s), 1 - \mathtt{f}_x(s))$. A direct calculation then shows that $( \boldsymbol{\Xi}^{\ell \ell}_{\mathrm{Bro}} ; 1/2)$ is $G^{\ell\ell}$-growing. Combining with Lemma \ref{lem:alpha-vs-cumulant}, this shows that $\alpha_c(\boldsymbol{\Xi}^{\ell\ell}_{\mathrm{Bro}}) = 1/2$ since the cumulant function associated to the Brownian CRT blows up at $1/2$, see \cite[Chapter 3, Example 3.6]{bertoin2024self}. See also Example \ref{ex:stable_crt_mass_locally_largest} for a general perspective on this function.
\end{example}

\begin{example}[Brownian CRT by mass, size-biased bifurcator] \label{ex:brownian_crt_mass_size_biased} Ellaborating on the previous example, if one instead takes the size-biased bifurcator, following a size-biased particle at each branch point, then the Brownian CRT of mass $1$ can equivalently be described as the ssMt starting from $1$ with characteristics $(0,0, \boldsymbol\Xi^*_{\mathrm{Bro}} ; 1/2)$ where $\boldsymbol{\Xi}^*_{\mathrm{Bro}}$ is the corresponding size-biased splitting measure:
\begin{eqnarray} \int  \boldsymbol{\Xi}^*_{\mathrm{Bro}} (\mathrm{d} \mathbf{y})F( {y_0},{y_1},...) \propto \int_{0}^1 \frac{s\,\mathrm{d}s}{(s(1-s))^{3/2}} \,  F(s,1-s,0,\cdots), \quad s \in [0,1].\end{eqnarray}
This has the form $\mathsf M_\gamma$ for $\gamma=1/2$ and $\Theta(\d \mathbf{y})=\delta_{(1,0,0,\dots)}(\d \mathbf{y})$ in \eqref{eq:mass_form}. In particular, $(\boldsymbol{\Xi}^*_{\mathrm{Bro}},1/2)$ is $G^\mathsf{M}$-growing by Proposition~\ref{prop:fct-mag-G-mass-height} and again the self-similarity parameter is critical $\alpha_c(\boldsymbol{\Xi}^*)= 1/2$ using Lemma \ref{lem:alpha-vs-cumulant}.
\end{example}

\begin{example}[Brownian CRT by mass with a weird bifurcator] \label{ex:brownian_crt_mass_weird}
Still continuing with the above examples, we now consider the bifurcator of \eqref{eq:defllbrow} obtained as in \eqref{eq:defbifurcatorp} with the  function
\begin{equation}\label{eq:bif_weird}
    \mathfrak{p}^{ \mathrm{weird}}(s,1-s) = \left(p(s),1-p(s)\right), \quad \mbox{with } p(s)= \frac{s(11/10-s)}{s(11/10-s) + (1-s)(1/10 + s)}.
\end{equation}
If the obtained splitting measure $\boldsymbol\Xi^\mathrm{weird}_{\mathrm{Bro}}$ is $(G; \alpha)$ growing for $G_x(s) = (\mathtt{f}_x(s), \mathtt{f}_x(1-s))$ then  \eqref{eq:quasipreservation} gives a differential equation which fixes $\mathtt{f}_x(s)$, see the forthcoming Proposition \ref{prop:solution-conservative-binary-case} for details. The monotonicity condition \eqref{eq:monotonicity} in the conservative binary case boils down to:
\[
1 - \frac{1-s}{x} \leq \mathtt{f}_x(s) \leq \frac{s}{x} \qquad \text{ for all } 0 < x \leq 1 \text{ and } s \in [0,1].
\]
Solving numerically this equation, it is seen in Figure~\ref{fig:graph_bifurcator} that for $x = 1/2$, the monotonicity is satisfied for $\alpha = 1/5$ and is violated for $\alpha = 1/2$. So, $\alpha_c(\boldsymbol{\Xi}^{\text{weird}}_{\text{Bro}}) < 1/2$. In particular, $\alpha_c$ heavily depends on the bifurcator of a given ssMt!

\begin{figure}
    \centering
    \includegraphics[width=0.4\linewidth]{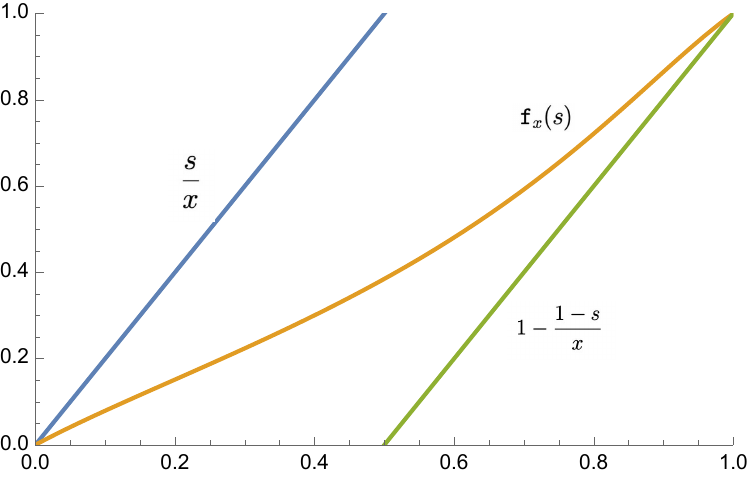} \includegraphics[width=0.4\linewidth]{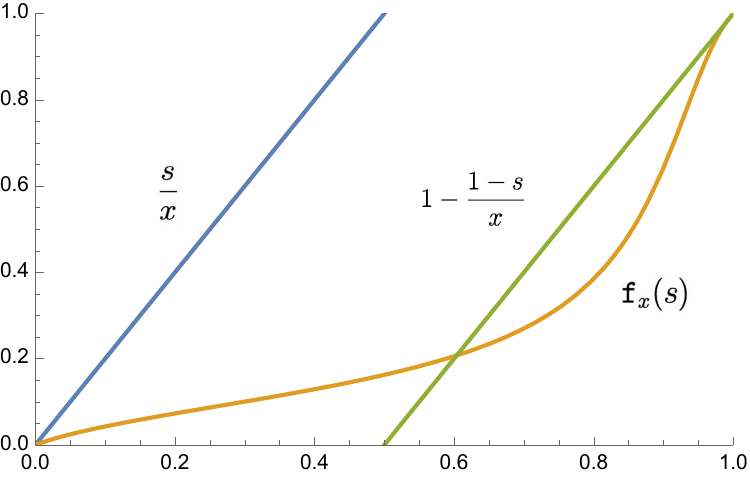}
    \caption{For $x = 1/2$, the monotonicity condition is satisfied for $\alpha = 1/5$ on the left and it is \textbf{not} for $\alpha = 1/2$ on the right. }
    \label{fig:graph_bifurcator}
\end{figure}
\end{example}

We now turn to a variant of the Brownian CRT which appears as a ssMt with drift and different self-similarity index.
\begin{example}[Brownian CRT by height, locally largest]  \label{ex:brownian_crt_height}
As illustrated by \cite[Chapter 3, Example 3.5]{bertoin2024self}, the Brownian CRT of height $1$, once decorated by the function which to a vertex associates the height of the fringe subtree originating from this vertex, is a  ssMt defined by characteristics $(\mathrm{a} = -1, \sigma^2 = 0, \boldsymbol{\Xi}_{\mathrm{Bro}}^{\mathrm{h}}; \alpha = 1)$ with splitting measure
\[ 
\int_{\mathcal{E}} \boldsymbol\Xi_{\mathrm{Bro}}^h( \d \mathbf{y}) F( y_0, y_1, \ldots) \propto \int_{0}^1 \frac{\d h}{h^2}F( 1, h, 0,\cdots).
\]
Notice that this correspond to the locally largest bifurcator and that the characteristics of the ssMt include a drift term since the decoration process along the first branch is remarkably easy, namely $X^{(x)}_t = x-t$ for $t \in [0,x]$ under $ \mathbb{Q}_x$. The growing family is of the ``height" form $\mathsf H_1$ as in \eqref{eq:height_form} and it follows by Proposition~\ref{prop:fct-mag-G-mass-height} that $(\boldsymbol{\Xi}_{\text{Bro}}^h; 1)$ is $G^\mathsf{H}$-growing. This corresponds to the critical self-similarity parameter since by Lemma \ref{lem:alpha-vs-cumulant} we have $\alpha_c(\boldsymbol{\Xi}_{\text{Bro}}^h) \leq 1$, because the cumulant $\kappa$ explodes at $1$, see \cite[Eq (3.7),Chapter 3, Example 3.5]{bertoin2024self}.
\end{example}

\subsection{\texorpdfstring{$\beta$}{β}-stable trees}
The $\beta$-stable trees \cite{le2002random} are stable generalization of the Brownian CRT and appear as scaling limits of critical Galton--Watson trees whose offspring distribution is in the domain of attraction of the $\beta$-stable law for $\beta \in (1,2)$. 
\begin{example}[$\beta$-stable trees by mass, size-biased]  \label{ex:stable_crt_mass_size_biased} Fix $\beta \in (1,2)$. By a result of Miermont \cite{Miermont2003}, the $\beta$-stable trees can also be considered as fragmentation trees, and can therefore be re-cast in the ssMt formalism. As in Examples \ref{ex:brownian_crt_mass_locally_largest} and \ref{ex:brownian_crt_mass_size_biased} , the decoration of any vertex is again given by the mass of the fringe subtree dangling from this vertex. This gives a ssMt with characteristic quadruplet $(\mathrm{a}=0, \sigma^2=0, \boldsymbol{\Xi}^{\ell\ell}_{\beta\text{-st}}; \alpha=1-\frac 1\beta)$, where the splitting measure $\boldsymbol{\Xi}^{\ell\ell}_{\beta\text{-st}}$ is given by
\begin{align}\label{eq:def-ll-stable-mass}
\int_{\mathcal{E}}  \boldsymbol{\Xi}^{\ell\ell}_{\beta\text{-st}}(\d \mathbf{y}) \,F(y_0, y_1, \ldots) \propto \E \left[ S_1 F\left( \frac{\Delta S_{t_i}}{S_1}: i \geq 0 \right) \right],
\end{align}
where $\Delta S_{t_i}$ are the jumps, ranked in decreasing order, of a $1/\beta$-stable subordinator $(S_t)_{0 \leq t \leq 1}$ started from $0$, see \cite[Chapter 3, Example 3.7]{bertoin2024self}. The above splitting measure is thus the locally largest bifurcator. As we will see later in Example \ref{ex:stable_crt_mass_locally_largest}, the growing function for this bifurcator is not explicit.  We will rather consider here the size-biased bifurcator $\boldsymbol{\Xi}^{*}_{\beta\text{-st}}$ obtained via \eqref{eq:defbifurcatorp} using the function $\mathfrak{p}^*$ in \eqref{def:pstar}. This bifurcator is probabilistically easier to handle since it has form $\mathsf M_\gamma$ with $\gamma=1-1/\beta$ as in \eqref{eq:mass_form}.  This can be extracted from \cite[Corollary 10]{haas2009spinal} but we include a quick derivation: by definition of the size-biased bifurcator and of the distribution $\boldsymbol{\Xi}^{\ell\ell}_{\beta\text{-st}}$, for any bounded measurable $F\colon \mathcal{E}\to \R$, we have
    \begin{align*}
        \int_{\mathcal{E}}  \boldsymbol{\Xi}^{*}_{\beta\text{-st}}(\d \mathbf{y}) \,F(y_0, y_1, \ldots)&= 
        \int_{\mathcal{E}}  \boldsymbol{\Xi}^{\ell\ell}_{\beta\text{-st}}(\d \mathbf{y}) \,\sum_{i\geq0} \frac{y_i}{y_0+y_1+\dots}\,F(y_i, y_1, \dots,y_{i-1},y_{i+1},\dots)\\
        &\propto   \E \left[ S_1\cdot \sum_{i}\frac{\Delta S_{t_i}}{S_1}\cdot F\left(\frac{\Delta S_{t_i}}{S_1}; \left( \frac{\Delta S_{t_j}}{S_1}: j \geq 0,j\neq i \right) \right) \right],\\
        &=  \E \left[ \Delta S_{t_i}\cdot F\left(\frac{\Delta S_{t_i}}{S_1}; \left( \frac{\Delta S_{t_j}}{S_1}: j \geq 0,j\neq i \right) \right) \right].
    \end{align*}
    The magnitude of the jumps of $(S_t,0\leq t\leq 1)$ form a Poisson point process with intensity proportional to $\d x/x^{1+1/\beta}$. Hence, by a Palm formula, the last display is proportional to
    \begin{align*}
        & \int_0^\infty \frac{\d x}{x^{1+1/\beta}}\cdot x \cdot \E\left[F\left(\frac{x}{S_1+x},\frac{\Delta S_{t_j}}{S_1+x}: j \geq 0\right)\right]
        =   \E\left[\int_0^\infty \frac{\d x}{x^{1/\beta}} F\left(\frac{x}{S_1+x},\frac{\Delta S_{t_j}}{S_1+x}: j \geq 0\right)\right] \\
        &\hspace{.3\textwidth}=   \E\left[\int_0^1\frac{S_1\cdot (1-y)^{-2}\,\d y}{(S_1 y/(1-y))^{1/\beta}} F\left(y,(1-y)\frac{\Delta S_{t_j}}{S_1}: j \geq 0\right)\right],
    \end{align*}
    where we made the change of variables $y= x/(S_1+x)$, that is $x=S_1 y/(1-y)$. The latter equals
    \begin{align*}
        \int_0^1\nu_{1-1/\beta}(\d y)\int_{\mathcal{E}^{\downarrow}}\boldsymbol{\Theta}_{\beta\text{-st}}^{\ell \ell}(\d\mathbf y')\,F\bigl(y,(1-y)\mathbf y'\bigr),
    \end{align*}
    where $\boldsymbol{\Theta}_{\beta\text{-st}}^{\ell \ell}$ is the Borel measure \footnote{In fact, $\boldsymbol{\Theta}_{\beta\text{-st}}^{\ell \ell}$, normalized to a probability measure, is a Poisson--Dirichlet distribution see \cite[Corollary 10]{haas2009spinal}.} on $\mathcal{E}^{\downarrow}$ given by
    \begin{align*}
        \int_{\mathcal{E}^\downarrow} \boldsymbol{\Theta}_{\beta\text{-st}}^{\ell \ell}(\d \mathbf{y}) F(\mathbf{y}) = \E \left[ S_1^{1 - 1/\beta} F \left( \frac{\Delta S_{t_j}}{S_1}: j \geq 0 \right) \right].
    \end{align*}

This proves our claim that the size-biased bifurcator $\boldsymbol{\Xi}^{*}_{\beta\text{-st}}$ has form $\mathsf M_\gamma$ with $\gamma=1-1/\beta$. In particular, $(\boldsymbol{\Xi}^{*}_{\beta\text{-st}},1-1/\beta)$ is $G^\mathsf{M}$-growing and $\alpha_c(\boldsymbol{\Xi}^{*}_{\beta\text{-st}}) = 1-1/\beta$ by Proposition~\ref{prop:fct-mag-G-mass-height} and Lemma \ref{lem:alpha-vs-cumulant}.
\end{example}

\begin{example}[$\beta$-stable trees by mass, locally-largest] \label{ex:stable_crt_mass_locally_largest} Using the preceding example, it is then possible to integrate out the randomness of the size-biased bifurcator to prove that the locally largest bifurcator $(\boldsymbol{\Xi}^{\ell\ell}_{\beta\text{-st}},1-1/\beta)$ is $G^{\ell\ell}$-growing for the family of function defined by \eqref{def:gellell}. In particular, the restriction of those functions to $\mathbf{y} =(y_0,1-y_0, 0, 0,0,...)$ recovers \eqref{def:fellell}.  See Example \ref{ex:stable_crt_mass_locally_largest_proof} for a proof using the generator.
\end{example}

As in Example \ref{ex:brownian_crt_height}, one can also interpret the $\beta$-stable trees conditioned on having height $1$ as a ssMt:
\begin{example}[$\beta$-stable trees by height, locally-largest] \label{ex:stable_crt_height} Following \cite[Chapter 3, Example 3.7]{bertoin2024self} the $\beta$-stable tree of height $1$, once equipped with the decoration which associates to each point the height of its fringe subtree is a ssMt (under $\mathbb{Q}_1$) with  characteristic quadruplet $( \mathrm{a}=-1, \sigma^2=0, \boldsymbol \Xi^{\mathrm{h}}_{\beta\text{-st}}; 1)$. Let $\theta=\frac{1}{\beta-1}$ and let $F\colon \mathcal{E}\to\R_+$ be a Borel function, then $\boldsymbol \Xi^{\mathrm{h}}_{\beta\text{-st}}$ acts on $F$ by

\begin{equation}\label{eq:split_beta_height}
    	\int_\mathcal{E}\boldsymbol{\Xi}^{h}_{\beta\text{-st}}(\d \mathbf{y}) F(\mathbf{y})  \propto  \int_0^\infty \frac{\d r}{r^\beta} \exp \left( -r \theta^\theta \right) 
    \mathbb E^{(r)}\left[F(1,Y_1,Y_2,\dots)\right],
\end{equation}

where under $\mathbb P^{(r)}$, the sequence $(Y_i)_i$ is the decreasing re-arrangement of a PPP with intensity measure $\mathbbm{1}_{[0,1]}(y) \cdot r\cdot(\theta/y)^{1+\theta}\d y$. Since $Y_1=\max_i Y_i$, by a Palm formula, we have:
\begin{align*}
    \mathbb E^{(r)}\left[F(1,Y_1,Y_2,\dots)\right]
        &=\mathbb E^{(r)}\left[\sum_j \mathbbm{1}_{\{Y_j > \sup_{i \neq j} Y_i\}} F(1, Y_j, (Y_i\colon i\geq 1, i\neq j ))\right]\\
        &=\int_0^1 \d y\, r\left(\frac{\theta}{y}\right)^{1+\theta}\mathbb E^{(r)}\left[\mathbbm{1}_{\{y> \sup_{i} Y_i\}} F(1, y, (Y_i\colon i\geq 1))\right]\\
        &=\int_0^1 \d y\, r\left(\frac{\theta}{y}\right)^{1+\theta}\exp(-r\cdot\theta^\theta\cdot(y^{-\theta} -1))\,\mathbb E^{(r,y)}\left[ F(1, y, (Z_i\colon i\geq 1))\right],
\end{align*}
where $\exp(-r\cdot\theta^\theta\cdot(y^{-\theta}-1))$ is the probability that $(Y_i)_i$ does not meet $[1,y]$, and under $\mathbb P^{(r,y)}$, the sequence $(Z_i)_i$ is distributed as $(Y_i)$ under $\mathbb P^{(r)}$, but conditioned to not meet $[1,y]$ --- this is just a PPP with intensity $\mathbbm{1}_{[0,y]}(z) \cdot r\cdot(\theta/z)^{1+\theta} \d z$.
Plugging back in \eqref{eq:split_beta_height}, we have:
\begin{align*}
    \int_\mathcal{E}\boldsymbol{\Xi}^{h}_{\beta\text{-st}}(\d \mathbf{y}) F(\mathbf{y}) 
        &\propto
        \int_0^1 \frac{\d y}{y^{1+\theta}}
        \int_0^\infty
        \frac{\d r}{r^{\frac{1}{\theta}}}
            \,\exp(-r\cdot(\theta/y)^\theta)\,\mathbb E^{(r,y)}\left[ F(1, y, (Z_i\colon i\geq 1))\right]
\end{align*}
Performing the change of variables $u = r / y^\theta$ yields
\begin{align*}
        \propto
        \int_0^1\frac{\d y}{y^2}
        \int_0^\infty
        \frac{\d u}{u^{\frac{1}{\theta}}}
            \,\exp(-u \cdot \theta^\theta)\,\mathbb E^{(u\cdot y^\theta,y)}\left[ F(1, y, (Z_i\colon i\geq 1))\right]
\end{align*}
where $(Z_i)_i$ is thus distributed as a Ppp with intensity $\mathbbm{1}_{[0,y]}(z) \cdot uy^\theta\cdot(\theta/z)^{1+\theta}\d z$. By the mapping theorem, dividing by $y$ all the atoms $Z_i$ modifies exactly the intensity to $\mathbbm{1}_{[0,1]}(z) \cdot u \cdot(\theta/z)^{1+\theta}\d z$, which no longer depends on $y$. Thus, we write
\begin{align*}
    \int_\mathcal{E}\boldsymbol{\Xi}^{h}_{\beta\text{-st}}(\d \mathbf{y}) F(\mathbf{y}) 
        &\propto
        \int_0^1 \nu_{-1}(\d y)
        \int_0^\infty
        \frac{\d u}{u^{\frac{1}{\theta}}}
            \,\exp(-u \cdot \theta^\theta)\,\mathbb E^{(u)}\left[ F(1, y, (y \cdot S_i\colon i\geq 1))\right]
\end{align*}
where $(S_i)_i$ are the atoms of a Ppp with the intensity mentioned above. One recognizes immediately the height form $\mathsf{H}_1$. By Proposition~\ref{prop:fct-mag-G-mass-height}, the couple $(\boldsymbol{\Xi}_{\beta\text{-st}}^h; 1)$ is $G^\mathsf{H}$-growing. In particular, we have that $\alpha_c(\boldsymbol{\Xi}_{\beta\text{-st}}^h) = 1$ by Lemma~\ref{lem:alpha-vs-cumulant} even though the cumulant $\kappa$ explodes at $\gamma = \frac{1}{\beta - 1} \geq 1$, stressing that the bound in \eqref{eq:bound_alpha} does not always match $\inf\{ \gamma > 0 : \kappa(\gamma) < + \infty \}$.

\end{example}

\subsection{Haas--Stephenson fragmentation trees}
We finish this first list of examples with a conservative but non-binary fragmentation tree. As we will see in Example \ref{ex:haas_stephenson_mass_cex} the family of growing functions is not unique in this case.

For $k\geq 1$, we let $\mathcal S_{k+1}$ be the simplex $\mathcal S_{k+1}=\{(s_0,\dots,s_k)\in\R_+^{k+1}\colon s_0+\cdots+ s_k=1\}$. Integration of a nonnegative function $f\colon \mathcal S_{k+1}\to \R^+$ is understood as follows:
\begin{align*}
    \int_{\mathcal S_{k+1}}\d\mathbf s\, f(\mathbf s):=\int_{0\leq s_0+\cdots+s_{k-1}\leq 1}\d s_0\dots \d s_{k-1}\,f(s_0,\dots, s_{k-1},s_k),
\end{align*}
with the implicit notation $s_k=1-s_0-\dots-s_{k-1}$.

\begin{example}[Haas-Stephenson $k$-fragmentation tree by mass, special bifurcator] \label{ex:haas_stephenson_mass} Fix an integer $k \geq 1$. The Haas-Stephenson $k$-fragmentation trees are the scaling limits of random trees grown on a R\'emy-type algorithm \cite{haas2015scaling}. By \cite[Chapter 3, Example 3.9]{bertoin2024self} they can be seen as ssMt with characteristics $( \mathrm{a}=0, \sigma^2=0, \boldsymbol{\Xi}_{k\text{-HS}}^{\ell\ell}; \alpha =\frac{1}{{k+1}})$ where
\begin{align}\label{eq:HaasStephenson}
    &\int_{\mathcal{E}}
        \boldsymbol{\Xi}^{\ell\ell}_{k\text{-HS}}(\d \mathbf{y})\,
        F(y_0,y_1,\dots)\nonumber\\
    &\qquad\propto
        \int_{\mathcal S_{k+1}}
            \d \mathbf s \,\mathbbm{1}_{\{s_0\geq s_1\geq\dots\geq s_k\}}\cdot
            \frac{
                    (1-s_0)^{-1}+\dots + (1-s_k)^{-1}
                }{
                    (s_0\dots s_k)^{1-\alpha_k}
                }
            \cdot
             F(s_0,\dots,s_k,0,0,\dots).
\end{align}
We will see later in Proposition \ref{prop:locallylargestthebest} that $(\boldsymbol{\Xi}_{k\text{-HS}}^{\ell\ell}; \frac{1}{k+1})$ is growing but not explicitly. As in Example \ref{ex:stable_crt_mass_size_biased} we shall change bifurcator to get an explicitly growing ssMt. Specifically, we consider the bifurcator $\boldsymbol{\Xi}^{\mathfrak{p}}_{k\text{-HS}}$ associated as in \eqref{eq:defbifurcatorp} with the family $\mathfrak p=(p_i)_i$ given by
\begin{align*}
    p_i(s_0,\dots, s_k,0,0,\dots)=
        \mathbbm{1}_{i\leq k}\cdot \frac{(1-s_i)^{-1}}{(1-s_0)^{-1}+\dots + (1-s_k)^{-1}}
    \qquad \mathbf{s}\in \mathcal S_{k+1},
\end{align*}
then $\boldsymbol{\Xi}^{\mathfrak{p}}_{k\text{-HS}}$ has form $\rm M_\gamma$ with $\gamma=1/(k+1)$. Let us justify this: By Definitions \eqref{eq:defbifurcatorp} and \eqref{eq:HaasStephenson}, for all bounded measurable function%
    \footnote{It suffices to consider such $F$ depending on the first $k+1$ coordinates only, since $\boldsymbol{\Xi}^{\mathfrak{p}}_{k\text{-HS}}$ is supported on $\mathcal S_{k+1}$.}
$F\colon \R^{k+1}\to \R$ we have, with the notation $\rho_i(\mathbf s)=\frac1{1-s_i}(s_0,\dots,s_{i-1},s_{i+1},\dots)$ for $0\leq i\leq k$,
    \begin{align*}
        &\int_{\mathcal{E}} {\boldsymbol{\Xi}}^{\mathfrak p}_{k\text{-HS}}(\d \mathbf{y})\, F(y_0,,\dots,y_k)
            = \int_{\mathcal{E}}
                 \boldsymbol{\Xi}_{k\text{-HS}}(\d \mathbf{y})\, 
                 \sum_{i=0}^k\frac{(1-s_i)^{-1}}{(1-s_0)^{-1}+\dots + (1-s_k)^{-1}}
                 F\left(s_i,(1-s_i)\rho_i(\mathbf s)\right)\\
            &\hspace{2em}\propto\sum_{i=0}^{k}\int_{0\leq s_0+\cdots+s_{k-1}\leq 1}\d s_0\cdots\d s_{k-1}\cdot\mathbbm{1}_{\{s_0\geq s_1\geq\dots\geq s_k\}}\cdot
                \frac{ (1-s_i)^{-1}}{(s_0\dots s_k)^{1-\alpha_k}}\,F\left(s_i,(1-s_i)\rho_i(\mathbf s)\right),
    \end{align*}
    again with the implicit notation $s_k=1-s_0-\dots-s_{k-1}$.
    Performing for each $0\leq i \leq k$ the change of variables  $(y, s'_0,\dots,s'_{k-2})=(s_i,\frac{1}{1-s_i}(s_0,\dots,s_{i-1},s_{i+1},\dots,s_{k-1}))$, whose Jacobian determinant term is $(1-y)^{k-1}$, we obtain the following expression:
    \begin{align*}
            & 
                \sum_{i=0}^{k}
                \int_{0}^1
                \int_{0\leq s'_0+\cdots+s'_{k-2}\leq 1}
                      (1-y)^{k-1} \,\d y\,\d s'_0\cdots\d s'_{k-2}\cdot \mathbbm{1}_{\{s'_0\geq \dots\geq s'_{i-1}\geq \frac{y}{1-y}\geq s'_{i}\geq\dots\geq s'_{k-1}\}}  \\
                &\hspace{.25\linewidth}
                    \cdot\frac{ (1-y)^{-1}}{\bigl(y\cdot(1-y)s'_0\cdots (1-y)s'_{k-1}\bigr)^{1-\alpha_k}}\,F\left(y,(1-y)\cdot(s'_0,\dots,s'_{k-1})\right),
    \end{align*}
    where again we use the notation $s'_{k-1}:=1-s'_0-\dots-s'_{k-2}$.
    Now, exchanging sum and integral, the summation only removes the indicator function. After re-organizing and a reformulation with an integral on $\mathcal S_k $, we thus obtain:
    \begin{align*}
                \int_{0}^1\d y\, \frac{1}{y^{1-\alpha}(1-y)^{2-k\cdot\alpha}}
                \int_{\mathcal S_k}\frac{\d\mathbf s'}{\bigl(s'_0\cdots s'_{k-1}\bigr)^{1-\alpha}}\,F\left(y,(1-y)\mathbf s'\right) \quad \mbox{with } \alpha = \frac{1}{k+1}.
    \end{align*}
    Now, since $\alpha=1/(k+1)$, we have $1-\alpha=1-1/(k+1)$ and $2-k \alpha=1+1/(k+1)$, so that the last display has the form:
    \begin{align*}
                \int_{0}^1 \nu_{1/(k+1)}(\d y)
                \int_{\mathcal{E}}\boldsymbol{\Theta}_{k\text{-HS}}(\d \mathbf{y}')\, F\left(y,(1-y)\mathbf y'\right),
    \end{align*}
    for some Borel measure $\Theta^{{\rm HS},k}$ on $\mathcal{E}$.
    This proves our claim that this special bifurcator $\boldsymbol{\Xi}^{\mathfrak{p}}_{k\text{-HS}}$ has form $\rm M_\gamma$ with $\gamma=1/(k+1)$. In particular, $(\boldsymbol{\Xi}^{\mathfrak{p}}_{k\text{-HS}}; 1/(k+1))$ is $G^{\rm M}$-growing by Proposition~\ref{prop:fct-mag-G-mass-height}, and $\alpha = 1/(k+1)$ is a critical self-similarity parameter by Lemma \ref{lem:alpha-vs-cumulant}.
\end{example}

\section{The generator} \label{sec:generator}

Given  $(\boldsymbol\Xi ; \alpha)$, the $G$-growing condition of Definition \ref{def:growing}  may seem hard to verify, since we have to construct the family of functions $(G_x : x >0)$ beforehand. However, due to the required semi-group property of these functions, under mild differentiability conditions, these are specified by a so-called \textbf{generator}, namely a ``vector field'' on $\mathcal E$ which describes the infinitesimal action of the family $G_\cdot$. In smooth cases, one can then translate the requirements of Definition \ref{def:growing} in terms of the generator. Specifically, we first reparametrize a family of growing functions $(G_x : x >0)$ satisfying the requirements of Definition \ref{def:growing} using
\begin{equation}\label{eq:reparametrize}
    \mathtt{G}_t = G_{\mathrm{e}^{-t}}, \quad t \in \mathbb{R},
\end{equation}
so that $\mathtt{G}$ satisfies the additive semi-group property $\mathtt{G}_{t}\circ \mathtt{G}_s = \mathtt{G}_{t+s}$ for $s,t \in \mathbb{R}$. We can now state the definition of the \textbf{generator}. Below, \textit{differentiation} and \textit{integration} of function  $\mathcal{E} \to\mathcal \R$ are understood as \textit{coordinate-wise}.
\begin{definition}[The generator]\label{def:generator}
    Let $(\mathtt{G}_t\colon t\in\R)$ be a one-parameter group of mappings $\mathcal E\to\mathcal E$. We say that  $(\mathtt{G}_t\colon t\in\R)$ \textit{admits $\mathtt{V}$ as generator} if for all $\mathbf y\in\mathcal E$, the coordinate functions of the mapping $t\mapsto \mathtt{G}_t(\mathbf y)$ are differentiable at $t=0$ and 
    \begin{align*}
        \mathtt V\colon
        \begin{pmatrix}
            \mathcal E    &   \longrightarrow & \R^\N\\
            \mathbf y       &   \longmapsto     & \partial_t  \mathtt{G}_t(\mathbf y)\bigl.\bigr|_{t=0}
        \end{pmatrix}.
    \end{align*}
    In particular, if $t \mapsto \mathtt{G}_t$ is differentiable then by \eqref{eq:reparametrize} we deduce that $\mathtt{G}_t$ is a flow for the autonomous differential equation i.e.
\begin{eqnarray} \partial_t \mathtt{G}_t = \mathtt{V}(\mathtt{G}_t), \quad \mathtt{G}_0 = \mathrm{Id}. \end{eqnarray} 
\end{definition}

Since $\mathcal{E}$ is infinite dimensional, existence and uniqueness of the above type of differential equations is not \textit{a priori} granted and we shall impose further regularity assumptions to ensure that the family $\mathtt{G}$ is characterized by its generator. Before that, let us treat the case of binary conservative splitting measure for which the growing family $G$ is unique if it exists.

\subsection{The binary conservative case}\label{sec:bin_cons}

In order to build intuition on the generator, let us first restrict our attention to the case of a binary conservative splitting measure $\boldsymbol{\Xi}$ with a continuous positive density $\lambda$ on $(0,1)$, i.e.
$$ \int_\mathcal{E}  \boldsymbol{\Xi} (\mathrm{d} \mathbf{y})F( y_0, y_1, \ldots) = \int_{0}^1 \mathrm{d}s \lambda(s) F(s,1-s,0,\cdots),$$ and recall from the remarks below Definition \ref{def:growing} that we restrict to the case where $\int_0^1 \lambda = \infty$ otherwise the splitting measure is not growing. In this case, there is a unique continuous family of functions that satisfies the quasi-preservation of the measure:

\begin{proposition}[Quasi-preservation fixes the growing family] \label{prop:solution-conservative-binary-case}
	Assume that $\lambda$ is positive and continuous on $(0,1)$, with $\int_{0}^1 \lambda(s)\d s=+\infty$. Then there exists a unique family $\mathtt{G}_t = G_{\mathrm{e}^{-t}}$ of continuous bijections of $\{ (y,1-y) : y \in [0,1]\}$ quasi-preserving $\boldsymbol\Xi$ as in   
    \eqref{eq:ode_xi}. It is given by $\mathtt{G}_t = (\mathtt{f}_t, 1-\mathtt{f}_t)$ where $\mathtt{f}_t : [0,1] \to [0,1]$ are prescribed by 
    \begin{eqnarray} \label{eq:edoconservatif} \left\{\begin{array}{c} \mathtt{f}_0 = \mathrm{Id}\\
    \partial_t \mathtt{f}_t = \mathtt{v}(\mathtt{f}_t),
    \end{array} \right.\quad \mbox{where} \quad \mathtt{v}(s) = -\alpha \frac{\int_0^s \d u \lambda(u)}{\lambda(s)}.  \end{eqnarray}     
\end{proposition}

\begin{proof} Notice that since $\mathtt{G}_t$ is continuous and bijective on $\{(y,1-y) : y \in [0,1]\}$ the same is true for $\mathtt{f}_t$. By continuity in $t$ and the initial condition $\mathtt{f}_0= \rm Id$ we deduce that $\mathtt{f}_t$ are increasing continuous bijections $[0,1]\to [0,1]$ for every $t\geq 0$. Specifying \eqref{eq:quasipreservation} to $[0, \mathtt{f}_t(s)]$ we deduce that 
$$ I(\mathtt{f}_t(s)) = \mathrm{e}^{-\alpha t} I(s), \quad \mbox{where} \quad I(s) = \int_0^{s} \d u \lambda(u).$$
  Since $I$ is a $\mathcal{C}^1$ diffeomorphism of $[0,1) \to [0, \infty)$ we deduce that \begin{equation}\label{eq:sol_ode_cons_bin_case}
        \mathtt{f}_t(s) = I^{-1} ( \mathrm{e}^{-\alpha t} I(s)),
    \end{equation}
    which after differentiation is easily seen to be equivalent to \eqref{eq:edoconservatif}. 
\end{proof}

In the above binary conservative situation, the generator of $\mathtt{G}$ is $$\mathtt{V}:\left\{\begin{array}{ccc} \{ (s,1-s) : s \in [0,1]\} & \to & \mathbb{R}^2\\
s &\mapsto& (\mathtt{v}(s),-\mathtt{v}(s)).\end{array}\right.$$ It is clear from the flow property of \eqref{eq:edoconservatif}  that the family of homeomorphisms  $\mathtt{G}$ satisfies the additive version of the semi-group property. One can now check the other assumptions of Definition \ref{def:growing} using the generator, and foremost the monotonicity property:
\begin{proposition}[Monotonicity via the generator] \label{prop:monotonicityviageneratoriff} The family of functions $\mathtt{G}$ produced in Proposition \ref{prop:solution-conservative-binary-case} satisfies the (reparametrized) monotonicity property \eqref{eq:monotonicity} if and only if $\mathtt{v}(s) \geq s-1$, for all $s \in[0,1)$, or equivalently with the above notation that 
$$ \mathtt{V}\leq \mathrm{Id}.$$
\end{proposition}

\begin{proof}
In the reparametrization \eqref{eq:reparametrize} the monotonicity condition \eqref{eq:monotonicity} reduces to checking $$ \mathrm{e}^{-t} \mathtt{f}_t(s) \leq s\quad  \mbox{and} \quad  \mathrm{e}^{-t} (1-\mathtt{f}_t(s)) \leq 1-s.$$
Given that $\mathtt{f}_0(s)=s$, the above inequalities are granted if their derivatives with respect to $t$ are always non-positive, that is if $$ \left\{ \begin{array}{l} \partial_t( \mathrm{e}^{-t} \mathtt{f}_t(s)) \underset{\eqref{eq:edoconservatif}}{=} \mathrm{e}^{-t}( \mathtt{v}( \mathtt{f}_t(s))-\mathtt{f}_t(s)) \leq 0\\
\partial_t( \mathrm{e}^{-t} (1-\mathtt{f}_t(s)) \underset{\eqref{eq:edoconservatif}}{=} \mathrm{e}^{-t}( -\mathtt{v}( \mathtt{f}_t(s))-(1-\mathtt{f}_t(s))) \leq 0. \end{array}\right.
$$ 
Since $\mathtt{f}_{t} ([0,1]) = [0,1]$, this is again equivalent to
        \[
        \sup_{0 \leq s < 1} \frac{\mathtt{v}(s)}{s} \leq 1 \qquad \text{ and } \qquad \sup_{0 \leq s < 1} \frac{-\mathtt{v}(s)}{1 - s} \leq 1.
        \]
        Since $\mathtt{v} \leq 0$ only the second condition is relevant and it is actually necessary by looking at $t=0$. \end{proof}

Similarly, the condition \eqref{ass:lipschitz} in Definition \ref{def:growing} can be checked directly on the generator:
\begin{proposition}[Integrated Lipschitz condition via generator] \label{prop:integratedLipschitzgenerator} A necessary and sufficient condition for the family of functions $\mathtt{G}$ produced in Proposition \ref{prop:solution-conservative-binary-case} to satisfy \eqref{ass:lipschitz} is that $\mathtt{v} \in \mathbb{L}^2(\boldsymbol\Xi)$.
\end{proposition}

\begin{proof} In this proof we use $f_x = \mathtt{f}_{-\log x}$. We first assume that $\mathtt v\in\mathbb L^2(\boldsymbol\Xi$. Fix  $0 < x' \leq x$ in a compact $K \subset (0, \infty)$ so that we need to verify that 
        \begin{equation}\label{eq:ineq-lip-L2-bin}
            \int_0^1 \left( x({f}_x(s) - 1) - x'(f_{x'}(s) - 1) \right)^2 \lambda(s) \d s \leq C_K (x - x')^2.
        \end{equation}
        Introducing the function $\Theta(x,s) = x(f_x(s)-1)$ with derivative $\partial_x \Theta(x,s) = x \left( \mathtt{v}({f}_x(s)) - ({f}_x(s) - 1)\right)$ the LHS is equal to 
        \begin{eqnarray*} LHS &=&  \int_0^1 \d s \lambda(s) \left( \int_x^{x'} \d u \ \partial_x \Theta(u,s) \right)^2\\
        &=& \int_0^1 \d s \lambda(s) \left( \int_x^{x'} \frac{\d u}{x-x'} \cdot (x-x') \cdot u \left( \mathtt{v}({f}_u(s)) - ({f}_u(s) - 1)\right) \right)^2\\
        & \underset{\mathrm{Jensen}}{\leq} & \int_0^1 \d s \lambda(s)  \int_x^{x'} \frac{\d u}{x-x'} \left( (x-x') \cdot u \left( \mathtt{v}({f}_u(s)) - ({f}_u(s) - 1)\right) \right)^2\\
        &\underset{\rm Fubini}{=}& (x-x') \int_x^{x'} \d u\  u^2 \int_0^1 \d s \lambda(s) \left( \mathtt{v}(f_u(s)) - ({f}_u(s) - 1)\right)^2\\
        & \underset{\mathrm{quasi-preservation}}{=} & (x-x') \int_x^{x'} \d u\  u^2  u^{-\alpha}\int_0^1 \d s \lambda(s) \left( \mathtt{v}(s) - (s - 1)\right)^2.
        \end{eqnarray*}
        Since $\boldsymbol\Xi$ integrates both $\mathtt{v}(s)$ and $(s-1)^2$ near $1$, the above integral is bounded by $C_K (x-x')^2$ for some constant $C_K>0$.
        Conversely, if for every compact $K\subset(0,\infty)$ there exists a constant $C_K>0$ such that~\eqref{eq:ineq-lip-L2-bin} holds, then we take some compact $K$ having $1$ in its interior and use Fatou's lemma:
        \begin{align*}
            \int_0^1 \d s \lambda(s) \left( \mathtt{v}(s) - (s - 1)\right)^2
                &= \int_0^1 \d s \lambda(s)\, \lim_{x\to 1}
                   \left(\frac{\Theta(x,s)-\Theta(1,s)}{x-1}\right)^2\\
                &\leq \liminf_{x\to 1}\int_0^1 \d s \lambda(s)\, 
                   \left(\frac{\Theta(x,s)-\Theta(1,s)}{x-1}\right)^2\\
               & \leq C_K.
        \end{align*}
        Since $\boldsymbol\Xi$ integrates $(s-1)^2$, this gives that  $\mathtt v\in\mathbb L^2(\boldsymbol\Xi)$ as needed.
    \end{proof}

\subsection{General case: nicely growing splitting measures}
We shall now extend Propositions \ref{prop:solution-conservative-binary-case}, \ref{prop:monotonicityviageneratoriff} and \ref{prop:integratedLipschitzgenerator} to the general case. As said above, the first issue is the possible non-existence/uniqueness of solutions to differential equations of the type $\mathtt{G}_t = \mathtt{V}(\mathtt{G}_t)$. A second issue is that in the general case, even in smooth cases the generator $\mathtt{V}$ may not be unique and several families of growing functions could make the pair $(\boldsymbol\Xi; \alpha)$ grow. However, we shall see that $\mathtt{V}$ is solution to a partial differential equation (PDE) involving $\boldsymbol{\Xi}$ and $\alpha$ thus restricting considerably the possibilities.
We fix some $q\in [1,\infty]$ once and for all, to be chosen depending on convenience in applications. The $\ell_q$-norm will be denoted by $|\cdot|$.
 
\begin{definition}[Nicely growing] \label{def:nicelygrowing} We let $\boldsymbol\Xi$ be a splitting measure supported in $\mathcal E\cap \ell_q$ and $\alpha >0$. We say that $(\boldsymbol{\Xi}; \alpha)$ is nicely growing if there exists 
$$ \mathtt{V} : \mathcal O \subset \ell_q \to \mathbb{R}^\mathbb{N},$$
where $\mathcal O$ is an open subset of $\ell_q$ such that  $\boldsymbol\Xi(\mathcal E\setminus\mathcal O)=0$, satisfying the following properties:
\begin{itemize}
\item \textbf{Inequalities.} We have $\mathtt{V}(\mathbf{y}) \leq \mathbf{y}$ coordinate-wise for $\mathbf y\in\mathcal O\cap \mathrm{Supp}(\boldsymbol\Xi)$.
\item \textbf{Divergence PDE.} We have $\mathrm{div}(\mathtt V\,\boldsymbol{\Xi})=-\alpha \boldsymbol{\Xi}$, in the sense that for every cylindrical test function $F\colon \mathcal O \to \R_+$, we have
    \begin{align*}
        \int_{\mathcal{E}} \langle\mathtt V,\nabla F\rangle\, \d \boldsymbol{\Xi}
            = \alpha\int_{\mathcal{E}} F\,\d \boldsymbol{\Xi}.
    \end{align*}
\item \textbf{$L^2$ condition.} The $\ell_q$-norm $|\mathtt{V}|$ of the generator is in $L^2(\boldsymbol\Xi)$.

\item  \textbf{Regularity of $\mathtt V$. } With respect to the $\ell_q$-norm, $\mathtt{V}$ is Lipschitz and $C^1$-smooth.%
    \footnote{
        The latter means that: (i) for all $\mathbf y\in\mathcal O$, there is a bounded linear operator $\mathtt D\mathtt V(\mathbf y)\colon \ell_q\to\ell_q$, the differential of $\mathtt V$, satisfying that
$
    \bigl|\mathtt V(\mathbf y +t\mathbf z)-\mathtt V(\mathbf y)-\mathtt D\mathtt V(\mathbf y)\cdot\mathbf z \bigr|=o\bigl(|\mathbf z|\bigr)
$
as $|\mathbf z|\to 0$, with $|\cdot|$ the $\ell_q$-norm; and (ii), the mapping $\mathbf y\mapsto \mathtt D\mathtt V(\mathbf y)$ is continuous in operator norm.
    }
\item \textbf{No escape from $\mathcal O$.} There is $\varepsilon>0$ such that for all $\mathbf y\in \mathcal O$, we have $\mathbf y+t\mathtt V(\mathbf y)\in \mathcal O$ for all $|t|<\varepsilon$.
\end{itemize}
\end{definition}

These conditions are sufficient for the purpose of this paper, and in particular will allow us to verify the $G$-growing condition in the important case of the mass fragmentation of $\beta$-stable trees, seen through the lens of the \textit{locally largest} bifurcator, see Example~\ref{ex:stable_crt_mass_locally_largest}.
Note that we did not pursue optimal conditions and that the theory can certainly be extended: in particular, we expect that the \emph{regularity} and \emph{no escape} conditions can be weakened.
The above terminology is justified by the following proposition:
\begin{proposition}[Nicely growing $\Rightarrow$ growing]\label{prop:nicely-implies-grow} If $(\boldsymbol\Xi ; \alpha)$ is nicely growing in the sense of Definition \ref{def:nicelygrowing} then it is $G$-growing in the sense of Definition \ref{def:growing}, by setting $\mathcal E_0=\mathcal O\cap \mathrm{Supp}(\boldsymbol\Xi)$ and  $G=(G_x : x>0)$ the family of mappings $\mathcal O\to\mathcal O$ characterized by the following ODE for all $\mathbf y\in\mathcal O$
\begin{align*}
    \begin{cases}
        \partial_x{G}_{x}(\mathbf y) = -\frac{1}{x}\mathtt{V}\bigl({G}_x(\mathbf y)\bigr),
            & x>0,\\
            G_1(\mathbf y)=\mathbf y.
    \end{cases}
\end{align*}
\end{proposition}
\begin{proof}
Let us equivalently define $\mathtt{G}_{t}=G_{\mathrm{e}^{-t}}$, satisfying for all $\mathbf y\in\mathcal O$ the ODE
\begin{align*}
    \begin{cases}
        \partial_t\mathtt{G}_{t}(\mathbf y) = \mathtt{V}\bigl(\mathtt{G}_{t}(\mathbf y)\bigr),
            & t\in\R,\\
            \mathtt{G}_{0}(\mathbf y)=\mathbf y.
    \end{cases}
\end{align*}
Since $\mathtt{V} : (\mathcal O, \ell_q) \to (\mathbb{R}^\mathbb{N}, \ell_q)$ is supposed to be Lipschitz, with $\mathcal O$ open in $\ell_q$, the Cauchy--Lipschitz theorem ensures for all $\mathbf y\in\mathcal O$ the existence and uniqueness of a solution $t\mapsto\mathtt{G}_t(\mathbf{y})$ defined on a maximal interval $I(\mathbf y)\subset\R$.
Let $\mathbf y\in\mathcal O$, and let us show that $I(\mathbf y)=\R$. The \emph{no escape} condition applied iteratively with the given global $\varepsilon$ ensures that for all $t\in\R$, the Euler scheme $S_{n,0}=\mathbf y$ and $S_{n,k+1}=S_{n,k}+\frac{t}{n}\mathtt V(S_{n,k})$ stays inside $\mathcal O$ as soon as $n\geq |t|/\varepsilon$. Together with the Lipschitz assumption on $\mathtt V$, this entails that the Euler scheme approximation $(S_{n,\lfloor ns/t\rfloor}\colon -t\leq s\leq t)$ converges uniformly to a solution of the ODE $\partial_t\mathbf y_t=\mathtt V(\mathbf y_t)$ with $\mathbf y_0=\mathbf y$. Hence $I(\mathbf y)\supset (-t,t)$ for all $t>0$, and thus $I(\mathbf y)=\R$.
Hence $\mathtt G$ is now well-defined as a mapping $\R\times\mathcal O\to\mathcal O$, which is in fact continuous by a classical application of the  Grönwall lemma, using again the Lipschitz assumption on $\mathtt V$.
We still have to verify the $G$-growing conditions.

First, by differentiation with respect to $t$ with $t'$ fixed, the semi-group property $\mathtt G_{t+t'}=\mathtt G_t\circ \mathtt G_{t'}$ directly follows from the ODE $\partial_t \mathtt{G}_t(\mathbf{y}) = \mathtt{V}(\mathtt{G}_t(\mathbf{y}))$ with initial condition $\mathtt G_0=\mathrm{id}$.
Deriving the monotonicity property of $G_x$ (or of the reparametrized $\mathtt{G}_t$) functions from the inequalities $\mathtt{V}(\mathbf{y}) \leq  \mathbf y$ on $\mathcal E_0$ is mutatis mutandis the same proof as Proposition \ref{prop:monotonicityviageneratoriff} --- we just need to check that $\mathtt G_t$ sends $\mathcal E_0=\mathcal O\cap\mathrm{Supp}(\boldsymbol\Xi)$ to itself (since the inequalities are only assumed there), which can be done as follows:  (i) the above Euler scheme argument proves that $\mathtt G_t(\mathcal O)\subset \mathcal O$, and (ii) the quasi-preservation assumption that we verify below implies further $\mathtt G_t(\mathcal O\cap\mathrm{Supp}(\boldsymbol\Xi))\subset \mathrm{Supp}(\boldsymbol\Xi)$.
Then, verifying \eqref{ass:lipschitz} from the $L^2$-condition is parallel to the proof of Proposition \ref{prop:integratedLipschitzgenerator}.

Lastly, we still need to verify the quasi-preservation assumption, \textit{i.e.}~that the measures $\boldsymbol\Xi_t=(\mathtt G_t)_\sharp\boldsymbol \Xi$ for $t\in\R$, are actually given by $\boldsymbol\Xi_t=\mathrm e^{\alpha t}\,\boldsymbol\Xi$. First, for all cylindrical test functions $F$ on $\mathcal O$, we have:
\begin{align*}
   \frac{\d}{\d t} \int_{\mathcal O} F(\mathtt G_t(\mathbf y))\,\d\boldsymbol\Xi
    = \int_{\mathcal O} \langle\nabla F(\mathtt G_t),\mathtt V(\mathtt G_t)\rangle\,\d\boldsymbol\Xi,
\end{align*}
which is simply obtained by exchanging derivation and integration.
It follows from the ODE defining $\mathtt G$ and the $C^1$-smoothness of $\mathtt V$ that $\mathtt G_t$ is itself $C^1$-smooth for all $t$, with differential $\mathtt D \mathtt G_t$ satisfying the ODE $\partial_t \left(\mathtt D \mathtt G_t(\mathbf y)\right)=\mathtt D\mathtt V(\mathtt G_t(\mathbf y))\cdot \mathtt D \mathtt G_t(\mathbf y)$.
We now claim that $\mathtt D\mathtt G_t\cdot\mathtt V=\mathtt V\circ \mathtt G_t$, an identity which expresses that $\mathtt V$ is preserved when transported along its own flow.
Temporarily taking this identity for granted, the RHS of the last display becomes
\begin{align*}
    \int_{\mathcal O} \langle\nabla F(\mathtt G_t),\mathtt D\mathtt G_t\cdot\mathtt V\rangle\,\d\boldsymbol\Xi
        = \int_{\mathcal O} \langle(\mathtt D\mathtt G_t)^\top\nabla F(\mathtt G_t),\mathtt V\rangle\,\d\boldsymbol\Xi
        &= \int_{\mathcal O} \langle \nabla (F\circ \mathtt G_t),\mathtt V\rangle\,\d\boldsymbol\Xi\\
        &= \alpha \int_{\mathcal O} F\circ \mathtt  G_t\,\d\boldsymbol\Xi,
\end{align*}
where the last identity uses that $\mathtt V$ satisfies the divergence PDE.
The two preceding displays together entail that for all cylindrical test functions $F$, we have:
\begin{align*}
     \frac{\d}{\d t}\int_{\mathcal O}F\,\d\boldsymbol\Xi_t = \alpha \int_{\mathcal O}F\,\d\boldsymbol\Xi_t\qquad\text{for all $t\in\R$}.
\end{align*}
This gives that $\int_{\mathcal O}F\,\d\boldsymbol\Xi_t=\mathrm e^{\alpha t}\int_{\mathcal O}F\,\d\boldsymbol\Xi$ for all $t$ and thus the result.

It remains to justify our claim that  $\mathtt D\mathtt G_t\cdot\mathtt V=\mathtt V\circ \mathtt G_t$ for all $t$. For a fixed $\mathbf y\in\mathcal O$, it follows from the chain rule that both $t\mapsto \mathtt D\mathtt G_t(\mathbf y)\cdot\mathtt V(\mathbf y)$ and $t\mapsto\mathtt V(\mathtt G_t(\mathbf y))$ are solutions in the indeterminate $(\mathbf u_t)_t\in C^1(\R,\ell_q)$ of the ODE given by $\partial_t \mathbf u_t=\mathtt D\mathtt V(\mathtt G_t(\mathbf y))\cdot \mathbf u_t$, with initial condition $\mathbf u_0=\mathtt V(\mathbf y)$. Notice that $(t,\mathbf y)\mapsto\mathtt D\mathtt V(\mathtt G_t(\mathbf y))$ is continuous in both coordinates and locally Lipschitz in its second coordinate --- by composition of a bounded linear operator with a $C^1$-smooth function --- so that the Cauchy--Lipschitz theorem ensures that solutions of the preceding ODE with a fixed initial condition are unique. Hence $\mathtt D\mathtt G_t\cdot\mathtt V=\mathtt V\circ \mathtt G_t$, thus proving the claim and finishing our proof.
\end{proof} 

\begin{proposition}[From bifurcator to locally largest]  \label{prop:locallylargestthebest}
    Let $\boldsymbol\Xi^{\ell\ell}$ be a \textit{locally largest} splitting measure, and let $\boldsymbol\Xi^{\mathfrak p}$ be a bifurcator specified by a family $\mathfrak p=(p_i)_i$ as in \eqref{eq:defbifurcatorp}, but defined on an open subset $\mathcal O\subset \ell_q$. Let $\alpha>0$ and suppose that $(\boldsymbol\Xi^{\mathfrak p};\alpha)$ is nicely growing with generator $\mathtt V\colon\mathcal O\to\R^\N$. If we define the generator $\mathtt V^{\ell\ell}$ by
    \begin{align}\label{eq:def-symmetrized-V}
        \mathtt V^{\ell\ell}(\mathbf y)=\sum_{i\geq0} p_i(\mathbf y)\,\sigma_i^{-1}\mathtt V(\sigma_i\mathbf y),
        \qquad \mathbf y\in \mathcal O,
    \end{align}
    where $\sigma_i\colon (y_0,(y_1,\dots))\mapsto (y_i,(y_0,\dots,y_{i-1},y_{i+1},\dots))$ puts the $i$-th coordinate to the front, then $\mathtt V^{\ell\ell}$ satisfies the \textit{nicely growing} conditions for $(\boldsymbol\Xi^{\ell\ell};\alpha)$ except possibly the \textbf{regularity} condition.
    It thus suffices to check the latter condition to prove that $(\boldsymbol\Xi^{\ell\ell};\alpha)$ is nicely growing.
\end{proposition}

\begin{proof}
    Let us inspect the required conditions on $\mathtt V^{\ell\ell}$.
    \begin{itemize}
        \item \textbf{Inequalities.}
        For $\mathbf y\in \mathcal O\cap \mathrm{Supp}(\boldsymbol\Xi^{\ell\ell})$, we have $\mathtt V^{\ell\ell}(\mathbf y)=\sum_{i\geq0} p_i(\mathbf y)\,\sigma_i^{-1}\mathtt V(\sigma_i\mathbf y)\leq\sum_i  p_i(\mathbf y)\mathbf y=\mathbf y$, coordinate-wise, where we used that the inequalities satisfied by $\mathtt V$ include $\mathtt V(\sigma_i\mathbf y)\leq\sigma_i\mathbf y$.
        \item \textbf{Divergence PDE.} Since the sequences $\mathbf y$ in the support of $\boldsymbol\Xi^{\ell\ell}$ are ordered, we only need to chack the divergence PDE against cylindrical test functions $F\colon \mathcal O\to \R_+$ which are permutation invariant. For such an $F$, we have:
    \begin{align*}
        \int_{\mathcal{O}} \langle\mathtt V^{\ell\ell},\nabla F\rangle\, \d \boldsymbol{\Xi}^{\ell\ell}
            &=  \int_{\mathcal{O}}  \boldsymbol{\Xi}^{\ell\ell}(\d y) \sum_i p_i(\mathbf y)\langle\sigma_i^{-1}\mathtt V(\sigma_i\mathbf y),\nabla F(\mathbf y)\rangle\\
            &=  \int_{\mathcal{O}}  \boldsymbol{\Xi}^{\ell\ell}(\d y) \sum_i p_i(\mathbf y)\langle\mathtt V(\sigma_i\mathbf y),\nabla F(\sigma_i\mathbf y)\rangle\\
            &\underset{\eqref{eq:defbifurcatorp}}{=}  \int_{\mathcal{E}}  \boldsymbol{\Xi}(\d y) \langle\mathtt V(\mathbf y),\nabla F(\mathbf y)\rangle\\
             &= \alpha\int_{\mathcal{O}} F\, \d \boldsymbol{\Xi}\\
            &= \alpha\int_{\mathcal{O}} F\, \d \boldsymbol{\Xi}^{\ell\ell}.
    \end{align*}
    The second equality follows by re-indexation of the sum defining the scalar product, and the easily verified identity $\sigma_i \nabla F=\nabla(F\circ \sigma_i)$, while the last equality uses that $\boldsymbol\Xi$ is permutation-invariant.
    \item \textbf{$L^2$ condition.} By Jensen's inequality and the permutation-invariance of the $\ell_q$-norm, we have:
    \begin{align*}
        \int_{\mathcal O}\boldsymbol\Xi^{\ell\ell}(\d \mathbf y)\,|\mathtt V^{\ell\ell}|^2
            \leq \int_{\mathcal O}\boldsymbol\Xi^{\ell\ell}(\d \mathbf y)\,\sum_i p_i(\mathbf y)\,\left|\sigma_i^{-1}\mathtt V(\sigma_i\mathbf y)\right|^2
            &=  \int_{\mathcal O}\boldsymbol\Xi^{\ell\ell}(\d \mathbf y)\,\sum_i p_i(\mathbf y)\,\left|\mathtt V(\sigma_i\mathbf y)\right|^2\\
            &=  \int_{\mathcal O}\left|\mathtt V\right|^2\, \d \boldsymbol\Xi,
    \end{align*}
    which is finite since $\mathtt V$ is in $L^2(\boldsymbol\Xi)$ by assumption.
    \item \textbf{No escape from $\mathcal O$.}
    The \textit{no escape} condition for $\mathtt V^{\ell\ell}$ follows from the expression
    \begin{align}\label{eq:cond-to-get-regularity-V}
        \mathbf y+t\cdot\mathtt V^{\ell\ell}(\mathbf y)
            =\sum_i p_i(\mathbf y)\sigma_i^{-1}\left(\sigma_i\mathbf y+t\cdot\mathtt V(\sigma_i\mathbf y)\right),
            \qquad \mathbf y\in \mathcal{O},\quad t\in\R,
    \end{align}
    the identity $\sum_i p_i=1$, and the the no escape condition for $\mathtt V$.
\end{itemize}
This proves that the \textit{nicely growing} conditions on $\mathtt V^{\ell\ell}$ with respect to $(\boldsymbol\Xi^{\ell\ell},\alpha)$ hold, except maybe the \textbf{regularity} condition.
\end{proof}

Before turning back to examples, we conclude with a useful partial converse to Proposition~\ref{prop:nicely-implies-grow}, allowing to gain information on the generator from information on the $G$-family.

\begin{proposition}[Generator from $G$-family]\label{prop:gen-from-G}
    Let $(\boldsymbol\Xi,\alpha)$ be a $G$-growing pair for some family $(G_x:x > 0)$ of mappings $\mathcal E_0\to\mathcal E_0$. Assume that the $G_x$ extend from $\mathcal E_0$ to some $\ell_q$ open subset $\mathcal O\subset \ell_q$, and that for all $\mathbf y\in\mathcal O$, the curves $x\mapsto G_x(\mathbf y)$ are differentiable at $x=1$ with respect to the $\ell_q$-norm. If we set
    \begin{align*}
        \mathtt V(\mathbf y)= -\partial_x (G_x(\mathbf y))\bigl.\Bigr|_{x=1},\qquad \mathbf y\in\mathcal O,
    \end{align*}
    then $\mathtt V$ satisfies \textbf{Inequalities} and \textbf{Divergence PDE} among the \emph{nicely growing} conditions of Definition~\ref{def:nicelygrowing}. In particular, if the \emph{$L^2$}, \emph{regularity} and \emph{no escape} conditions are also satisfied, then $(\boldsymbol\Xi,\alpha)$ is nicely growing with generator $\mathtt V$, and $G$ is precisely the $G$-family constructed in Proposition~\ref{prop:nicely-implies-grow}.
\end{proposition}

\begin{proof}
    For $\mathbf y\in\mathcal O\cap \mathrm{Supp}(\boldsymbol{\Xi})$, the curve $x\mapsto x\cdot G_x(\mathbf y)$ is coordinate-wise non-decreasing by the $G$-growing assumption. Differentiating at $x=1$ yields the \textit{Inequalities} condition on $\mathtt V$.
    On the other hand, for all cylindrical test functions $F\colon\mathcal O\to\R_+$, the quasi-preservation condition on $G$ gives:
    \begin{align*}
        \int_{\mathcal O} F\circ G_x \,\d\boldsymbol\Xi = x^{-\alpha}\int_{\mathcal O} F\,\d\boldsymbol\Xi.
    \end{align*}
    Differentiating on both sides at $x=1$ gives:
    \begin{align*}
        \int_{\mathcal O} \langle\nabla F, -\mathtt V\rangle \,\d\boldsymbol\Xi = -\alpha\int_{\mathcal O} F\,\d\boldsymbol\Xi,
    \end{align*}
    which yields the \textit{Divergence PDE} condition for $\mathtt V$.
    Lastly, we need to justify that when the \textit{nicely growing} conditions hold, $G$ is precisely the $G$-family constructed in Proposition~\ref{prop:nicely-implies-grow}. In this case, $\mathtt V$ is Lipschitz and thus the ODE $\partial_x\mathbf y_x=-\frac 1 x\mathtt V(\mathbf y_x)$ has unique maximal solutions, so that we only need to verify that $G$ satisfies it. This follows by differentiation at $x=1$ of the semi-group property $ G_{x\cdot x'}= G_x\circ  G_{x'}$.
\end{proof}

\subsection{Examples, again}
We now present ``implicit" examples of growing ssMt which are found using the generator tools developed above.

\begin{example}[The $\frac{ \mathrm{d}s}{(s(1-s))^\gamma}$-family]\label{ex:s(1-s)fam} Consider for $\gamma \in (1,3)$ the following binary conservative splitting measures
\begin{equation} \label{eq:split_beta}
	\int_\mathcal{E} \boldsymbol\Xi_{\gamma\text{-bin}}^{\ell\ell}(\d \mathbf y) F(\d \mathbf{y}) \propto \int_{1/2}^1 \frac{\d s}{(s(1-s))^{\gamma}} F(s,(1-s),0,...).
\end{equation}
The constraints $\gamma > 1$ and $\gamma <3$ respectively come from the fact we restricted to infinite splitting measures which integrate $(1-s)^2$ near $1$.  The case $\gamma = \frac{3}{2}$  has already been encountered in Example~\ref{ex:brownian_crt_mass_locally_largest}. Since we are in the binary conservative case, we can apply Proposition \ref{prop:solution-conservative-binary-case}. The generator associated to $(\boldsymbol \Xi_{\gamma\text{-bin}}^{\ell\ell} ; \alpha)$ is given in terms of the incomplete Euler's $\beta$-function: with the same notation as in Proposition \ref{prop:solution-conservative-binary-case} we have $ \mathtt{v}_{\alpha, \gamma} = -\alpha \cdot \mathtt{v}_\gamma$ where
$$
\mathtt{v}_\gamma(s)= \big(s(1 -s)\big)^\gamma \left(\beta_{s}(1-\gamma,1-\gamma) - \beta_{1/2}(1 - \gamma, 1 - \gamma)\right), \quad \mbox{ with } \beta_z(a,b) = \int_0^z \d t \ t^{a-1} (1-t)^{b-1}.$$
This has the easier expressions when $\gamma$ takes half-integer values 
\begin{itemize}   
\item $\mathtt{v}_{1/2}(s) = \displaystyle \sqrt{s(1-s)} \cdot  \frac{\pi - 4 \arcsin(\sqrt{s})}{2}$,
\item $\mathtt{v}_{1}(s) = s(1 - s) \cdot \log((1-s)/s)$,
\item $\mathtt{v}_{3/2}(s) = s(1-s) \cdot 2(1-2s),$
\item $\mathtt{v}_{2}(s) = s(1 - s)  \cdot   \Big(1 - 2  s +  4  s (1 - s) \mathrm{ArcTanh}(1-2s)\Big),$
\item $\mathtt{v}_{5/2}(s) = s(1 - s) \cdot \frac{2}{3}\big(1 + 2  s  (3 + 4  s  (2s-3))\big).$
\end{itemize} 
Using symmetry with respect to $1/2$ the condition of Proposition \ref{prop:monotonicityviageneratoriff} for $(\boldsymbol\Xi_{\gamma\text{-bin}}^{\ell\ell}; \alpha)$ to be growing is $\alpha \mathtt{v}_\gamma(s)\leq s$.  For $\gamma \in (1, 3/2]$ the curve $\mathtt{v}_\gamma$ is convex on $[0,1/2]$ so that the former inequality can be checked using derivatives at zero only and yields  $$\alpha_c(\boldsymbol\Xi_{\gamma\text{-bin}}^{\ell\ell}) = \gamma-1, \quad \mbox{for} \gamma \in (1, \frac{3}{2}],$$ which matches the upper bound on $\alpha_c$ given by Lemma \ref{lem:alpha-vs-cumulant}.
However, when $\gamma>3/2$ the curve $\mathtt{v}_\gamma$ is not convex anymore and we have $\alpha_c(\boldsymbol\Xi_{\gamma\text{-bin}}^{\ell\ell}) < \gamma-1$, see Figure \ref{fig:curve_alpha} for numerical values and Examples \ref{ex:brownian_GFT} and \ref{ex:aidekon_da_silva} for the specific examples $\gamma = \frac{5}{2}$ and $\gamma=2$.
\begin{figure}
    \centering
    \includegraphics[width=0.9\linewidth]{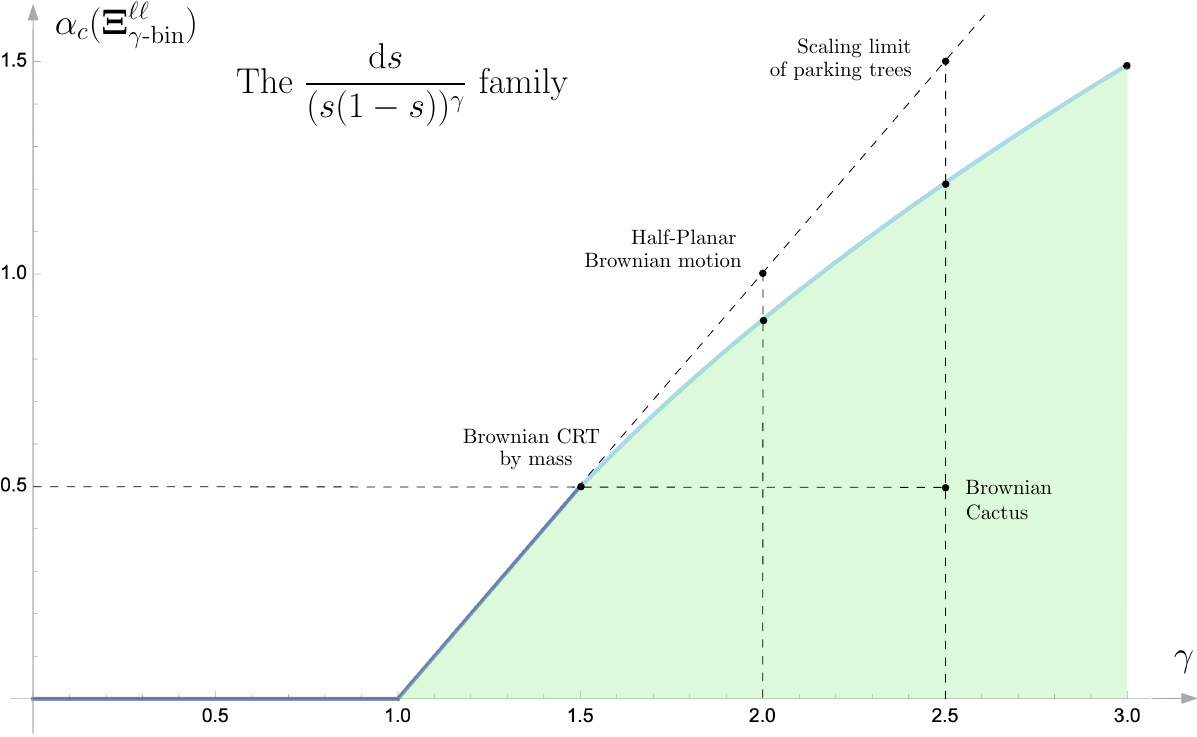}
    \caption{A plot of the critical values $\alpha_c(\boldsymbol\Xi_{\gamma\text{-bin}}^{\ell\ell})$ as a function of $\gamma$.}
    \label{fig:curve_alpha}
\end{figure}
\end{example}

\begin{example}[The Brownian growth-fragmentation tree] \label{ex:brownian_GFT} Let us single out an important case from the preceding example, which is the case $\gamma=5/2$ corresponding to the Brownian growth-fragmentation tree. This is formally the ssMt with characteristic quadruplet $( \mathrm{a}, \sigma^2 = 0, \boldsymbol{\Xi}_{\text{BroGF}}^{\ell\ell}; \alpha = 1/2)$ where the locally largest splitting measure $\boldsymbol{\Xi}_{\text{BroGF}}^{\ell\ell}$ is given by
\begin{equation*}
	\int_\mathcal{E} \boldsymbol{\Xi}_{\text{BroGF}}^{\ell\ell}(\d \mathbf{y}) F(y_0, y_1, \ldots) = \frac{3}{4\sqrt{\pi}} \int_{1/2}^1 \frac{\d s}{(s(1-s))^{5/2}} F(s, 1-s, 0,0,  \cdots).
\end{equation*}
Since $\boldsymbol{\Xi}_{\text{BroGF}}^{\ell\ell}$  does not integrate $1-s$ near $1$, compensation is needed and the value of $ \mathrm{a}$ needs to be carefully adjusted in the L\'evy-Khintchine formula (although its precise value is irrelevant regarding growing abilities). With the above normalization it is given by $
\mathrm{a}_\text{BroGF} = \frac{4(7 - 3\pi)}{3 \pi}$, 
see \cite[Chapter 3, Example 3.9]{bertoin2024self} for details. By the preceding example, if $(\boldsymbol{\Xi}_{\text{BroGF}}^{\ell\ell} ; \alpha)$ is $G$-growing, the (negative of the first coordinate of the) generator associated to $\mathtt{G}$ is 
\[
\alpha \cdot \mathtt{v}_{5/2}(s) = \frac{2 \alpha }{3} s(1 - s) (1 + 2 s (3 + 4 s (-3 + 2 s))).
\]
In particular, by Propositions \ref{prop:monotonicityviageneratoriff} and \ref{prop:integratedLipschitzgenerator}, the pair $(\boldsymbol{\Xi}_{\text{BroGF}}^{\ell\ell} ; \alpha)$ is growing iff $\alpha \cdot \mathtt{v}_{5/2}(s) \leq s$. Adjusting $\alpha$ so that $ \alpha \cdot \mathtt{V}_{5/2}(s)$ is tangent to the first bisector, we deduce that $\alpha_c \equiv \alpha_c(\boldsymbol{\Xi}_{\text{BroGF}}^{\ell\ell})$ is solution to the system 
$$ \alpha \cdot \mathtt{v}_{5/2}(s) = s \quad \mbox{and} \quad \alpha \cdot  \partial_s \mathtt{v}_{5/2}(s) =1.$$ After polynomial elimination of the variable $s$, we deduce that the critical value $\alpha_c$ is a root of the equation
$$ 256 -162 \alpha_c -42 \alpha_c^2 +\alpha_c^3= 0.$$
Numerically, this yields $\alpha_c \approx 1.211...$. In particular, the Brownian Growth-Fragmentation with self-similarity index $1/2$ as appearing in the Brownian disk (a.k.a.~the Brownian cactus) is growing. However, the ssMt with characteristics $( \mathrm{a}_{\text{BroGF}}, 0, \boldsymbol{\Xi}_{\text{BroGF}}^{\ell\ell}; \frac{3}{2})$, which e.g. appears as the scaling limits of peeling trees or fully parked trees \cite{contat2025universality} is \textbf{not} growing. 
\end{example}

\begin{example}[Aidekon Da-Silva]\label{ex:aidekon_da_silva} The case $\gamma=2$ of Example \ref{ex:s(1-s)fam} also deserves elaboration. The Aidekon Da-Silva ssMt \cite{AidekonDaSilva2022}, \cite[Chapter 3, Example 3.13]{bertoin2024self} has characteristics $(\mathrm{a} = 0, \sigma^2 = 0, \boldsymbol{\Xi}_{\text{ads}}^{\ell\ell} = \frac{2}{\pi} (\boldsymbol\Xi_{\text{ads}}^- + \boldsymbol\Xi^+_{\text{ads}}); \alpha = 1)$, where $\boldsymbol\Xi^+_{\text{ads}}$ is a non-conservative splitting measure given by 
\begin{equation} \label{eq:ads+}
	\int_\mathcal{E} \boldsymbol\Xi_{\text{ads}}^+(\d \mathbf y) F(y_0, y_1, \ldots)  = \int_{0}^\infty \frac{\d s}{(s(1+s))^2} F(1+s,s,0,...)
\end{equation}
and $\boldsymbol\Xi_{\text{ads}}^-$ is the conservative binary splitting measure 
\begin{equation} \label{eq:ads-}
	\int_\mathcal{E} \boldsymbol\Xi_{\text{ads}}^-(\d \mathbf y) F(y_0, y_1, \ldots)  = \int_{1/2}^1 \frac{\d s}{(s(1-s))^2} F(s,1-s,0,...).
\end{equation}
The unexpected fact is that a ssMt under $\mathbb{Q}_1$ is in fact a random scaling of an independent Brownian CRT of mass $1$, see \cite[Chapter 3, Example 3.13]{bertoin2024self}. Yet, $(\boldsymbol{\Xi}_{\text{ads}}^{\ell\ell} ; 1)$ is \textbf{not} growing! Indeed, using the explicit expression of the generator for the binary conservative part of the splitting measure, the critical value $ 
 \alpha_c \equiv \alpha_c(\boldsymbol\Xi_{\text{ads}}^-)$ must satisfy 
 $$ \alpha_c \cdot \mathtt{v}_2(s) = s \quad \mbox{and} \quad \alpha_c \cdot \partial_s \mathtt{v}_2(s) = 1,$$ which after a few manipulation shows that  $\alpha_c \approx 0.886...$ is solution to $$(\alpha_c-3)   \tanh\left(\frac{(\alpha_c-3) (9 + \alpha_c)}{16 \alpha_c}\right) = (\alpha_c + 1).$$
 In particular, although eventually coding for a very natural tree structure (for the self-similarity parameter $\alpha=1$), the Aidekon Da-Silva tree does \textbf{not} grow.
 \end{example}
 
 \begin{example}[Proof of Example \ref{ex:stable_crt_mass_locally_largest}]   \label{ex:stable_crt_mass_locally_largest_proof}  
 We have seen in Example~\ref{ex:stable_crt_mass_size_biased} that the size-biased bifurcator $\boldsymbol\Xi^*_{\beta-\rm st}$ of the ssMt encoding the mass fragmentation of a $\beta$-stable tree is such that the pair $(\boldsymbol\Xi^*_{\beta-\rm st},1-1/\beta)$ is $G^{\rm M}$-growing, where $G^{\rm M}_x$ is defined for $x>0$ in Proposition~\ref{prop:fct-mag-G-mass-height} as
 \begin{align*}
     G^{\rm M}_x(y_0, y_1, \ldots) = \left(\frac{xy_0}{xy_0+(1-y_0)},\frac{1}{xy_0+(1-y_0)}\cdot (y_1,y_2,\dots)\right),\qquad \mathbf y\in\mathrm{Supp}(\boldsymbol\Xi).
 \end{align*}
 We consider $\mathcal O\subset \ell_1$ the open subset defined as $$\mathcal O=\Bigl\{\mathbf y\in\ell_1\colon\sum_{j\geq0} y_j\in\left(\frac12,\frac32\right),\quad\forall i,y_i>0\Bigr\}.$$
 Then, $\mathbf \Xi(\mathcal E\setminus \mathcal O)=0$, and we can extend $G^{\rm M}_x$ for $x>0$ to a mapping $\mathcal O\to\mathcal O$ by setting
 \begin{align*}
      G^{\rm M}_x(y_0, y_1, \ldots) =
        \left(\frac{xy_0\Sigma_1}{xy_0+(\Sigma_1-y_0)},\frac{\Sigma_1}{xy_0+(\Sigma_1-y_0)}\cdot (y_1,y_2,\dots)\right),\qquad \mathbf y\in\mathcal O,
 \end{align*}
 where we use the notation $\Sigma_1=\sum_{j\geq0}y_j$.
 For all $\mathbf y\in\mathcal O$, we have
 $$ \frac{G^{\rm M}_x(\mathbf y)-\mathbf y}{x-1}=-\frac{\mathtt V^{\rm M}(\mathbf y)\Sigma_1}{(x-1)y_0+\Sigma_1},\qquad x\neq1,$$
 where we define
 \begin{align*}
     \mathtt V^{\rm M}(\mathbf y):=\left(-y_0\left(1-\frac{y_0}{\Sigma_1}\right),\frac{y_0}{\Sigma_1} \cdot(y_1,y_2,\dots)\right),\qquad \mathbf y\in\mathcal O.
 \end{align*}
 In particular, $\mathtt V^{\rm M}(\mathbf y)$ is minus the derivative at $1$ of $x\mapsto G_x(\mathbf y)$, with respect to the $\ell_1$ norm.
 By Proposition~\ref{prop:gen-from-G}, and the fact that the $G^{\rm M}$-growing condition for $(\boldsymbol\Xi^*_{\beta-\rm st},1-1/\beta)$ is satisfied by Example~\ref{ex:stable_crt_mass_size_biased}, we get that $\mathtt V^{\rm M}$ satisfies the \textit{Inequalities} and \textit{Divergence PDE} conditions in the Definition~\ref{def:nicelygrowing} of \textit{nicely growing} pairs. We will justify below that the $L^2$, \emph{regularity}, and \emph{no escape} conditions also hold. In particular, $(\boldsymbol\Xi^*_{\beta-\rm st},1-1/\beta)$ is nicely growing with the generator $\mathtt V^{\rm M}$.
 Now, since $\boldsymbol\Xi^*_{\beta-\rm st}$ is the size-biased bifurcator defined in \eqref{def:pstar} of the locally largest splitting measure $\boldsymbol\Xi^{\ell\ell}_{\beta-\rm st}$ defined in \eqref{eq:def-ll-stable-mass}, we get by Proposition~\ref{prop:locallylargestthebest} --- modulo the verification of the \textit{regularity} condition which will also be done below --- that $\boldsymbol\Xi^{\ell\ell}_{\beta-\rm st}$ is nicely growing with generator $\mathtt V^{\ell\ell}$ given by%
    \footnote{Notice that $\mathcal O$ is invariant under the action of each $\sigma_i$.}
 \begin{align*}
     \mathtt V^{\ell\ell}(\mathbf y)=\sum_i \frac{y_i}{\Sigma_1}\sigma_i^{-1}\mathtt V^{\rm M}(\sigma_i\mathbf y),
        \qquad \mathbf y\in \mathcal O.
 \end{align*}
 Therefore, Proposition~\ref{prop:nicely-implies-grow} implies that $\boldsymbol\Xi^{\ell\ell}_{\beta-\rm st}$ is $G^{\ell\ell}$-growing with $G^{\ell\ell}$ solution of the ODE
 \begin{align}\label{eq:ODE-G-stable-masse}
     \partial_x G^{\ell\ell}_x(\mathbf y) = -\frac{1}{x}\mathtt V^{\ell\ell}(G^{\ell\ell}_x(\mathbf y)),
     \qquad \mathbf y\in \mathcal O,\,x>0\qquad\text{and}
     \qquad G^{\ell\ell}_1=\mathrm{Id}.
 \end{align}
 From the explicit expressions above and the fact Since $\Sigma_1=1$ on $\mathrm{Supp}(\boldsymbol\Xi)$, we have that the $i$-the coordinate $\mathtt V^{\ell\ell}_i$ of the generator $ \mathtt V^{\ell\ell}$ is given for $\mathbf y\in\mathcal O\cap \mathrm{Supp}(\boldsymbol\Xi)$ by the expression:
 \begin{align*}
     \mathtt V^{\ell\ell}(\mathbf y)
        = \sum_i {y_i}\cdot \sigma_i^{-1}\mathtt V^{\rm M}(\sigma_i\mathbf y)
        &=\sum_i y_i\cdot
        \Bigl(y_i \cdot\bigl(y_0,\dots,y_{i-1}\bigr),
            -y_i(1-y_i),y_i \cdot\bigl(y_{i+1},y_2,\dots\bigr)
        \Bigr)\\
        &= \Bigl(-y_0(1-y_0)+y_0\textstyle\sum_{j\neq0}y_j^2,-y_1(1-y_1)+y_1\textstyle\sum_{j\neq1}y_j^2,\dots\Bigr).
 \end{align*}
 All in all, the coordinates $\mathtt V^{\ell\ell}_i$ of the generator $\mathtt V^{\ell\ell}$ are given for $\mathbf y\in\mathcal O\cap \mathrm{Supp}(\boldsymbol\Xi)$ by 
 \begin{align*}
      \mathtt V^{\ell\ell}_i(\mathbf y)
        =-y_i(1-y_i)+y_0\textstyle\sum_{j\neq i}y_j^2
        =-y_i\cdot(y_i-\textstyle\sum_j y_j^2).
 \end{align*}
 Putting it all together, the family $G^{\ell\ell}$ constructed above satisfies the ODE~\eqref{eq:ODE-G-stable-masse}, and for $\mathbf y\in\mathcal O\cap\mathrm{Supp}(\boldsymbol\Xi)$, by the above calculations the coordinates $\mathtt y_i$ of the curve $x\mapsto G^{\ell\ell}_x(\mathbf y)$ satisfy the differential system:
 \begin{align}\label{eq:expr-ll-gen-stable-mass}
    \begin{cases}
        \partial_x \mathtt{y}_i(x)=\frac{1}{x} \cdot \mathtt{y}_i(x)\cdot\left(\mathtt{y}_i(x)-\sum_{j}\mathtt{y}_j(x)^2\right),   & i\geq0,\\
        \mathtt{y}_i(1)= y_i.
    \end{cases}
\end{align}
This concludes the example, modulo the technical verifications that we have postponed and that we now undertake.
We need to justify that $\mathtt V^{\rm M}$ satisfies the $L^2$, \emph{regularity}, and \emph{no escape} conditions of Defininition~\ref{def:nicelygrowing} for the pair $(\boldsymbol\Xi^*_{\beta-\rm st},1-1/\beta)$, and that $\mathtt V^{\ell\ell}$ satisfies the \textit{regularity} condition of Defininition~\ref{def:nicelygrowing}.
\begin{itemize}
    \item \textbf{$L^2$-condition for $\mathtt V^{\rm M}$.} Since $\boldsymbol\Xi^*_{\beta-\rm st}(\d\mathbf y)$ integrates $(1-y_0)^2$, to get the $L^2$-condition of Definition~\ref{def:nicelygrowing}, it is sufficient to show that $|\mathtt V^{\rm M}(\mathbf y)|\leq C\cdot |1-y_0|$ for some $C>0$ and all $\mathbf y\in\mathcal O\cap\mathrm{Supp}(\boldsymbol\Xi)$. But for such $\mathbf y$ in the support, we have $\Sigma_1=1$ so that  $\mathtt V^{\rm M}(\mathbf y)=(-y_0(1-y_0),y_0 \cdot(y_1,y_2,\dots))$, and  also $\sum_{j\geq 1}y_j=1-y_0$, thus giving the $\ell_1$-norm $|\mathtt V^{\rm M}(\mathbf y)|= 2 y_0(1-y_0)$.
    \item \textbf{Regularity condition for $\mathtt V^{\rm M}$.} By a similar reasoning as above, if instead of $\mathbf y\in\mathcal O\cap\mathrm{Supp}(\boldsymbol\Xi)$ we take $\mathbf y\in\mathcal O$, then we have the inequality for $\ell_1$-norm $|\mathtt V^{\rm M}(\mathbf y)|\leq  2 \frac{y_0}{\Sigma_1}(1-\frac{y_0}{\Sigma_1})\leq 2$, so that $\mathtt V^{\rm M}$ is bounded in the $\ell_1$-norm on $\mathcal O$. It is also $C^1$ with respect to the $\ell_1$-norm, as a mapping $\mathbf y\mapsto (f_1(\mathbf y)y_0,f_2(\mathbf y)\cdot(y_1,y_2,\dots))$ with $f_1$ and $f_2$ two $C^1$ functions $\mathcal O\to\R$ with uniformly bounded partial derivatives --- we leave it to the reader to check that this is sufficient. Its differential at $\mathbf y \in\mathcal O$ is the linear operator $\mathtt D\mathtt V^{\rm M}(\mathbf y)\colon\ell_1\to\ell_1$, given for $\mathbf z\in\ell_1$ by:
    \begin{align*}
        \mathtt D \mathtt V^{\rm M}(\mathbf y)\cdot\mathbf z
        =(u_0,u_1,\dots),\qquad
        u_i
            =
                \begin{cases}
                    -z_0\left(1-\frac{y_0}{\Sigma_1}\right)+\frac{2 y_0z_0}{\Sigma_1}-\sum_{j\geq0}\frac{y_0^2 z_j}{\Sigma_1^2}, & i=0\\
                    \frac{y_0 z_i+z_0 y_i}{\Sigma_1}-\sum_{j\geq0}\frac{y_0 y_i z_j}{\Sigma_1^2}, & i\geq1.
                \end{cases}
    \end{align*}
    When $\mathbf y\in\mathcal O$, the  the $\ell_1$-norm of $\mathtt D \mathtt V^{\rm M}(\mathbf y)\cdot\mathbf z$ is thus seen to be bounded by a universal constant times the $\ell_1$-norm of $\mathbf z$. This bounds the operator norm of $\mathtt D \mathtt V^{\rm M}(\mathbf y)$ by a universal constant, so that $\mathtt V^{\rm M}$ is globally Lipschitz as needed.
    \item \textbf{No escape condition for $\mathtt V^{\rm M}$.}
    We claim that the no escape condition of Definition~\ref{def:nicelygrowing}, for $\mathtt V^{\rm M}$ with respect to $\mathcal O$, holds with the choice $\varepsilon=1$.
    First, the coordinates of $\mathtt V^{\rm M}$ sum to $0$, so that for all $\mathbf y \in\mathcal O$ and $t\in\R$, the sum of coordinates of $\mathbf y+t\mathtt V^{\rm M}(\mathbf y)$ is unchanged and equal to $\Sigma_1\in(\frac12,\frac32)$. It remains to be checked that for all $\mathbf y\in\mathcal O$ and all $|t|<1$, the coordinates of $\mathbf y+t\mathtt V^{\rm M}(\mathbf y)$ stay positive. Indeed, since $y_0/\Sigma_1\in (0,1)$, the coordinate with index $0$ is $y_0(1-t(1-\frac{y_0}{\Sigma_1}))\geq y_0(1-(1-\frac{y_0}{\Sigma_1}))=\frac{y_0^2}{\Sigma_1}>0$, and the coordinates with index $i\geq 1$ are given by $y_i(1+t\frac{y_0}{\Sigma_i})\geq y_i(1-\frac{y_0}{\Sigma_i})>0$.
    \item \textbf{Regularity condition for $\mathtt V^{\ell\ell}$.} 
    The boundedness and Lipschitz condition on  $\mathtt V^{\ell\ell}$ follow from that on $\mathtt V^{\rm M}$ and the expression \eqref{eq:expr-ll-gen-stable-mass}. Finally, $\mathtt V^{\rm \ell\ell}$ is also  $C^1$ with respect to the $\ell_1$-norm, as a mapping $\mathbf y\mapsto (f_0(\mathbf y)y_0,f_1(\mathbf y)y_1,\dots)$ with a family $(f_j,j\geq0)$ of $C^1$ functions $\mathcal O\to\R$ with uniformly bounded partial derivatives. We do not need to calculate its differential explicitely since the Lipschitz assumption is already verified.
\end{itemize}
The technical verifications for Example~\ref{ex:stable_crt_mass_locally_largest_proof} are now complete.
 \end{example}
 
 The last example is actually a counter-example.

 \begin{example}[A counter-example] \label{ex:haas_stephenson_mass_cex} There is generically no uniqueness of the growing family of functions when the support of the splitting measure is genuinely multi-dimensional. This follows from the fact that there are typically \textit{many} divergence-free vector fields in a domain in $\R^d$ with $d\geq2$, which can be used to give a distinct solution to the (linear!) divergence PDE in Definition~\ref{def:nicelygrowing}. For instance when $d\geq2$,  any compactly supported smooth function $\varphi\colon U\to \R$ in a domain $U\subset \R^2$ gives rise to a divergence-free vector field by taking the gradient rotated by a quarter-turn as follows: $$\nabla^\perp\varphi:=(-\partial_y \varphi, \partial_x\varphi).$$
 Without going into formal details, we exemplify in our context as follows.

For the locally largest bifurcator of the Haas--Stephenson $k$-ary fragmentation trees~\cite{haas2015scaling}, see Example~\ref{ex:haas_stephenson_mass}, one can solve for a generator $\mathtt V$ satisfying the divergence PDE from Definition~\ref{def:nicelygrowing}, and then verify the other \textit{nicely growing} conditions in order to construct a $G$-family using Proposition~\ref{prop:locallylargestthebest}. In the case $k=3$, the support of the splitting measure is a simplex which can be identified with the domain $\mathbb D $ of $\R^2$ defined by the equations $ 0<1-x-y<y<x$, upon setting $(x,y,1-x-y)=(y_0,y_1,y_2)$. Translating the divergence PDE from Definition~\ref{def:nicelygrowing}, we can look for a smooth vector field $\mathtt V\colon\mathbb D\to \R^2$ satisfying $$\mathrm{div}(f\mathtt V)=-\alpha f,$$ in the indeterminate $\mathtt V$, with $f$ an explicit smooth density. On the top left in Figure~\ref{fig:HS-non-unique-vector-fields}, we plot a numerical solution obtained by looking for a vector field $\mathtt V$ of the form $\mathtt V(x,y)=u(x,y)\cdot(1-x,-y)$ with $u$ a function; which we can solve for by the method of characteristics. Then, on the top right of Figure~\ref{fig:HS-non-unique-vector-fields}, we have defined the vector field $\mathtt W$ so that $$f\mathtt W=f\mathtt V+\nabla^\perp\varphi,$$ for a small enough smooth bump function $\varphi$. This ensures that $\mathtt W$ also satisfies the PDE $$\mathrm{div}(f\mathtt W)=-\alpha f.$$
 Now, among the \textit{nicely growing} conditions from Definition~\ref{def:nicelygrowing}, the more constraining one is certainly the \textit{Inequalities} that must be satisfied. It turns out that our choice of $\mathtt V$ satisfies the inequalities $\mathtt V(\mathbf y)\leq\mathbf y$ as is visible in the bottom left of Figure~\ref{fig:HS-non-unique-vector-fields} for the coordinate with index $0$. We also illustrate at the bottom right of Figure~\ref{fig:HS-non-unique-vector-fields} how, by choosing a small enough bump $\varphi$, the resulting vector field $\mathtt W$ that we have constructed satisfies the required inequalities $\mathtt W(\mathbf y)\leq\mathbf y$, here plotted for the coordinate with index $0$. Hence, we cannot expect uniqueness of the possible generators in the nicely growing condition of Definition~\ref{def:nicelygrowing}, and in turn\textbf{ we cannot expect uniqueness of the $G$-family} in genuinely multi-dimensional cases: in fact there is a continuum of solutions since in the above example, the space of admissible bumps $\varphi$ is infinite-dimensional as we only require the bump to be small enough and not located near the $\mathbf y$ saturating the inequalities for $\mathtt V$.
 \end{example}
 
 \begin{figure}
     \centering
     \includegraphics[width=.35\linewidth]{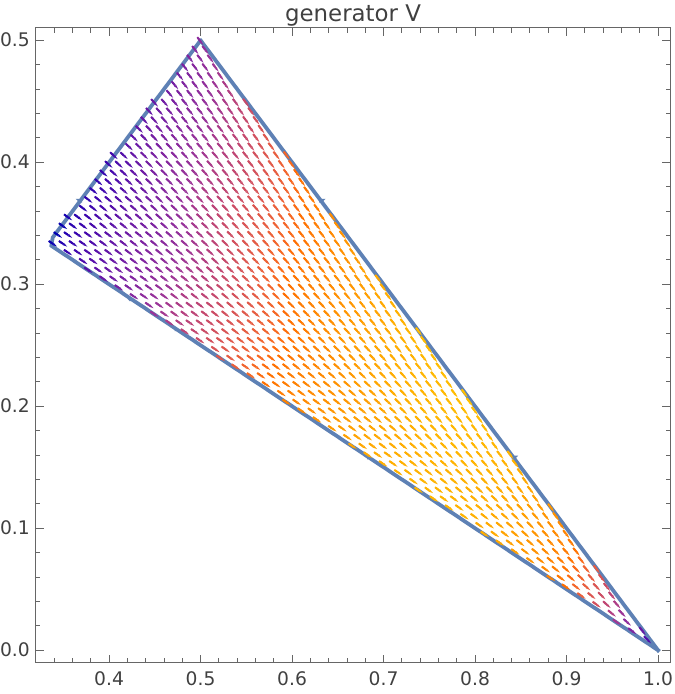}\hspace{2em}
     \includegraphics[width=.35\linewidth]{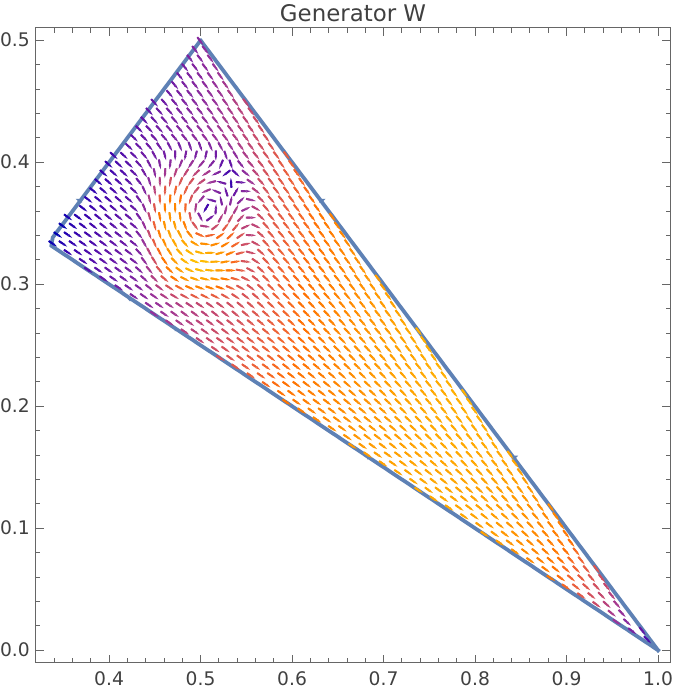}\vspace{2em}
     \includegraphics[width=.37\linewidth]{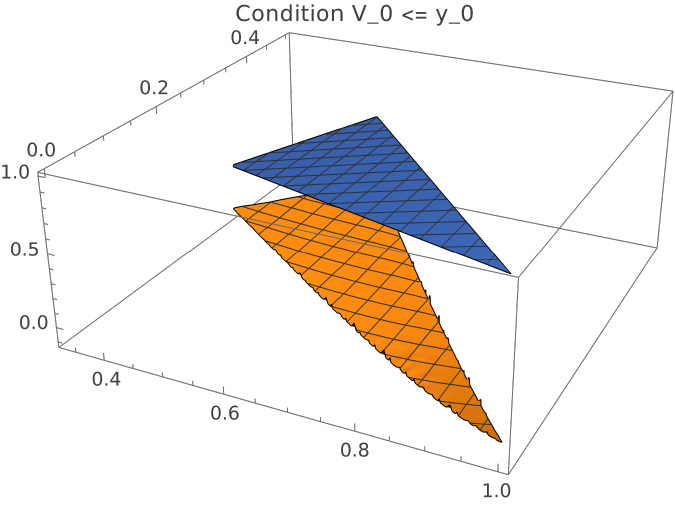}\hspace{1em}
     \includegraphics[width=.37\linewidth]{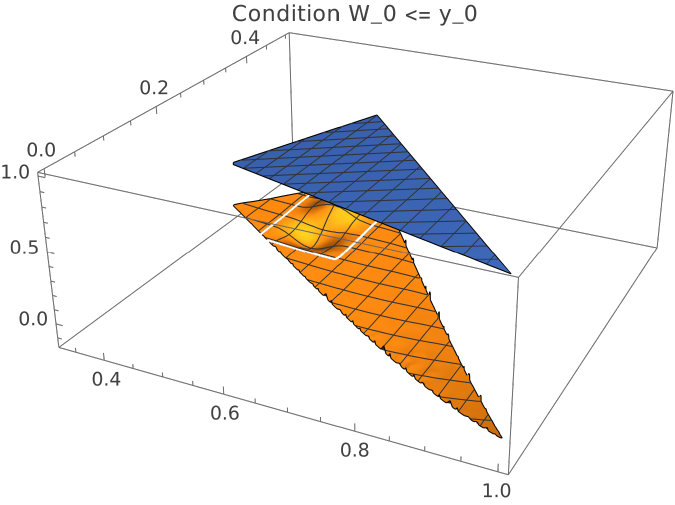}
     \caption{Illustration of Example~\ref{ex:haas_stephenson_mass_cex}. \textsc{Top:} Vector plots of the generators $\mathtt V$ on the left and $\mathtt W$ on the right. \textsc{Bottom:} Plots illustrating the coordinate-wise inequalities $\mathtt V(\mathbf y)\leq\mathbf y$ and $\mathtt W(\mathbf y)\leq\mathbf y$, here for the coordinates with index $0$. The blue surfaces represent as functions of $(y_0,y_1)$ the coordinate $y_0$, while the orange surfaces represent the coordinates $\mathtt V_0$ on the left and $\mathtt W_0$ on the right. In particular, at least on these plots, we see that the inequalities $\mathtt V_0(\mathbf y)\leq y_0$ and $\mathtt W_0(\mathbf y)\leq y_0$ are both satisfied.}
     \label{fig:HS-non-unique-vector-fields}
 \end{figure}

\section{Appendix}
In  this appendix, we collect various extensions of our results under less stringent hypotheses. 

\subsection{If \texorpdfstring{$\Xi_0$}{Ξ₀} does not integrate \texorpdfstring{$y_0^2$}{y₀²} at infinity}
\label{ss:general_case}  If the measure $\Xi_0$ only integrates $( y_0 - 1)^2$ around $1$, i.e. \eqref{eq:comp_all_jumps}) does not hold, then this is only a technical issue and the above results still hold. We recall $\mathcal{E}_0 = \{ (y_0, \mathbf{y}) \in \mathcal{E} : y_0 \in [e^{-1}, e] \}$ and we repeat that one has to alter the condition \eqref{ass:lipschitz} by replacing $\mathcal{E}$ with $\mathcal{E}_0$. This entails that the drift coefficient changes to a variable drift:
\begin{equation}\label{eq:variable_drift1}
    \beta(x) := \left( a + \sigma^2/2 + \int_{G_{x}^{-1} (\mathcal{E}_0)} (y_0 - 1 - \log y_0) \ \mathbf{\Xi}(\d \mathbf{y}) \right) \cdot x^{1 - \alpha},
\end{equation}
or, equivalently, by quasi-preservation \eqref{eq:quasipreservation},
\begin{equation}\label{eq:variable_drift2}
    \beta(x) := \left( a + \sigma^2 / 2 \right) x^{1-\alpha} + x \int_{\mathcal{E}_0} \left(  G_x^{(0)}(\mathbf{y}) - 1 - \log G_x^{(0)}(\mathbf{y}) \right) \ \mathbf{\Xi}(\d \mathbf{y}),
\end{equation}
so that the SDE still yields the law of the pssMp with characteristics $( \mathrm{a}, \sigma^2, \Lambda_0 ; \alpha)$.  We remark that for $x \in K$ a compact of $(0, \infty)$ and $(y_0, \mathbf{y}) \in \mathcal{E}_0$, 
\[
0 < c_K = \inf \{z_0 : (z_0, \mathbf{z}) \in G_x^{-1}(\mathcal{E}_0) \} \leq \sup\{z_0 : (z_0, \mathbf{z}) \in G_x^{-1}(\mathcal{E}_0) \} = C_K < + \infty,
\]
by the monotonicity property of $G$. One can bound the integral in \eqref{eq:variable_drift1} by
\[
\int_{G_{x}^{-1} (\mathcal{E}_0)} (y_0 - 1 - \log y_0) \ \mathbf{\Xi}(\d \mathbf{y}) \leq \int_{c_K}^{C_K} (y_0 - 1)^2 \ \Xi_0(\d y_0) < +\infty.
\]
This shows that the variable drift $\beta$ remains locally Lipschitz. The large jumps of $\mathcal{N}$, those corresponding to $y_0 \notin [e^{-1}, e]$, are then added to the SDE \eqref{eq:SDE} manually in a procedure called \textit{interlacing}. The rest of the proof goes on mutatis mutandis. \\

\subsection{In presence of killing}\label{sec:killing}
Throughout the paper we did not consider a possible killing term $\mathtt{k} = \boldsymbol{\Xi}(\mathbf{0})$, even though there exist various models of ssMt that have one (see for example \cite[Example~3.11]{bertoin2024self}). Here we discuss what changes need to be adopted in order to integrate a killing term in our framework and we shall see that this alters some of the results, notably the continuity of $x \mapsto \underline{\mathtt{T}}_x$. \\

Starting with pssMp, we first produce solutions $(\hat{X}^{(x)}, \hat{\eta}^{(x)})$ as in Section~\ref{sec:deco-repro} by considering the same Ikeda-Watanabe argument as in Proposition~\ref{prop:puss_until_tau}. The only difference is that now the processes $\hat{X}^{(x)}$ may not touch $0$ since their underlying L\'evy processes may drift to $\infty$. We then reintegrate the killing as follows: Consider $ \mathcal{E}$ an independent exponential variable of mean $ \frac{1}{\mathtt{k}}$ and define the decoration-reproduction process $(X^{(x)}, \eta^{(x)})$ by killing $(\hat{X}^{(x)}, \hat{\eta}^{(x)})$ at 
$$ z^{(x)} \quad \mbox{ such that } \quad \int_0^{z^{(x)}} \frac{ \mathrm{d}s}{( X_s^{(x)} )^\alpha} = \mathcal{E}.$$
It is easily seen using the inverse Lamperti transformation \eqref{eq:inverselamperti} that $z^{(x)}$ is almost surely finite, and that $({X}^{(x)}, \eta^{(x)})$ indeed has law $P_x$. Furthermore, since $x \mapsto X^{(x)}$ is increasing, it follows that $x \mapsto z^{(x)}$ is also increasing. It is not hard to see that the monotonicity and synchronicity of the jumps of Theorem~\ref{thm:deco-repro-grow} still hold. Furthermore, the presence of killing does not compromise the measurability property. There is almost surely no reproduction at the killing time, since $\mathcal{E}$ is independent of $\mathcal{N}$ and since $z^{(x)}$ is previsible conditionally on $\mathcal{E}$. By the same method, this yields the measurability property. On the contrary, the continuity of $x \mapsto (X^{(x)}, \eta^{(x)})$ is not guaranteed anymore.

\paragraph{Loss of continuity of $x \mapsto \mathtt{T}_x$ in presence of killing.}
The absorption time $x \mapsto z^{(x)}$ is almost surely continuous, this can be seen directly from its definition. Similarly, the resulting process $x \mapsto X^{(x)}$ remains a.s. continuous for the Skorokhod distance. One might then be tempted to think that the resulting process $x \mapsto \mathtt{T}_x$ remains continuous as well. However, the problem is that the reproduction process $x \mapsto \eta^{(x)}$ ceases to be continuous. With positive probability, there is a time $\tau$ such that $\mathcal{N}$ has an atom with $|y_1| \geq \eta$ for some $\eta > 0$ and let $\chi$ a label such that $z^{(\chi)} = \tau$. Since $X$ is almost surely positive at this time, i.e. $J := X^{(\chi)}_{z^{(\chi)}-} > 0$, the reproduction process $\eta^{(\chi)}$ has an atom at time $\tau$ with vertical coordinate $J \cdot G_J(y_1) > 0$. However, for any $\varepsilon > 0$, the process $\eta^{(\chi - \varepsilon)}$ does \textbf{not} have this atom. The big difference with the situation without killing is the fact that $x G_x(y_1) \to 0$ for $x \to 0$, which is the case when $X^{(x)}$ is continuously absorbed at $z^{(x)}$ instead of killed, i.e. $X^{(x)}_{z^{(x)}-} = 0$. On the resulting tree $\mathtt{T}_x$, one can see this as branches that slowly start growing in the case without killing versus large branches that suddenly appear in the case with killing. The resulting dynamic $x \mapsto \mathtt{T}_x$ is thus not continuous, but we stress the fact that the rest of Theorem~\ref{thm:mainprecise} still holds.

\printbibliography

\end{document}